\documentclass[12pt]{article}
\usepackage{float}
\usepackage[german,english]{babel}
\usepackage{epsf,array}
\usepackage{latexsym,bbm}
\usepackage{epsfig,appendix}
\usepackage[]{algorithm2e}
\usepackage {amssymb,amsthm,amsfonts}
\usepackage {amsmath,comment,nccmath}
\usepackage{color,enumitem}
\usepackage{setspace}
\usepackage{url,caption,subcaption}
\usepackage{hyperref,appendix}
\def\bm#1{\mbox{\boldmath{$#1$}}}

\usepackage{rotating}
\usepackage[longnamesfirst]{natbib}
\bibpunct{(}{)}{;}{´a}{,}{,}
\renewcommand{\cite}{\citep}

\DeclareMathOperator*{\argmin}{arg\,min}
\newtheorem{lemma}{Lemma}
\newtheorem{theorem}{Theorem}
\newtheorem{prop}{Proposition}

\makeatother
\makeatletter
\renewcommand{\@fnsymbol}[1]{\@arabic{#1}}
\makeatother
\title{Efficient nonparametric estimation of Toeplitz covariance matrices}

\author{Karolina Klockmann\footnote{Department of Statistics and Operations Research, Universit\"at Wien,
	Oskar-Morgenstern-Platz 1, 1090 Wien, Austria }
\and Tatyana Krivobokova$^{1}$}

\begin{document}
\maketitle

\begin{abstract}
	\baselineskip=15pt \noindent 
	\\
	A new efficient nonparametric estimator for Toeplitz covariance matrices is proposed. This estimator is based on a data transformation that translates the problem of Toeplitz covariance matrix estimation to the problem of mean estimation in an approximate Gaussian regression. The resulting Toeplitz covariance matrix estimator is positive definite by construction, fully data-driven and computationally very fast. 
	Moreover, this estimator is shown to be minimax optimal under the spectral norm for a large class of Toeplitz matrices. These results are readily extended to estimation of inverses of Toeplitz covariance matrices. 
	Also, an alternative version of the Whittle likelihood for the spectral density based on the discrete cosine transform is proposed. The method is implemented in the R package \texttt{vstdct} that accompanies the paper.
\end{abstract}
\baselineskip=20pt

\section{Introduction}
\label{sec:intro}
Estimation of covariance and precision matrices is a fundamental problem in statistical data analysis with countless applications in the natural and social sciences. A special type of covariance matrices that have each descending diagonal constant, known as Toeplitz matrices, arise in the study of stationary stochastic processes. Stationary stochastic processes are an important modeling tool in many applications, such as radar target detection, speech recognition,  modeling internet economic activity, electrical brain activity or the motion of crystal structures \citep[p.232]{du2020toeplitz, roberts2000hidden,quah2000internet,franaszczuk1985application,grenander1958toeplitz}.

The data for estimation are given as $n$ independent and identically distributed realizations of a $p$-dimensional vector having a zero mean and a Toeplitz covariance matrix $\Sigma=(\sigma_{|i-j|})_{i,j=1}^p$. Thereby, $p$ is assumed to grow while $n$ may be equal to $1$ or may tend to infinity as well. 
For $p\to \infty$ and $p/n \to c \in (0,\infty]$, the sample (auto-)covariance matrix is known to be an inconsistent estimator of $\Sigma$ in the spectral norm \citep[see e.g.,][chap. 2]{wu2009banding,pourahmadi2013high}.
Therefore, tapering, banding and thresholding of the sample covariance matrix have been proposed to regularize this estimator \citep[see][]{wu2012covariance,cai2013optimal}. The optimal rate of convergence for Toeplitz covariance matrix estimators was established in \citet{cai2013optimal}, who in particular showed that tapering and banding estimators attain the minimax optimal convergence rate over certain spaces of Toeplitz covariance matrices. Optimality of the thresholded estimator was shown only for the case $n=1$ \citep[see][]{wu2012covariance}.  However, all of these estimators have several practical drawbacks that affect their performance in small samples. 
First, additional manipulations with the estimators must be performed to enforce positive definiteness, see e.g., Section 5 in \citet{cai2013optimal}.  Second, the data-driven choice of the tapering, banding or thresholding parameter is not trivial in practice. For $n>1$, \citet{bickel2008regularized} proposed a cross-validation criterion that approximates the risk of the estimator. \citet{fang2016tuning} compared this method with a bootstrap based approximation of the risk in an intensive simulation study and recommended cross-validation over bootstrap. However, for small $n$ a cross-validated tuning parameter turns out to be very variable, while already for moderate $n$ it becomes numerically very demanding.  
For $n=1$, to the best of our knowledge, there is no fully data-driven approach for selecting the banding/tapering/thresholding parameter available. \citet{wu2009banding} suggested first to split the time series into non-overlapping subseries and then apply the cross-validation criterion of \citet{bickel2008regularized}.  However, it turns out that the appropriate choice of the subseries length is crucial for this approach, but cannot be done data-driven. 

In this work, an alternative way to estimate a Toeplitz covariance matrix and its inverse is proposed. Our approach exploits the one-to-one correspondence between Toeplitz covariance matrices and their spectral densities. First, the given data are transformed into approximate Gaussian random variables whose mean equals to the logarithm of the spectral density. Then, the log-spectral density is estimated by a periodic smoothing spline with a data-driven smoothing parameter. Finally, the resulting spectral density estimator is transformed into an estimator for $\Sigma$ or its inverse. It is shown that this procedure leads to an estimator that is fully data-driven, automatically positive definite and achieves the minimax optimal convergence rate under the spectral norm over a large class of Toeplitz covariance matrices. In particular, this class includes Toeplitz covariance matrices that correspond to long-memory processes with bounded spectral densities. Moreover, the computation is very efficient, does not require iterative or resampling schemes and allows to apply any inference and adaptive estimation procedures developed in the context of nonparametric Gaussian regression.  

Estimation of the spectral density from a single stationary time series is a research topic with a long history. Earlier nonparametric methods are based on smoothing the (log-)periodogram, which itself is not a consistent estimator \citep[][]{bartlett1950periodogram, welch1967use, thomson1982spectrum,wahba1980automatic}. Another line of nonparametric methods for estimating the spectral density is based on the Whittle likelihood, which is an approximation to the exact likelihood of the time series in the frequency domain. For example, \citet{pawitan1994nonparametric} estimated the spectral density from a penalized Whittle likelihood, while \citet{kooperberg1995rate} used polynomial splines to estimate the log-spectral density function maximizing the Whittle likelihood.
Recently, Bayesian methods for spectral density estimation have been proposed  \citep[see][]{choudhuri2004bayesian,edwards2019bayesian,maturana2021bayesian}, but these may become very computationally intensive in large samples due to posterior sampling.  

The minimax optimal convergence rate for nonparametric estimators of a H\"older continuous spectral density from a single Gaussian stationary time series was obtained by \citet{bentkus1985rate} under the $L_p$ norm, $1\leq p\leq \infty$. 
Only a few works on spectral density estimation show the optimality of the corresponding estimators. In particular, \citet{kooperberg1995rate}  and \citet{pawitan1994nonparametric} derived convergence rates of their estimators for the log-spectral density under the $L_2$ norm, while neglecting the Whittle likelihood approximation error. 

In general, most works on spectral density estimation 
do not exploit further the close connection to the corresponding Toeplitz covariance matrix estimation. In particular, an upper bound for the $L_\infty$ risk of a spectral density estimator automatically provides an upper bound for the risk of the corresponding Toeplitz covariance matrix estimator under the spectral norm. This fact is used to establish the minimax optimality of our nonparametric estimator for the Toeplitz covariance matrices. The main contribution of this work is to show that our proposed spectral density estimator is not only numerically very efficient, performing excellent in small samples, but also achieves the minimax optimal rate in the $L_\infty$ norm, which in turn ensures the minimax optimality of the corresponding Toeplitz covariance matrix estimator.

The paper is structured as follows. In Section~\ref{sec:likelihood}, the model is introduced and approximate diagonalization of Toeplitz covariance matrices with the discrete cosine transform is discussed. Moreover, an alternative version of the Whittle's likelihood is proposed. In Section~\ref{sec:method}, new estimators for the Toeplitz covariance matrix and the precision matrix are derived, while in Section~\ref{sec:theory} their theoretical properties are presented. In Section~\ref{sec:selection}, we  discuss how the parameters for our estimators are chosen. Section~\ref{sec:simulation} contains simulation results for Gaussian data and in Section~\ref{sec:nonGauss} the applicability of our method for non-Gaussian data is discussed. Section~\ref{sec:realdata} presents a real data example. The proofs are given in the appendix to the paper. Further simulation studies for non-Gaussian data are attached as supplementary material at the end.

\section{Set up and Diagonalization of Toeplitz Matrices}
\label{sec:likelihood}
Let $ Y_1,\ldots, Y_n\overset{\text{i.i.d.}}{\sim}{\mathcal{N}}_p( 0_p,\Sigma)$, where $\Sigma$ is a $p\times p$ positive semi-definite covariance matrix with a Toeplitz structure, that is, $\Sigma=(\sigma_{|i-j|})_{i,j=1}^p\succeq 0$. The sample size $n$ may tend to infinity or be a constant. The case $n=1$ corresponds to a single observation of a stationary time series and in this case the data are simply denoted by $Y\sim{\mathcal{N}}_p( 0_p,\Sigma)$.
The dimension $p$ is assumed to grow. 
The spectral density function $f$, corresponding to a Toeplitz covariance matrix $\Sigma$ with absolute summable sequence of covariances $(\sigma_k)_{k\in\mathbb{Z}}$, is given by
$$ f(x) = \sigma_0+2\sum_{k=1}^\infty\sigma_k\cos(kx), \quad x\in[-\pi,\pi].$$
The inverse Fourier transform implies
\begin{equation} \label{InvFourier}\sigma_k=\frac{1}{2\pi}\int_{-\pi}^\pi f(x)\cos(kx)\mbox{d}x=\int_0^1f(\pi x)\cos(k\pi x)\mbox{d}x.\end{equation}
If $\sum_{h=-\infty}^\infty |\sigma_h|=\infty$, i.e., the corresponding stochastic process is a long-memory process, the spectral density function is directly defined as a $2\pi$-periodic, non-negative function $f: [-\pi,\pi ] \to \mathbb{R}_{\geq0}$  that satisfies condition (\ref{InvFourier}). 
Hence, in case $f$ exists, $\Sigma$ is completely characterized by $f$. Furthermore, the non-negativity of the spectral density function implies the positive semi-definiteness of the covariance matrix. Moreover, the decay of the covariances $\sigma_k$ is directly connected to the smoothness of $f$. Finally, the convergence rate of a Toeplitz covariance estimator and that of the corresponding spectral density estimator are directly related via $\|\Sigma\| \leq \|f\|_\infty = \sup_{x\in[-\pi,\pi]}|f(x)|$, where $\|\cdot \|$ denotes the spectral norm \citep[see][chap.~5.2]{grenander1958toeplitz}.    

As in \citet{cai2013optimal}, we consider the class of positive semi-definite Toeplitz covariance matrices with H\"older continuous spectral densities. For $\beta=\gamma +\alpha>0$, where $\gamma\in\mathbb{N}\cup\{0\}$, $0<\alpha\leq 1$, and $0<M_0,M_1<\infty$, let
\begin{align*}
\mathcal{P}_\beta(M_0,M_1)= &\left \{ f \mid f: [-\pi,\pi]\to \mathbb{R}_{\geq 0}, \, \|f\|_{\infty}\leqslant M_0, \,  \|f^{(\gamma)}(\cdot+h)-f^{(\gamma)}(\cdot)\|_{\infty}\leqslant M_1|h|^{\alpha}\right \}.
\end{align*}
Furthermore, for $f\in\mathcal{P}_\beta(M_0,M_1)$  we denote by $\Sigma(f) \in \mathbb{R}^{p \times p} $ the corresponding $p\times p$ Toeplitz covariance matrix  obtained with the inverse Fourier transform (\ref{InvFourier}) and by $\Sigma^{-1}(f)$ the precision matrix.
The optimal convergence rate for estimating Toeplitz covariance matrices $\Sigma(f)$ where  $f\in {\mathcal{P}}_\beta(M_0,M_1)$ depends crucially on $\beta$. It is well known that the $k$th Fourier coefficient of a function whose $\gamma$th derivative is H\"older-continuous with exponent $\alpha\in(0,1]$ decays at least with order $\mathcal{O}(k^{-\beta})$ \citep[see][]{zygmund2002trigonometric}. Hence, $\beta$ determines the decay rate of the covariances $\sigma_k$, which are the Fourier coefficients of the spectral density $f$, as $k\rightarrow\infty$.  For $\beta\in(0,1/2]$, the class $\mathcal{P}_\beta(M_0,M_1)$ includes bounded spectral densities of certain long-memory processes. 

A connection between Toeplitz covariance matrices and their spectral densities is further exploited in the following lemma. 
\begin{lemma} \label{lemma:DCTdiag} 
	Let $\Sigma=\Sigma(f)$ with $f\in \mathcal{P}_\beta(M_0,M_1)$ and $x_j=(j-1)/(p-1)$ for $j=1,...,p$. Then 
	$$\left (D^T\Sigma D\right )_{i,j} = f(\pi x_j)\delta_{i,j} +\frac{1+(-1)^{|i-j|}}{2}\mathcal{O}\left(p^{-1}+ p^{-\beta}\log p\right ),$$
	where $\delta_{i,j}$ is the Kroneker delta, $\mathcal{O}(\cdot)$ terms are uniform over $i,j=1,\dots,p$ and  
	\begin{align*} 
	D&=\left(\frac{2}{p-1}\right)^{1/2}   \left [\cos \left \{\pi(i-1)\frac{j-1}{p-1} \right\} \right ]_{i,j=1}^p \text{ divided by } 2^{1/2} \text{ when } i \text{ or } j  \text { is } 1 \text{ or } p
	\end{align*}
	is the discrete cosine transform I matrix. 
\end{lemma} 

The proof can be found in Appendix ~\ref{app:lemma1}.
This result shows that the discrete cosine transform I matrix approximately diagonalizes Toeplitz covariance matrices and that the diagonalization error depends to some extent on the smoothness of the corresponding spectral density. 

In time series analysis the discrete Fourier transform matrix 
$F{=}p^{-1/2}\{\exp\left(2\pi \texttt{i} ij/p)\right\}_{i,j=1}^p,$
where $\texttt{i}$ is the imaginary unit, is typically employed to approximately diagonalize Toeplitz covariance matrices. Using the fact that $(F^T\Sigma F)_{i,i}=f(2\pi i/p)+o(1)$, \citet{whittle1957curve} introduced an approximation for the likelihood of a single Gaussian stationary time series (case $n=1$), the so-called Whittle likelihood
\begin{equation} \label{Whittle} 
\mathcal{L}( Y|f) 	\propto \exp \left \{ -\sum_{j=1}^{\left \lfloor p/2 \right \rfloor} \log f( 2\pi j/p) + \frac{I_j}{f(2\pi j/p )} \right \}.
\end{equation} 
The quantity  $I_j=|F_j^TY|^2$, where $F_j$ denotes the $j$th column of $F$, is known as the periodogram at the $j$th Fourier frequency.  Due to periodogram symmetry, only $\lfloor p/2\rfloor$ data points $I_1,...,I_{\lfloor p/2\rfloor}$ are available for estimating the mean $f(2\pi j/p)$ $(j=1,\ldots,\lfloor p/2\rfloor)$, where $\lfloor x\rfloor$ denotes the largest integer strictly smaller than $x$.
The Whittle likelihood has become a popular tool for parameter estimation of stationary time series, e.g., for nonparametric and parametric spectral density estimation or for estimation of the Hurst exponent, see e.g., \citet{walker1964asymptotic,taqqu1997robustness}.  

Lemma~\ref{lemma:DCTdiag} yields the following alternative version of the Whittle likelihood 
\begin{equation}\label{ourWhittle}
\mathcal{L}( Y|f) \propto \exp \left\{ -\sum_{j=1}^p \log f(\pi x_j) + \frac{W_j}{f(\pi x_j)} \right \}, 
\end{equation} 
where $W_j=({D}_j^TY)^2$, with $D_j$ denoting the $j$th column of $D$. Note that this likelihood approximation is based on twice as many data points $W_j$ as the standard Whittle likelihood. Thus, it allows for a more efficient use of the data $ Y$ to estimate the parameter of interest, such as the spectral density or the Hurst parameter. This is particularly advantageous in small samples. 

Equations (\ref{Whittle}) or (\ref{ourWhittle}) invite for the estimation of $f$ by maximizing the (penalized) likelihood over certain linear spaces, e.g., spline spaces, as suggested e.g., in \citet{kooperberg1995rate} or \citet{pawitan1994nonparametric}. However, such an approach requires well-designed numerical methods to solve the corresponding optimization problem, since the spectral density in the second term of (\ref{Whittle}) or (\ref{ourWhittle}) is in the denominator, which hinders derivation of a closed-form expression for the estimator and often leads to numerical instabilities. Also, the choice of the smoothing parameter becomes challenging. 

Therefore, we suggest an alternative approach that allows the spectral density to be estimated as a mean in an approximate Gaussian regression. Such estimators have a closed-form expression, do not require an iterative optimization algorithm and a smoothing parameter can be easily obtained with any conventional criterion. First note that if $Y\sim \mathcal{N}_p(0_p,\Sigma)$, with $\Sigma=\Sigma(f)$ and $f\in\mathcal{P}_\beta(M_0,M_1)$, then $D^TY\sim{\cal{N}}_p( 0_p,D^T \Sigma D)$. Hence, for $W_j=(D_j^TY)^2$ $(j=1,\ldots,p)$ it follows with Lemma~\ref{lemma:DCTdiag} that
\begin{equation} \label{ourGamma}
W_j \sim 
\Gamma \left\{1/2, 2f(\pi x_j)+\mathcal{O}\left(p^{-1}+p^{-\beta}\log p \right)\right \},  
\end{equation} 
where $\Gamma(a,b)$ denotes the gamma distribution with shape parameter $a$ and scale parameter $b$. The random variables $W_1,\ldots,W_p$ are only asymptotically independent. Obviously, $E(W_j)=f(\pi x_j)+o(1)$ for $j=1,\ldots,p$. To estimate $f$ from $W_1,\ldots,W_p$, one could use a generalized nonparametric regression framework with a gamma distributed response, see e.g., the classical monograph by \citet{hastie1990}. However, this approach  requires an iterative procedure for estimation, e.g., a Newton-Raphson algorithm, with a suitable choice for the smoothing parameter at each iteration step.  Deriving the $L_\infty$ rate for the resulting estimator is also not a trivial task. Instead, we suggest to employ a variance-stabilizing transform of \citet{cai2010nonparametricfest} that converts a gamma regression into an approximate Gaussian regression. 
In the next section we present the methodology in more detail for a general setting with $n\geq 1$.

\section{Methodology}
\label{sec:method}
Let $L_\delta=\{f:\inf_x f(x)\geq \delta\}$ for some $\delta>0$ and set ${\cal{F}}_\beta={\cal{P}}_\beta(M_0,M_1)\cap L_\delta$.
We consider estimation of $ \Sigma$ and $ \Omega= \Sigma^{-1}$ from a sample $ Y_1,..., Y_n\overset{\text{i.i.d.}}{\sim}\mathcal{N}_p( 0_p, \Sigma)$ where $\Sigma=\Sigma(f)$ with  $f\in{\cal{F}}_\beta$.
For $Y_i\sim  \mathcal{N}_p( 0_p,\Sigma)$ $(i=1,\ldots,n)$, it was shown in the previous section that with Lemma~\ref{lemma:DCTdiag} the data can be transformed into gamma distributed random variables $W_{i,j}=(D_j^TY_i)^2$ $(i=1,\ldots,n;\,j=1,\ldots,p)$, where for each fixed $i$ the random variable $W_{i,j}$ has the same distribution as $W_j$ given in (\ref{ourGamma}). Now the approach of \citet{cai2010nonparametricfest} is adapted to the setting $n\geq 1$. 

First, the transformed data points $W_{i,j}$ are binned, that is, fewer new variables $Q_k$ $(k=1,\ldots, T)$, with $T<p$, are built via $Q_k=\sum_{j=(k-1)p/T+1}^{kp/T}\sum_{i=1}^nW_{i,j}$ for $k=1,\ldots,T$. Note that the number of observations in a bin is $m=np/T$. In Theorem~\ref{theorem1} in Section \ref{sec:theory}, we show that setting $T=\lfloor p^\upsilon\rfloor$ for any $\upsilon\in(1-\min\{1,\beta\}/3,1)$ leads to the minimax optimal rate for the spectral density estimator. To simplify the notation, $m$ is handled as an integer (otherwise, one can discard several observations in the last bin).
Next, applying the variance-stabilizing transform $G(x)=2^{-1/2}\log\left(x/m\right)$ to each $Q_k$ yields new random variables $Y_k^*=2^{-1/2}\log \left (Q_k/m \right ) $ that are approximately Gaussian, as shown in \citet{cai2010nonparametricfest}. Since the spectral density is a function that is symmetric around zero and periodic on $[-\pi,\pi]$, one can mirror the resulting observations to use $ Y_{T}^*,\ldots,Y_2^*,Y_1^*,\ldots,Y_{T-1}^*$ for estimation. Renumerating the observations $Y_k^*$ and scaling the design points into the interval $[0,1)$ for convenience leads to an approximate Gaussian regression problem   
\begin{equation*}
Y_k^*\overset{\text{approx.}}{\sim} \mathcal{N} \left [ H\left\{f(x_k)\right\},m^{-1}\right ], \quad 
x_k=\frac{k-1}{2T-2} \quad (k=1,...,2T-2),
\end{equation*}
where $H(y)= 2^{-1/2}\left \{ \phi(m/2) + \log \left ( 2y/m\right ) \right\}$ and $\phi$ is the digamma function, see \citet[][]{cai2010nonparametricfest}. Now, the scaled and shifted log-spectral density $H(f)$ can be estimated with a periodic smoothing spline 
\begin{equation}
\label{eq:Spline}
\widehat{H(f)}(x)= \argmin_{s \in S_{\text{per}}(2q-1)}  \left [ \frac{1}{2T-2} \sum_{k=1}^{2T-2} \{ Y_k^* -s(x_k)\}^2 +h^{2q} \int_0^1 \{ s^{(q)}(x) \}^2 \, \mbox{d}x \right],
\end{equation}
where $h>0$ denotes a smoothing parameter, $q\in\mathbb{N}$ is the penalty order and $S_{\text{per}}(2q-1)$ is a space of periodic splines of degree $2q-1$. More details on periodic smoothing splines can be found in Appendix~\ref{app:SSper}.

Once an estimator $\widehat{H(f)}$ is obtained, application of the inverse transform function $H^{-1}(y)= m\exp  \{2^{1/2}y-\phi\left (m/2\right) \}/2$ yields the spectral density estimator $\hat{f}=H^{-1} \{\widehat{H(f)} \}$. Finally, the inverse Fourier transform leads to the following covariance matrix estimator
\begin{equation}
\label{eq:Sigma}
\widehat{\Sigma} = (\hat{\sigma}_{|i-j|})_{i,j=1}^p, \text{ with }\hat{\sigma}_k= \int_0^1 \hat{f}(x) \cos \left(k\pi x\right) \mbox{d}x\text{ for } k=0,....,p-1.
\end{equation}
The precision matrix $\Omega$ is  estimated by the inverse Fourier transform of the reciprocal of the spectral density estimator, i.e., 
\begin{equation}
\label{eq:Omega}
\widehat{\Omega} = (\hat{\omega}_{|i-j|})_{i,j=1}^p, \text{ with } \hat{\omega}_k = \int_0^1 \hat{f}(x)^{-1} \cos \left(k\pi x\right )\mbox{d}x \text{ for } k=0,....,p-1.
\end{equation}
The estimation procedure for $\widehat{\Sigma}$ and $\widehat{\Omega}$ can be summarized as follows.
\begin{enumerate}
	\item \textbf{Data transformation}: Define $W_{i,j}=(D_j^TY_i)^2$ $(i=1,\dots,n;\,j=1,\ldots,p)$, where $D$ is the $p\times p$ discrete cosine transform I matrix as given in Lemma~\ref{lemma:DCTdiag} and $D_j$ is its $j$th column.
	\item \textbf{Binning}: Set $T=\lfloor p^\upsilon\rfloor$ for any $\upsilon\in(1-\min\{1,\beta\}/3,1)$ and calculate
	$$Q_k=\sum_{j=(k-1)p/T+1}^{kp/T}\sum_{i=1}^n W_{i,j}, \quad (k=1,\ldots,T). $$
	\item \textbf{Variance-stabilizing transform}: Set $Y_k^*=2^{-1/2}\log \left (Q_k/m \right)$ for $k=1,...,T$ and $m =np/T$. Mirror the data to get $2T-2$ approximately Gaussian random variables $ Y_{T}^*,\ldots, Y_2^*,Y_1^*,\ldots, Y_{T-1}^*$. 
	\item \textbf{Gaussian regression}: Renumerate observations $Y^*_k$, scale the design points to $[0,1)$,  and estimate $H(f)$ with a periodic smoothing spline of degree $2q-1$ in an approximate Gaussian regression model 
	$$Y_k^*=H\{f(x_k)\}+\epsilon_k,\;\;x_k=\frac{k-1}{2T-2} \quad (k=1,\ldots,2T-2),$$ 
	where $\epsilon_k$ are asymptotically i.i.d. Gaussian variables.   
	\item \textbf{Inverse transform}:  Estimate the spectral density $f$ with $\hat{f}=H^{-1} \{\widehat{H(f)}  \},$ where \newline $H^{-1}(y)= m\exp \left \{2^{1/2}y-\phi\left (m/2\right) \right\}/2$, for a digamma function $\phi$. 
	\item \textbf{Estimators}:  Set $\widehat{\Sigma} = (\hat{\sigma}_{|i-j|})_{i,j=1}^p$ with $\hat{\sigma}_k= \int_0^1 \hat{f}(x) \cos\left (k\pi x\right) \mbox{d}x$ and   $\widehat{\Omega} = (\hat{\omega}_{|i-j|})_{i,j=1}^p$ with $\hat{\omega}_k = \int_0^1 \hat{f}(x)^{-1}\cos\left( k\pi x\right)\mbox{d}x$ $(k=0,...,p-1)$. 
\end{enumerate}

The estimators $\widehat{\Sigma}$ and $\widehat{\Omega}$ are positive definite matrices by construction, since the spectral density estimator $\hat{f}$ is non-negative by definition. For a detailed discussion on the choice of all parameters needed to obtain our estimators see Section \ref{sec:selection}.

\section{Theoretical Properties}
\label{sec:theory}
In this section, we study the asymptotic properties of the estimators $\hat{f}$, $\widehat{\Sigma}$ and $\widehat{\Omega}$. Let $\hat{f}=m\exp\{2^{1/2}\widehat{H(f)}-\phi(m/2)\}/2$ be the spectral density estimator defined in Section \ref{sec:method},
where $\widehat{H(f)}$ is given in (\ref{eq:Spline}), $m=np/T$ and $\phi$ is the digamma function. Furthermore, let $\widehat{\Sigma}$ be the Toeplitz covariance matrix estimator and $\widehat{\Omega}$ the corresponding precision matrix defined in equations (\ref{eq:Sigma}) and (\ref{eq:Omega}), respectively. The following theorem shows that both $\widehat{\Sigma}$ and $\widehat{\Omega}$ attain the minimax optimal rate of convergence over the class of Toeplitz matrices $\Sigma(f)$ such that $f\in \mathcal{F}_{\beta}$, $\beta>0$.
\begin{theorem} \label{theorem1} 
	Let $ Y_1,...., Y_n\overset{\text{i.i.d.}}{\sim} \mathcal{N}_p(0_p, \Sigma),\, n\geq 1,$ with $\Sigma=\Sigma(f)$ such that $f \in \mathcal{F}_{\beta}$ and $\beta=\gamma+\alpha>0$. If $h>0$ such that  $h\to 0$ and $hT\to \infty$, then with $T=\lfloor p^{\upsilon}\rfloor$ for any $\upsilon\in(1-\min\{1,\beta\}/3,1)$ and $q=\max\{1,\gamma\}$, the spectral density estimator $\hat{f}$, the corresponding covariance matrix estimator $\widehat{\Sigma}$ and the precision matrix estimator $\widehat{\Omega}$ satisfy for $p\to \infty$ and $n$ such that $p^{\min\{1,\beta\}}/n\to c\in(0,\infty]$  	
	\begin{eqnarray*}
		\sup _{f\in \mathcal{F}_{\beta}}\,	E_f \|\widehat{\Sigma}- \Sigma(f)\|^2 \leq \sup _{f\in \mathcal{F}_{\beta}}\, E_f \|\hat{f} - f\|_\infty^2 &=&\mathcal{O} \left\{ \frac{\log(np)}{nph}\right\} + \mathcal{O}(h^{2\beta})\nonumber\\
		\sup _{f\in\mathcal{F}_{\beta}}\,	E_f \|\widehat{\Omega} - \Sigma^{-1}(f)\|^2 &=&\mathcal{O} \left\{ \frac{\log(np)}{nph}\right\}  + \mathcal{O}(h^{2\beta}).
	\end{eqnarray*}
	For $h\asymp\left\{\log(np)/(np)\right \}^{\frac{1}{2\beta+1}}$ it follows that 
	\begin{eqnarray*}
		\sup _{f\in\mathcal{F}_{\beta}}\,	E_f \|\widehat{\Sigma}- \Sigma(f)\|^2  \leq  \sup _{f\in\mathcal{F}_{\beta}}\,E_f \|\hat{f} - f\|_\infty^2 &=& \mathcal{O}\left [ \left \{\frac{\log(np)}{np}\right \} ^{\frac{2\beta}{2\beta+1}}\right ]\nonumber \\
		\sup _{f\in \mathcal{F}_{\beta}}\,	E_f \|\widehat{\Omega} - \Sigma^{-1}(f)\|^2  &=&\mathcal{O}\left [ \left \{\frac{\log(np)}{np}\right \} ^{\frac{2\beta}{2\beta+1}}\right ].
	\end{eqnarray*}
\end{theorem}

Theorem~\ref{theorem1} is established under the asymptotic scenario with $p\rightarrow\infty$ and $n$ such that $p^{\min\{1,\beta\}}/n\to c\in(0,\infty]$, i.e., the dimension $p$ grows, while the sample size $n$ either remains fixed or also grows but not faster than $p^{\min\{1,\beta\}}$. Note that this asymptotic scenario covers the setting when the sample covariance matrix is inconsistent. In particular, for $\beta\geq 1$ the sample size $n$ does not grow faster than $p$ and adding more samples improves the convergence rate. If $\beta\in(0,1)$, then increasing $n$ with the rate faster than $p^\beta$ will not lead to a faster convergence rate, due to the diagonalization error from Lemma~\ref{lemma:DCTdiag}, which can be improved only by making additional assumptions on the spectral density. 

The minimax optimal convergence rates for estimating $\Sigma$ and $\Sigma^{-1}$ from $n$ i.i.d. Gaussian vectors $Y_1,...,Y_n$ with zero mean and a Toeplitz covariance matrix $\Sigma(f)$ with $f \in  \mathcal{F}_{\beta}$ have been established by \cite{cai2013optimal}. Since the lower bound rates given in Theorems 5 and 7 in \cite{cai2013optimal} match the upper bound rates obtained in our Theorem~\ref{theorem1}, we conclude that our estimator is minimax optimal. For non-Gaussian data the minimax optimal convergence rates for $\Sigma$ and $\Sigma^{-1}$ are not known. Note that ${\cal{F}}_\beta$ with $\beta>0$ includes bounded spectral densities of long-memory processes.

The proof of Theorem~\ref{theorem1} can be found in Appendix \ref{app:theorem1} and is the main result of our work. The most important part of this proof is the derivation of the convergence rate for the spectral density estimator $\hat{f}$ under the $L_\infty$ norm. In the original work, \citet{cai2010nonparametricfest} established an $L_2$ rate for a wavelet nonparametric mean estimator in a gamma regression where the data are assumed to be independent. In our work, the spectral density estimator $\hat{f}$ is based on the gamma distributed data $W_{i,1},\ldots,W_{i,p}$, which are only asymptotically independent. Moreover, the mean of these data is not exactly $f(\pi x_1),\ldots,f(\pi x_p)$, but is corrupted by the diagonalization error given in Lemma~\ref{lemma:DCTdiag}. This error adds to the error that arises via binning and VST and that describes the deviation from a Gaussian distribution, as derived in \citet{cai2010nonparametricfest}. Finally, we need to obtain an $L_\infty$ rather than an $L_2$ rate for our spectral density estimator. Overall, the proof requires different and partly novel tools than those used in \citet{cai2010nonparametricfest}. A particular challenge is the treatment of the dependence of $W_{i,1},\ldots,W_{i,p}$.  

To get the $L_\infty$ rate for $\hat{f}$, we first derive that for the periodic smoothing spline estimator $\widehat{H(f)}$ of the log-spectral density. To do so, we use a closed-form expression of its effective kernel obtained in \citet{schwarz2016unified}, thereby carefully treating various (dependent) errors that describe  deviations from a Gaussian nonparametric regression with independent errors and mean $f(\pi x_i)$. Note also that although the periodic smoothing spline estimator is obtained on $T$ binned points, the rate is given in terms of the vector dimension $p$ and the sample size $n$. Next, using the Cauchy-Schwarz inequality and a mean value argument, this rate is translated into the $L_\infty$ rate for the spectral density estimator $\hat{f}$. To obtain the rate for the Toeplitz covariance matrix estimator is enough to note that $E\|\widehat{\Sigma}-\Sigma\|^2\leq E\|\hat{f}-f\|^2_\infty$. 

\section{Practical Issues}
\label{sec:selection} 

Several choices must be made in practice to obtain our estimator. First, the data $Y_i\, (i=1,\ldots, n)$ are transformed into $W_{i,j}=(D_j^TY_i)^2 \, (j=1,\ldots,p)$ and are subsequently binned. 
According to Theorem \ref{theorem1}, the number of bins $T=\lfloor p^\upsilon\rfloor$ with any $\upsilon\in(1-\min\{\beta,1\}/3,1)$ leads to a minimax optimal estimator. That is, for $\beta\geq 1$ one can take any $\upsilon\in(2/3,1)$, while for $\beta<1$ the interval depends on $\beta$, for example, for $\beta=1/2$ the interval is $\upsilon\in(5/6,1)$. In our simulation studies we observed that the results are quite robust for various values for $\upsilon$. If no knowledge about $\beta$ is available, one can proceed as follows. Any standard test for long-range dependence can be performed and if the null hypothesis of long-range dependence is rejected, i.e., $\beta>1/2$, then any $\upsilon\in(5/6,1)$ can be taken. Otherwise, a smaller interval for $\upsilon$ should be considered. Additionally, one can always verify whether the chosen value of $T$ is appropriate by generating QQ-plots of $Y^*_k$, which ideally should show little departure from the normality. Several figures in the Supplementary Material show how the QQ-plots change for Gaussian data depending on $T$.

Once the data are transformed, the mean of $Y_k^*\,(k=1,\ldots, T)$, i.e., the log-spectral density, is estimated with a periodic smoothing spline. 
For this, one needs to choose basis functions of the periodic spline space, the penalty order $q\in\mathbb{N}$ and the smoothing parameter $h>0$.

The basis of a periodic spline space with knots put at the observations is a Fourier basis $\{2^{1/2}\cos(2\pi x),2^{1/2}\sin(2\pi x)\}$, evaluated at  $x_k=(k-1)/(2T-2)$ for $k=1,\ldots,2T-2$. 

The smoothing parameter $h$ can be chosen using any data-driven approach, such as (generalized) cross-validation or an empirical Bayes approach, see \citet{Wahba1985}. 

According to Theorem~\ref{theorem1}, the choice of $q$ should be related to the true smoothness $\beta$  of the spectral density in order to obtain the minimax optimal estimator. Assume that $q$ taken for estimation is larger than the true smoothness $\beta$. Then, the rate of convergence is determined by the true $\beta$ and is minimax, independent on how large $q$ is. In particular, if $\gamma=0$, then $\beta=\alpha\in(0,1)$ and taking $q=\max\{1,\gamma\}=1$ would lead to a minimax optimal estimator. If $q$ taken for the estimation is less than $\beta$, then the rate will depend on $q$ and on the choice of the smoothing parameter $h$. Assume that $\beta>q$ and the smoothing parameter $h$ is estimated with (generalized) cross-validation. Then, it has been shown in \cite{Wahba1985} that a periodic spline estimator adapts to the unknown smoothness up to $2q$. That is, if $\beta< 2q$, then the rate will be minimax optimal, while for  $\beta\geq 2q$ the rate will be determined by $q$. If $\beta>q$ and the smoothing parameter $h$ is estimated by the empirical Bayesian approach, then \cite{Wahba1985} has shown that the resulting estimator does not adapt to the extra smoothness. However, in small samples the empirical Bayes smoothing spline can perform similar or even better than the cross-validated smoothing spline due to smaller constants in the risk, see \cite{Kri2013}.

Since in practice $\beta $ is not known exactly, it is also not known which rate the estimator will have with a chosen $q$. However,  looking at the decay of the sample covariance functions one can get an idea whether $\beta$ is rather small, in case of a very slow decay, or rather large, in case of a very fast decay, and decide on the choice of $q$. A more attractive approach is to resort to adaptive estimation methods, that lead to the best possible estimators without prior knowledge on $\beta$. For example, the empirical Bayesian framework allows to estimate the unknown $\gamma$  under certain assumptions on the function space \citep{serra2017adaptive}. Other approaches for adaptive non-parametric estimation are aggregation and Lepski's method, see e.g., \cite{Chagny2016} for an overview and references. A detailed study of adaptive spectral density estimators is out of scope of our work.

In our simulation study and the real data example, presented in the next two sections, we discuss the parameter choices explicitly.
\section{Simulation Study} 
\label{sec:simulation}
In this section, we compare the performance of our proposed Toeplitz covariance estimator with the tapering estimator of \citet{cai2013optimal} and with the sample covariance matrix. A Monte Carlo simulation with $100$ samples is performed using R (version $4.1.2$, seed $42$). We consider Gaussian vectors $ Y_1,..., Y_n$ ${\overset{\text{i.i.d.}}{\sim}}\mathcal{N}_p(0_p, \Sigma)$  with (A) $p=5\,000,\, n=1$, (B) $p=1\,000,\, n=50$ and  (C) $p=5\,000,\, n=10$,  and  with the covariance functions $\sigma:\, \mathbb{Z} \to \mathbb{R},\, k\mapsto \sigma_k$ 
\begin{enumerate}
	\item[(1)] of a polynomial decay, i.e., $\sigma_k= 1.44(1+|k|)^{-5.1}$,
	\item[(2)] of an autoregressive process $y_t=\epsilon_{t}+0.1y_{t-1}-0.1y_{t-2}$ where $\epsilon_t$ is i.i.d. Gaussian noise and $\text{var}(\epsilon_t)=1.44$,
	\item[(3)] such that the corresponding spectral density is Lipschitz continuous but not differentiable: $f(x)= 1.44\{|\sin(x+0.5\pi)|^{1.7}+0.45\}$.
\end{enumerate}
In particular,  $\sigma_0=1.44$ in all three examples. Figure \ref{fig_Simulation_A} shows the spectral densities and the corresponding covariance functions for the three examples. 
\begin{figure}[h]
	\centering
	\begin{subfigure}[H]{0.32\textwidth}
		\caption*{ Process (1) }
		\includegraphics[width=\linewidth]{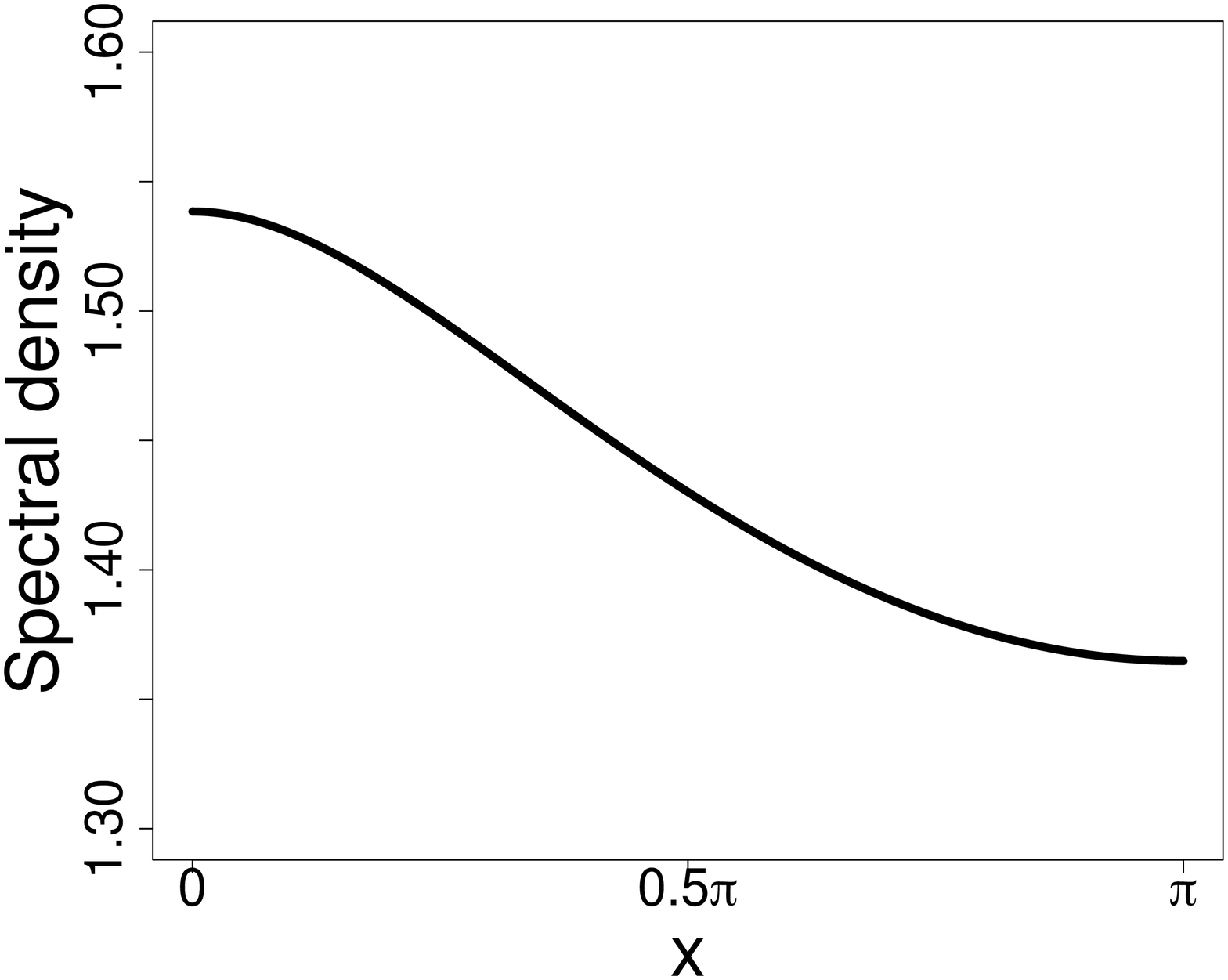} 
	\end{subfigure} \hfill
	\begin{subfigure}[h]{0.32\textwidth}
		\caption*{Process (2)}
		\includegraphics[width=\linewidth]{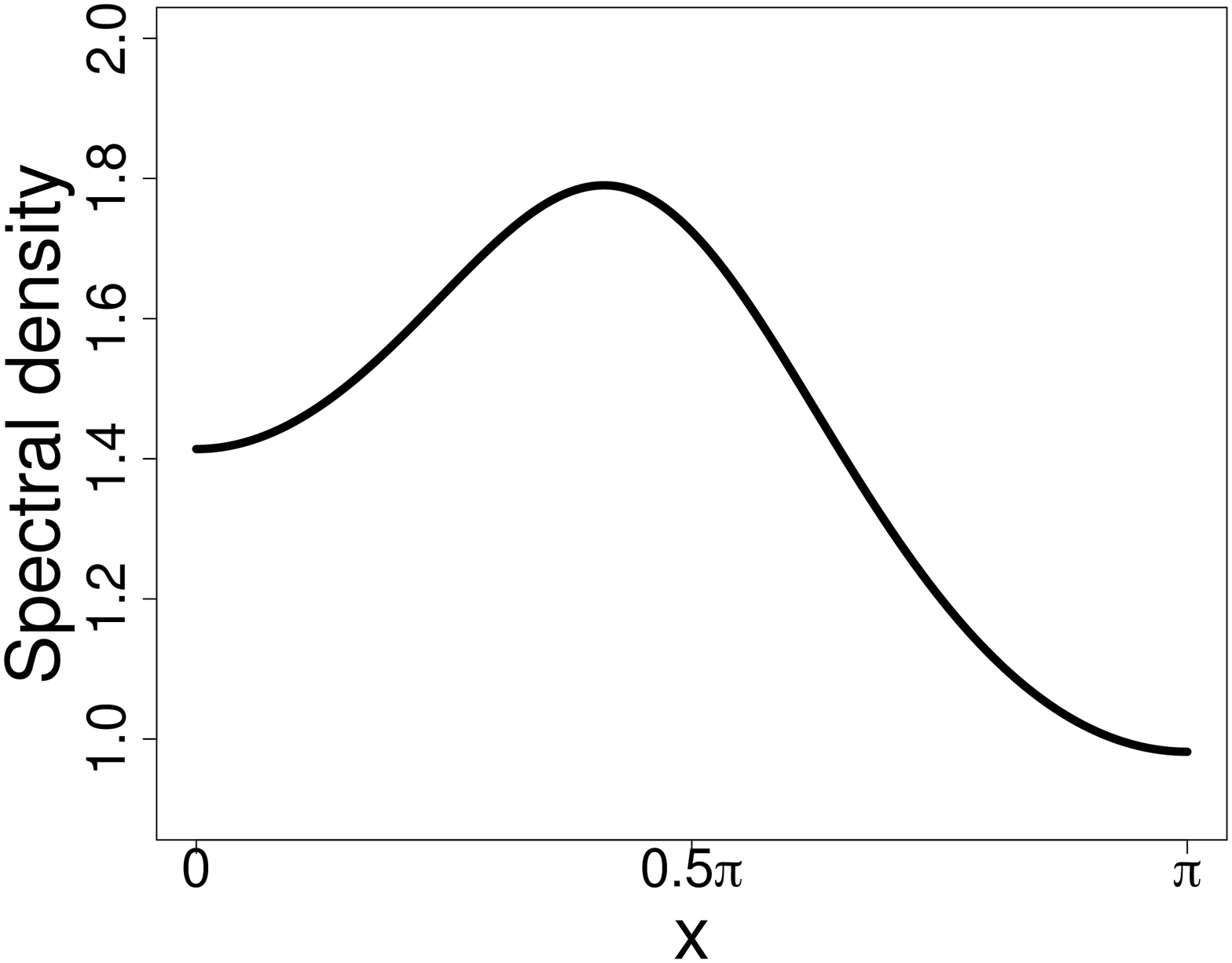} 
	\end{subfigure}\hfill
	\begin{subfigure}[h]{0.32\textwidth}
		\caption*{ Process (3) }
		\includegraphics[width=\linewidth]{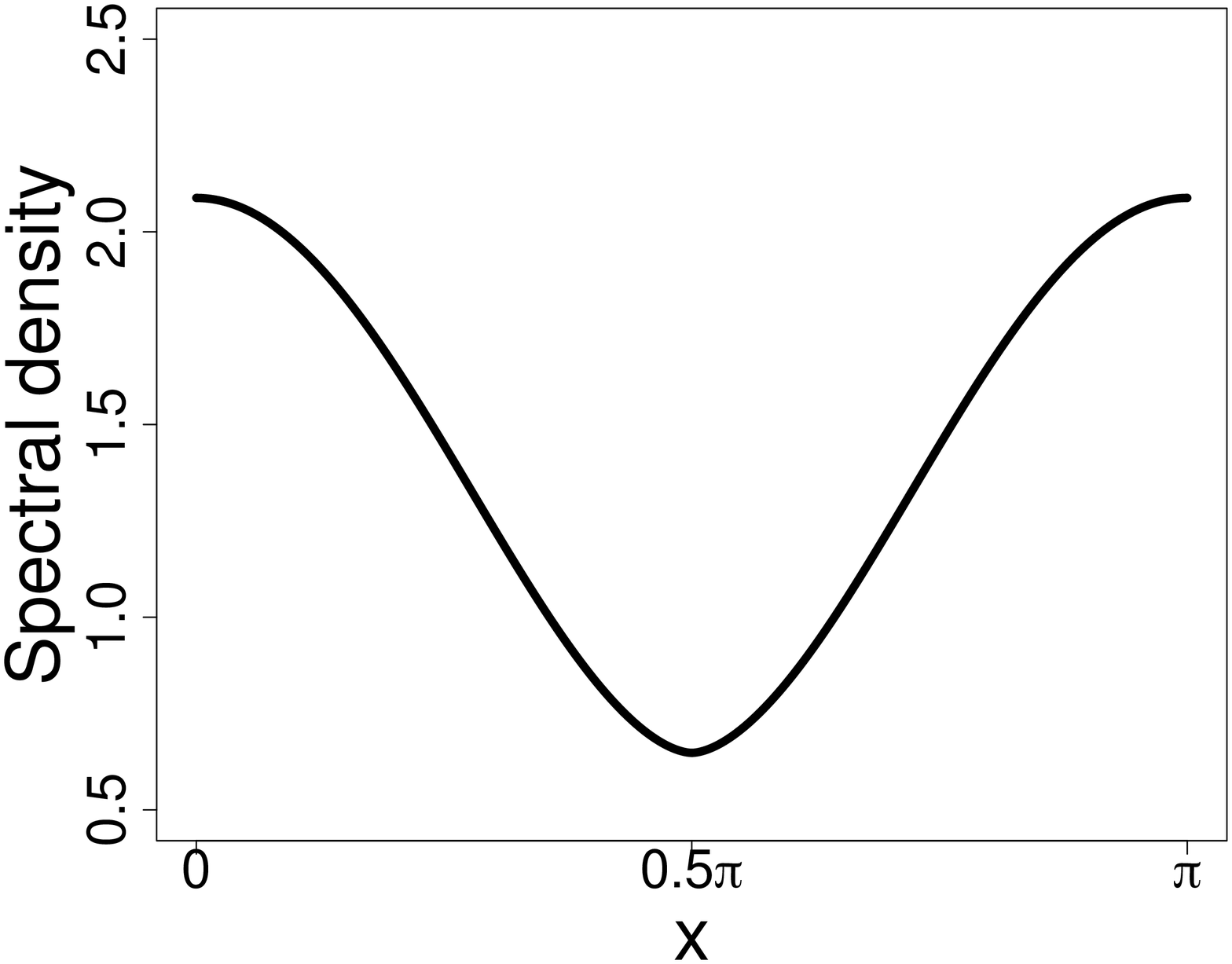} 
	\end{subfigure}\hfill
	\begin{subfigure}[h]{0.32\textwidth}
		\includegraphics[width=\linewidth]{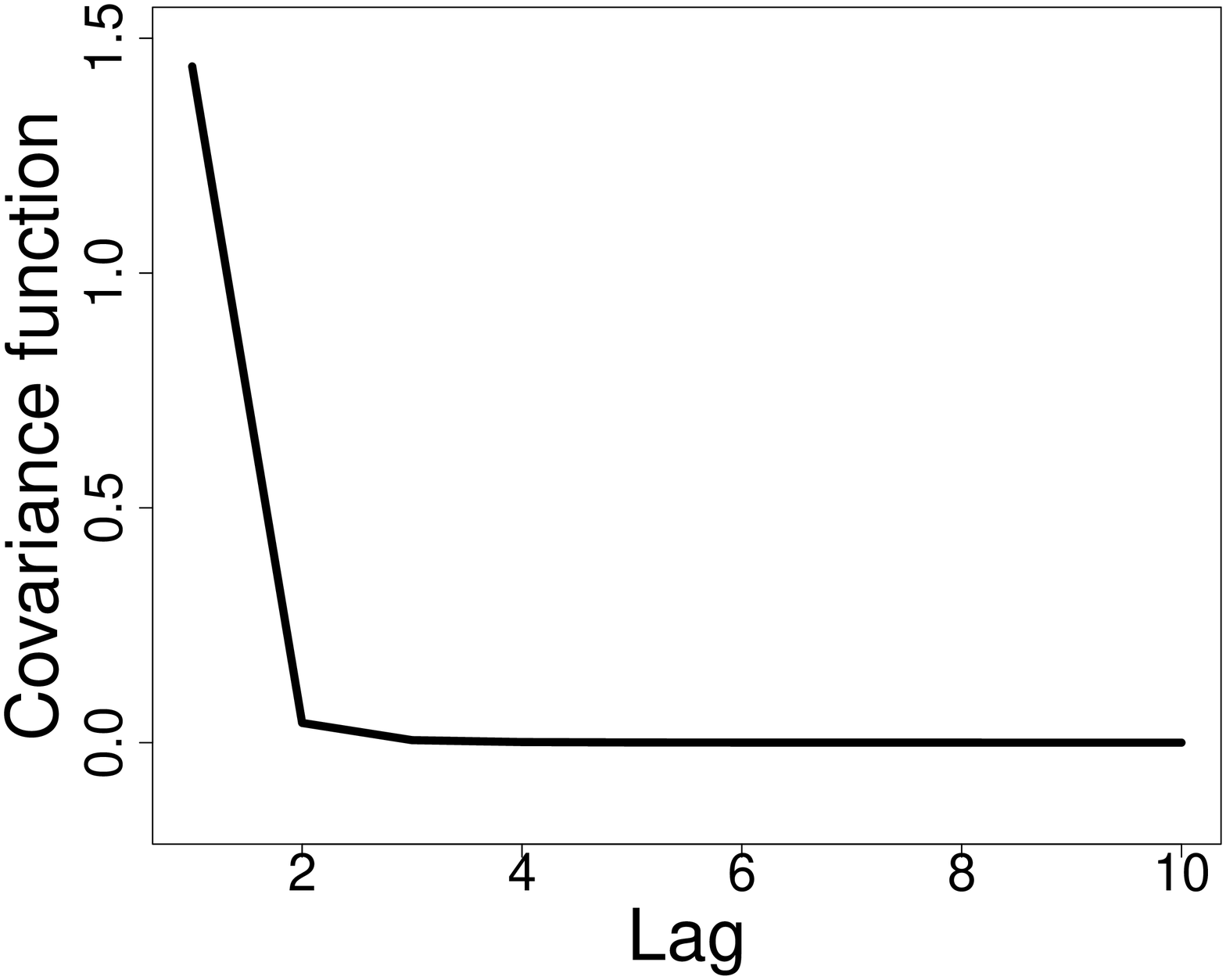} 
	\end{subfigure}\hfill
	\begin{subfigure}[h]{0.32\textwidth}		
		\includegraphics[width=\linewidth]{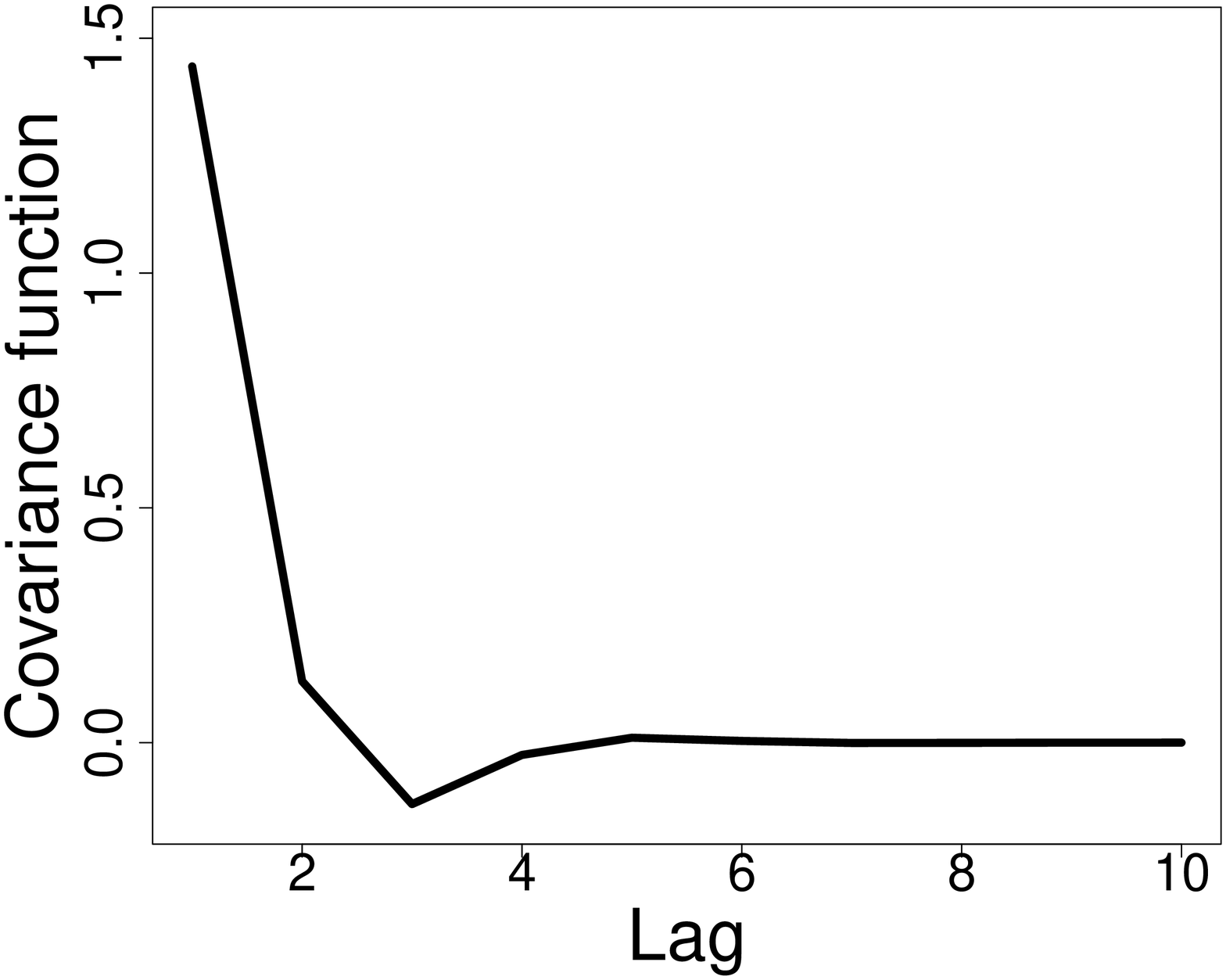} 
	\end{subfigure}\hfill
	\begin{subfigure}[h]{0.32\textwidth}		
		\includegraphics[width=\linewidth]{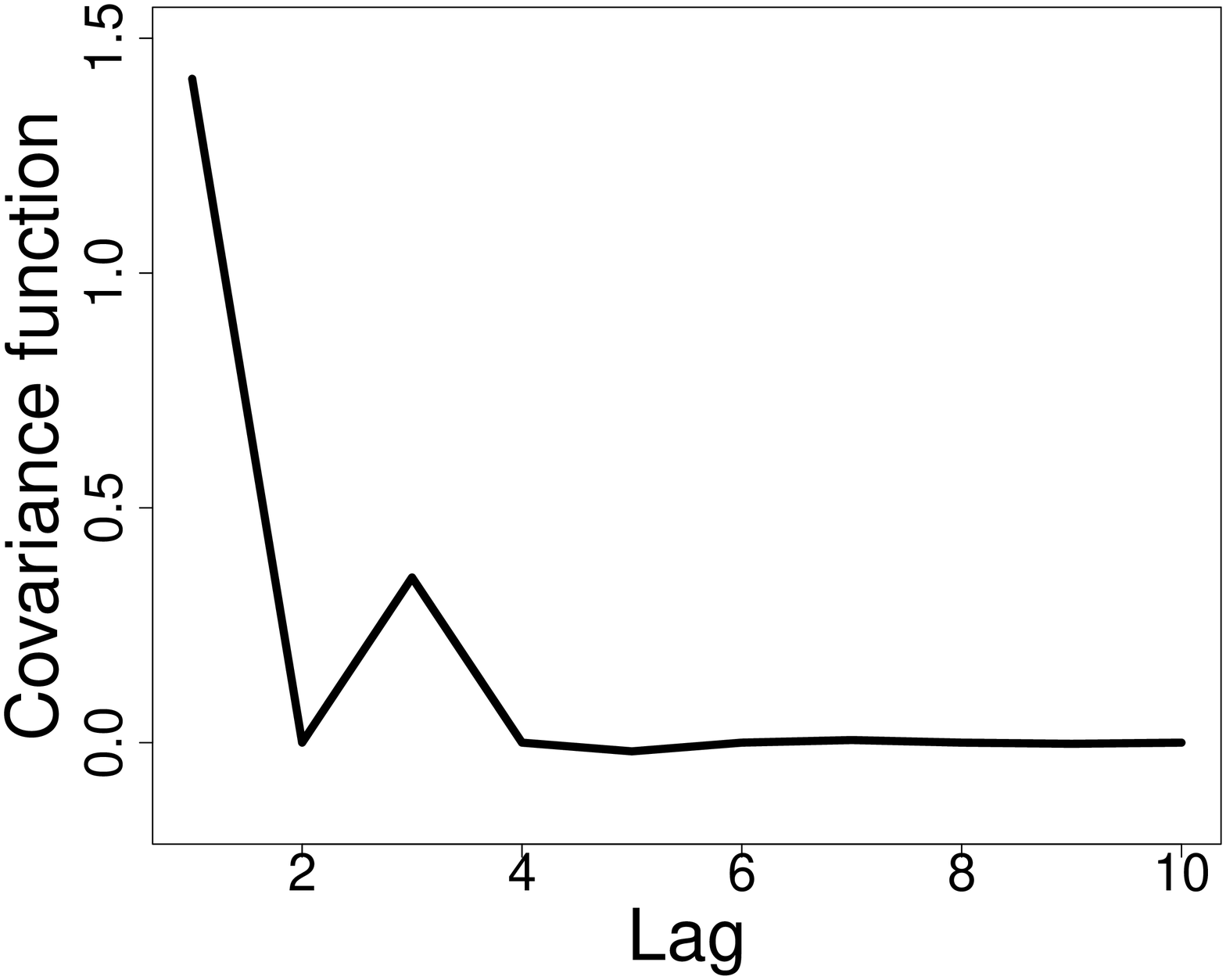} 
	\end{subfigure}\hfill
	\caption{ Spectral density functions (first row)  and covariance functions (second row) for examples (1)--(3). }
	\label{fig_Simulation_A}
\end{figure} 
To set the parameters note that the spectral density of the process (1) belongs to ${\cal{F}}_{4}$, the spectral density of the process (2) is an analytic function and the spectral density of the process (3) is from ${\cal{F}}_{1}$.

Since for all three processes $\beta\geq 1$, according to Theorem \ref{theorem1}, any value of $\upsilon\in(2/3,1)$ should lead to optimal estimation. We set the number of bins to $T=500$ bins in all scenarios, which corresponds to $\upsilon\approx 0.73$ for scenario (A) and (C) and $\upsilon\approx 0.90 $ for scenario (B). 

Setting $q=2$ would lead to the minimax optimal rates for processes (1) and (3), if the smoothing parameter is chosen by (generalized) cross-validation, see discussion in Section \ref{sec:selection}. According to arguments from the previous section, the convergence rate for the process (2) is the same as that for the process (1), while process (3) has a slower convergence rate. Hence, for a given scenario, we should observe similar magnitude of the average norms for processes (1) and (2), while somewhat larger values for process (3). Across scenarios, we expect scenario (A) to have larger values of the average norms, as only $p=5\,000$ observations are used, compared to $np=50\,000$ data points in scenarios (B) and (C).

To select the regularization parameter for our estimator, we implemented the restricted maximum likelihood method, generalized cross-validation and the corresponding oracle versions,  i.e., as  if $\Sigma$ were known. 

The sample covariance matrix $\widetilde{\Sigma}=(\tilde{\sigma}_{|i-j|})_{i,j=1}^p$ is defined as $n^{-1}\sum_{i=1}^{n}  Y_i Y_i^T$ with averaged diagonals to obtain Toeplitz structure. The tapering estimator with tapering parameter $k\leq p/2$ is defined as $\text{Tap}_k(\widetilde{\Sigma})=(\tilde{\sigma}_{|i-j|} w_{|i-j|})_{i,j=1}^p$, where ${w}_{m}=1$ when $m=0,...,k/2$,  ${w}_{m}=2-2m/k$ when $k/2<m\leq k$ and otherwise $w_m=0$, see \citet{cai2013optimal}.  The parameter $k$ can be selected using cross-validation \citep[see][]{bickel2008regularized} only if $n>1$, that is, under scenarios (B) and (C). For this, the $n$ observations are divided by $30$ random splits into a training set of size $n_1=2n/3$ and a test set of size $n_2=n/3$. Let $\Sigma_1^{\nu}$ and $\Sigma_2^{\nu}$ be the sample covariance matrices from the $\nu$th split. 
The tapering parameter $k$ is then estimated
as
$$ \hat{k}_{\text{cv}}=\arg \min_{k=2,3,...,p/2} \frac{1}{30}\sum_{\nu=1}^{30} \| \text{Tap}_k(\Sigma_1^ {\nu})-\Sigma_2^ {\nu}\|,$$
where $\text{Tap}_k(\Sigma_2^{\nu})$ is the tapering estimator with parameter $k$.  
If $n=1$, that is, under scenario (A), \citet{wu2009banding} suggest to split the time series $Y$ into $l$ non-overlapping subseries of length $p/l$ and then proceed as before to select the tuning parameter $k$. 
To the best of our knowledge, there is no data-driven method for the selection of $l$. Using the true covariance matrix $\Sigma$, we preselected oracle values $l=30$ subseries for the example $(1)$ and $l=15$ subseries for the examples $(2)$ and $(3)$. The parameter $k$ can then be chosen with cross-validation as above. We employ this approach under scenario (A) instead of an unavailable fully data-driven criterion and name it semi-oracle. Finally, for all three scenarios (A), (B) and (C), the oracle tapering parameter is computed using grid search for each Monte Carlo sample as $\hat{k}_{\text{or}}=\argmin_{k=2,3,...,p/2} \| \text{Tap}_k(\widetilde{\Sigma})-\Sigma\|, $ where $\widetilde{\Sigma}$ is the sample covariance matrix.  To speed up the computation, one can replace the spectral norm by the $\ell_1$ norm, as suggested by \citet{bickel2008regularized}.

\begin{table}[h]
	\centering
	\scriptsize
	\def~{\hphantom{0}}
	\caption{ (A) $p=5000,\, n=1$: Errors of the Toeplitz covariance matrix and the spectral density estimators with respect to the spectral and the $L_2$ norm;  average computation time of the covariance estimators in seconds for one Monte Carlo sample is in the last column; all numbers multiplied with 100 except the last column }{
		\begin{tabular}{lccccccc}
			&\multicolumn{2}{c}{ Process (1)} &\multicolumn{2}{c}{ Process (2) } &\multicolumn{2}{c}{ Process (3)} &time  \\
			& $\|\widehat{\Sigma}-\Sigma\|^2$ & $\|\hat{f}-f\|^2_2$ &  $\|\widehat{\Sigma}-\Sigma\|^2$ & $\|\hat{f}-f\|^2_2$ &  $\|\widehat{\Sigma}-\Sigma\|^2$ & $\|\hat{f}-f\|^2_2$ & in sec \\[5pt]		
			our method (GCV) & 0.688 & 0.255 & 1.591 & 0.439 & 3.401 & 0.606 & 4.235 \\ 
			our method (ML) & 0.591 & 0.224 & 1.559 & 0.417 & 3.747 & 0.628 & 4.224 \\ 
			tapering (semi-oracle) & 0.558 & 0.216 & 2.325 & 0.674 & 3.551 & 0.979 & 4.617 \\ 
			sample covariance & 16240.895 & 3810.680 & 20291.809 & 3694.392 & 24438.036 & 3809.486 & 0.342 \\ 
			our method (GCV-oracle) & 0.421 & 0.175 & 1.373 & 0.378 & 3.321 & 0.575 & \\ 
			our method (ML-oracle) & 0.464 & 0.186 & 1.487 & 0.391 & 3.781 & 0.624 & \\ 
			tapering (oracle) & 0.418 & 0.171 & 1.045 & 0.288 & 1.547 & 0.371 &  \\ 
	\end{tabular}}
	\label{tab_Simulation_A}
\end{table}
\begin{table}[h]
	\centering
	\scriptsize
	\def~{\hphantom{0}}
	\caption{ (B) $p=1000,\, n=50$: Errors of the Toeplitz covariance matrix and the spectral density estimators with respect to the spectral and the $L_2$ norm;  average computation time of the covariance estimators in seconds for one Monte Carlo sample is in the last column; all numbers multiplied with 100 except the last column }{
		\begin{tabular}{lccccccc}
			&\multicolumn{2}{c}{ Process (1)} &\multicolumn{2}{c}{ Process (2) } &\multicolumn{2}{c}{ Process (3)} &time  \\
			& $\|\widehat{\Sigma}-\Sigma\|^2$ & $\|\hat{f}-f\|^2_2$ &  $\|\widehat{\Sigma}-\Sigma\|^2$ & $\|\hat{f}-f\|^2_2$ &  $\|\widehat{\Sigma}-\Sigma\|^2$ & $\|\hat{f}-f\|^2_2$ & in sec \\[5pt]	
			our method (GCV) & 0.100 & 0.028 & 0.205 & 0.050 & 0.531 & 0.082 & 27.194 \\ 
			our method (ML) & 0.076 & 0.024 & 0.230 & 0.051 & 0.611 & 0.089 & 27.098 \\ 
			tapering (CV) & 0.110 & 0.031 & 0.218 & 0.055 & 0.348 & 0.073 & 23.908 \\ 
			sample covariance & 79.603 & 56.262 & 95.090 & 61.000 & 127.528 & 61.001 & 0.141 \\ 
			our method (GCV-oracle) & 0.062 & 0.020 & 0.163 & 0.043 & 0.462 & 0.074 & \\ 
			our method (ML-oracle) & 0.067 & 0.021 & 0.221 & 0.050 & 0.603 & 0.088 & \\ 
			tapering (oracle) & 0.057 & 0.020 & 0.133 & 0.037 & 0.265 & 0.055 &\\ 
	\end{tabular}}
	\label{tab_Simulation_B}
\end{table}

\begin{table}[h]
	\centering
	\scriptsize
	\def~{\hphantom{0}}
	\caption{ (C) $p=5000,\, n=10$: Errors of the Toeplitz covariance matrix and the spectral density estimators with respect to the spectral and the $L_2$ norm;  average computation time of the covariance estimators in seconds for one Monte Carlo sample is in the last column; all numbers multiplied with 100 except the last column }{
		\begin{tabular}{lccccccc}
			&\multicolumn{2}{c}{ Process (1)} &\multicolumn{2}{c}{ Process (2) } &\multicolumn{2}{c}{ Process (3)} &time  \\
			& $\|\widehat{\Sigma}-\Sigma\|^2$ & $\|\hat{f}-f\|^2_2$ &  $\|\widehat{\Sigma}-\Sigma\|^2$ & $\|\hat{f}-f\|^2_2$ &  $\|\widehat{\Sigma}-\Sigma\|^2$ & $\|\hat{f}-f\|^2_2$ & in sec \\[5pt]
			our method (GCV) & 0.088 & 0.026 & 0.217 & 0.050 & 0.593 & 0.079 & 4.260 \\ 
			our method (ML) & 0.078 & 0.023 & 0.231 & 0.050 & 0.677 & 0.086 & 4.251 \\ 
			tapering (CV) & 0.143 & 0.034 & 0.217 & 0.051 & 0.422 & 0.070 & 635.345 \\ 
			sample covariance & 673.122 & 370.946 & 792.714 & 360.687 & 1014071.587 & 375.728 & 1.189 \\ 
			our method (GCV-oracle) & 0.062 & 0.020 & 0.172 & 0.043 & 0.500 & 0.071 &  \\ 
			our method (ML-oracle) & 0.069 & 0.021 & 0.224 & 0.048 & 0.663 & 0.085 &  \\ 
			tapering (oracle) & 0.055 & 0.018 & 0.147 & 0.039 & 0.257 & 0.051 &  \\ 
	\end{tabular}}
	\label{tab_Simulation_C}
\end{table}
In Tables \ref{tab_Simulation_A}, \ref{tab_Simulation_B} and  \ref{tab_Simulation_C}, the errors of the Toeplitz covariance estimators with respect to the spectral norm and the average computation time for one Monte Carlo sample for all three processes are reported for scenarios  (A),  (B) and (C), respectively. To illustrate the goodness-of-fit of the spectral density, the $L_2$ norm $\|\hat{f}-f\|_2$ is also computed. 

The overall behavior of our estimator is exactly as expected. Moreover, the tapering and our estimator perform similar in terms of the spectral norm risk. This is not surprising as both estimators are proved to be rate-optimal. The oracle estimators show similar behavior, but are slightly less variable compared to the data-driven estimators.  Clearly, both the tapering and our estimators are superior to the inconsistent sample covariance matrix. In terms of computational time, both methods are similarly fast for scenarios (A) and (B). For scenario (C), the tapering method is much slower due to the multiple high-dimensional matrix multiplications in the cross-validation method. It is anticipated that for larger $p$ the tapering estimator is much more computationally intensive compared to our method. 


\section{Application to Non-Gaussian Data} \label{sec:nonGauss}
While we consider the rigorous theoretical study of our estimator for non-Gaussian data to be out of scope of this work, we would like to discuss in this section the application of our method to non-Gaussian data in practice. 

To apply our method, one needs to ensure that the transformed data $Y_k^*$ are approximately Gaussian with mean $H\{f(x_k)\}$. Our estimator will be minimax optimal if the deviation from Gaussianity of $Y_k^*$ is sufficiently small. In general, one needs to ensure that
\begin{itemize}
	\item[(i)] the mean of $(D_j^TY_i)^2$ is the spectral density up to a negligible error  $f(\pi x_j)\{1+o(1)\}$ , 
	\item[(ii)] the central limit theorem is applicable to $Q_k$ after appropriate centering and scaling,
	\item[(iii)] the $\log$-transform is variance-stabilizing for $(D_j^TY_i)^2$.
\end{itemize}
The first point is always satisfied, as long as both moments of $D_j^TY_i$ exist. Indeed, suppose that $Y_i$  has a Toeplitz covariance matrix $\Sigma$ and the marginal distribution is non-Gaussian with mean zero. Then, $E\{(D_j^TY_i)^2\}=(D^T\Sigma D)_{jj}=f(\pi x_j)\{1+o(1)\}.$ The last two points can be checked explicitly, if the distribution of $Y_i$ is known. For example, for gamma distributed data the second point is clearly satisfied and the third point is proved in the Supplementary Material. 

Of course, in practice the distribution of $Y_i$ is rarely known. Since our method relies on the asymptotic normality of $Y^*_k$, one can simply check whether the data deviate from normality strongly. 
Note that  application of normality tests for $Y^*_k$ might be misleading in small samples, since $Y^*_k$ are only asymptotically Gaussian,  even when $Y_i$ are Gaussian. 
Therefore, we suggest to generate a QQ-plot of $Y^*_k$. If this QQ-plot shows little deviations from Gaussianity, then our method can safely be applied in practice. If deviations from normality are substantial, this might indicate that the log-transform of the data is not suitable. To find an appropriate variance-stabilizing transformation, one can employ a Box-Cox transform. Note that to estimate the Box-Cox transformation parameter, one must take into account correlation of the binned data, e.g., by using the method of \cite{guerrero1993time} developed for time series. 

In the Supplementary Material we provide a small simulation study, as well as several QQ-plots of $Y^*_k$, where the distribution of $\Sigma^{-1/2}Y_i$ was taken to be gamma and uniform. It can be observed that in all examples the QQ-plots show little departure from the Gaussian distribution and all simulation results look very similar independent on the distribution of $Y_i$. 

\section{Application to Protein Dynamics}
\label{sec:realdata}
We revisit the data analysis of protein dynamics performed in \citet{krivobokova2012partial} and \citet{singer2016partial}. We consider data generated by the molecular dynamics simulations for the yeast aquaporin, the gated water channel of the yeast {\it Pichi pastoris}. Molecular dynamics simulations are an established tool for studying biological systems at the atomic level on timescales of nano- to microseconds. The data are given as Euclidean coordinates of all $783$ atoms of the aquaporin observed in a $100$ nanosecond time frame, split into $20\,000$ equidistant observations. Additionally, the diameter of the channel $y_t$ at time $t$ is given, measured by the distance between two centers of mass of certain residues of the protein. The aim of the analysis is to identify the collective motions of the atoms responsible for the channel opening. In order to model the response variable $y_t$, which is a distance, based on the motions of the protein atoms, we chose to represent the protein structure by distances between atoms and certain fixed base points instead of Euclidean coordinates. That is, we calculated
$$
X_t=\{d(A_{t,1},B_1),\ldots,d(A_{t,783},B_1),d(A_{t,1},B_2),\ldots,d(A_{t,783},B_y)\}\in\mathbb{R}^{4\cdot 783},
$$
where $A_{t,i}\in\mathbb{R}^3$, $i=1, \ldots,783$ denotes the $i$th atom of the protein at time $t$, $B_j\in\mathbb{R}^3 \quad (j=1,...,4)$, is the $j$th base point and $d(\cdot,\cdot)$ is the Euclidean distance. Figure \ref{fig_Aqua} shows the diameter $y_t$ and the distance between the first atom and the first center of mass. \\
\begin{figure}[h]
	\centering
	\begin{subfigure}[h]{0.48\textwidth}
		\includegraphics[width=\linewidth]{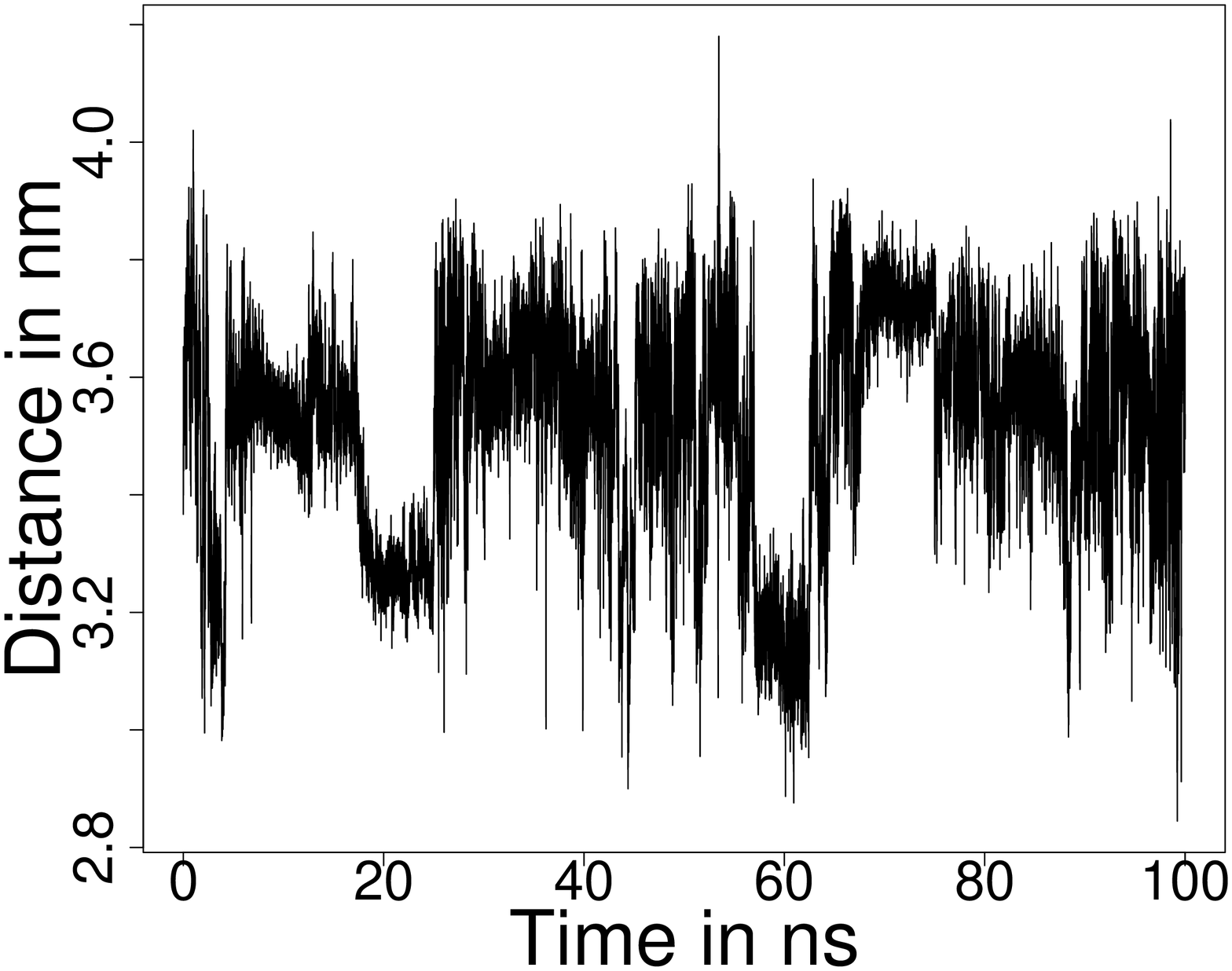}
	\end{subfigure}
	\begin{subfigure}[h]{0.48\textwidth}
		\includegraphics[width=\linewidth]{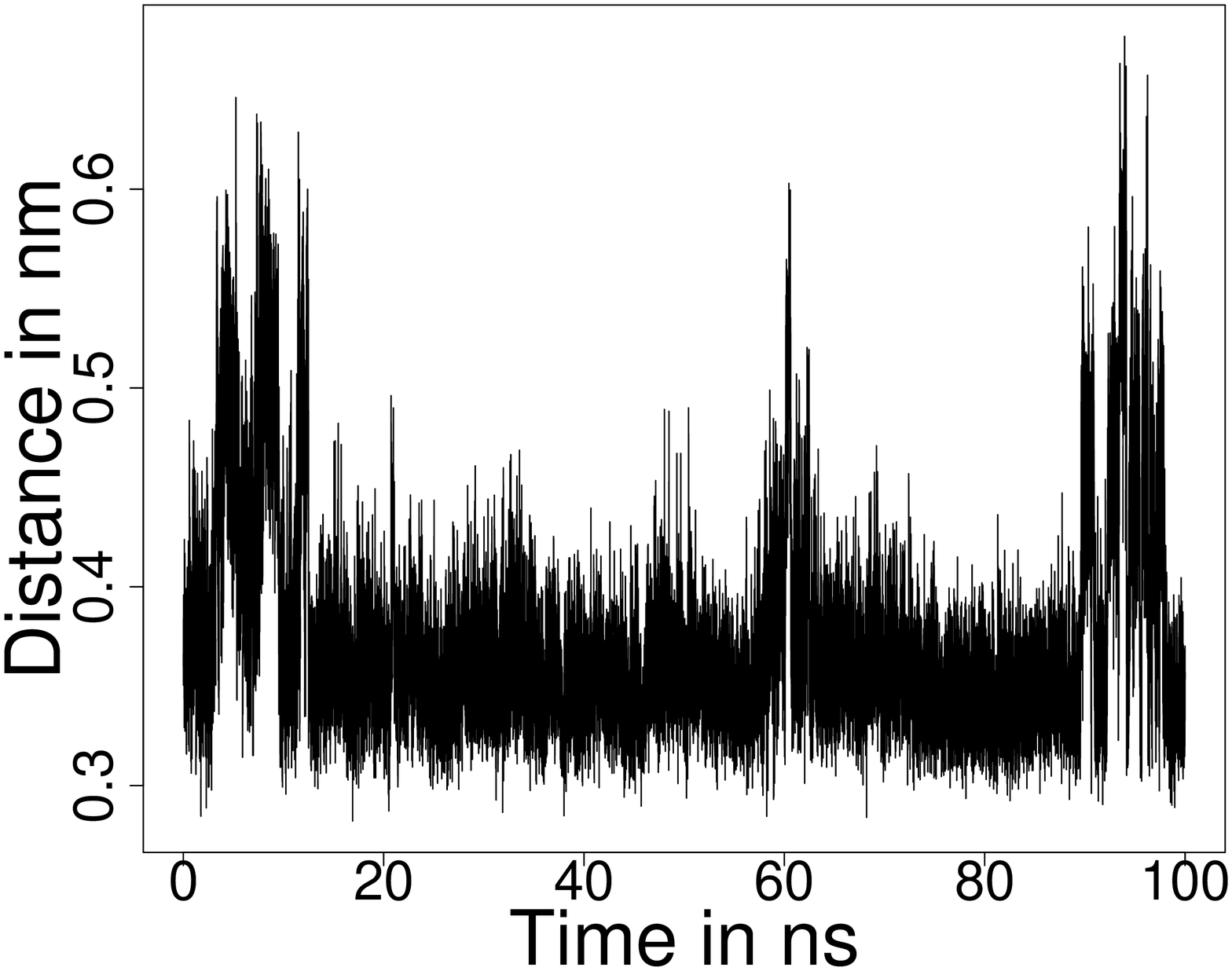}
	\end{subfigure}
	\caption{Distance between the first atom and the first center of mass of aquaporin (left) and the opening diameter $y_t$ over time $t$ (right).}
	\label{fig_Aqua}
\end{figure}
It can be concluded that a linear model $Y=Xb+\epsilon$ holds, where $Y=(y_1,\ldots,y_{20\,000})^T$, $X=(X_1^T,\ldots,X_{20\,000}^T)^T$, $b\in\mathbb{R}^{4\cdot 783}$, $\epsilon\in\mathbb{R}^{20\,000}$. This linear model has two specific features which are intrinsic to the problem: first, the observations are not independent over time and second, $X_t$ is high-dimensional at each $t$ and only few columns of $X$ are relevant for $Y$. \citet{krivobokova2012partial} have shown that the partial least squares (PLS) algorithm performs exceptionally well on this type of data, leading to a small-dimensional and robust representation of proteins, which is able to identify the atomic dynamics relevant for $Y$. \citet{singer2016partial} studied the convergence rates of the PLS algorithm for dependent observations and showed that decorrelating the data before running the PLS algorithm improves its performance. Since $Y$ is a linear combination of columns of $X$, it can be assumed that ${Y}$ and all columns of ${X}$ have the same correlation structure. Hence, it is sufficient to estimate $\Sigma=\mbox{cov}(Y)$ to decorrelate the data for the PLS algorithm, i.e., 
$\Sigma^{-1/2}Y=\Sigma^{-1/2}{X}b+\Sigma^{-1/2}\epsilon$
results in a standard linear regression with independent errors. 

Our goal now is to estimate $\Sigma$ and compare the performance of the PLS algorithm on original and decorrelated data.
For this purpose, we divided the data set into a training and a test set, each with $p=10\,000$ observations. First, we tested whether the data are stationary. The augmented Dickey-Fuller test confirmed stationarity for $Y$ with a $p$-value$<0.01$. The Hurst exponent of $Y$ is $0.85$, indicating moderate long-range dependence supported by a rather slow decay of the sample covariances, see the grey line in the left plot of Fig. \ref{PLS-results}. Therefore, we set $q=1$ for our estimator to match the low smoothness of the corresponding spectral density. Application of the R package \texttt{forecast}, that implements the approach of  \citet{guerrero1993time} to estimate the Box-Cox transform parameter, confirms that the log-transform is appropriate for these data. The bin number is set to $T=2500$, i.e., $\upsilon\approx 0.85$. The smoothing parameter of the log-spectral density is selected with generalized cross-validation. 
The black line in the left plot of  Fig. \ref{PLS-results} confirms that the covariance matrix estimated with our method almost completely decorrelates the channel diameter $Y$ on the training data set.  Next, we estimated  the regression coefficients $b$ with the usual partial least squares algorithm, ignoring the dependence in the data. Finally, we estimated $b$ with the partial least squares algorithm that takes into account dependence using our covariance estimator $\widehat{\Sigma}$. Based on these regression coefficient estimators, the prediction on the test set was calculated. The plot on the right side of Fig. \ref{fig_Aqua} shows the Pearson correlation between the true channel diameter on the test set and the prediction on the same test set based on raw (in grey) and decorrelated data (in black). Obviously, the performance of the partial least squares algorithm on the decorrelated data is significantly better for smaller numbers of components. In particular, with just one component, the correlation between the true opening diameter on the test set and its prediction that takes into account the dependence in the data is already $0.45$, while it is close to zero for the partial least squares method that ignores the dependence in the data. \citet{krivobokova2012partial} showed that the estimator of $b$ based on one PLS component is exactly the ensemble-weighted maximally correlated mode, which is defined as the collective mode of atoms that has the highest probability to achieve a specific alteration of the response $Y$. Therefore, an accurate estimator of this quantity is crucial for the interpretation of the results and can only be achieved if the dependence in the data is taken into account.  
\begin{figure}[h]
	\centering
	\begin{subfigure}[h]{0.48\textwidth} 	
		\includegraphics[width=\linewidth]{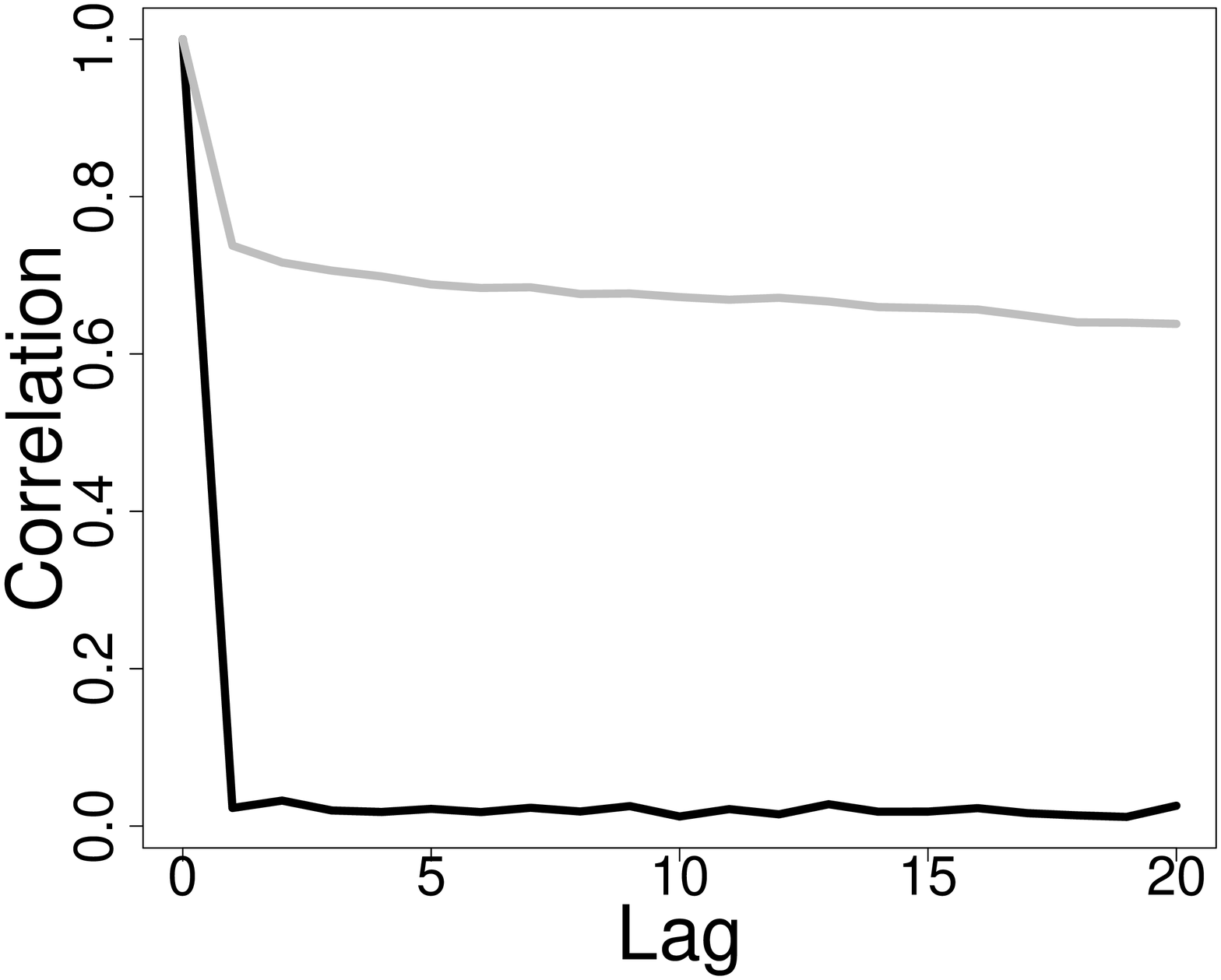}
	\end{subfigure}
	\begin{subfigure}[h]{0.48\textwidth}
		\includegraphics[width=\linewidth]{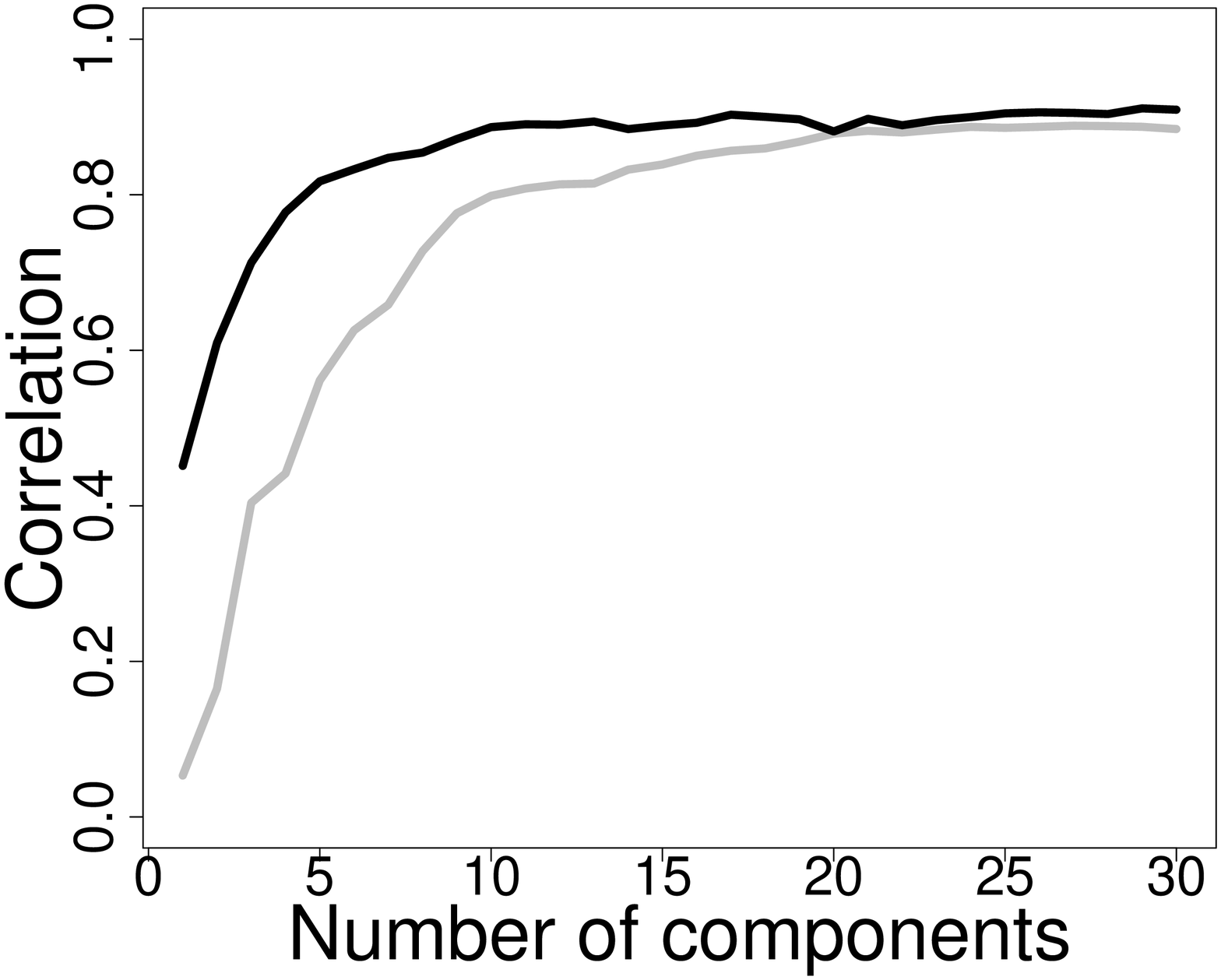}	
	\end{subfigure}		
	\caption{On the left, the correlation function of $Y$ (grey) and of $\widehat{\Sigma}^{-1/2}Y$ (black), where $\widehat{\Sigma}$ is estimated with our method; On the right, correlation between the true values on the test data set and prediction based on partial least squares (grey) and corrected partial least squares (black).}
	\label{PLS-results}
\end{figure}

Estimating $\Sigma$ with a tapered covariance estimator has two practical problems. First, since we only have a single realization of a time series $Y$, i.e., $n=1$, there is no data-driven method for  selecting the tapering parameter. Second, the tapering estimator turned out  to be not positive definite for the data at hand. To solve the second problem, we truncated the corresponding spectral density estimator $\hat{f}_{\text{tap}}$ to a small positive value, i.e., $\hat{f}_{\text{tap}}^+=\max\{\hat{f}_{\text{tap}}, 1/\log p \}$ \citep[see][]{cai2013optimal,mcmurry2010banded}. To select the tapering parameter with cross-validation, we experimented with different subseries lengths and found that the tapering estimator is very sensitive to this choice. For example, estimating the tapered covariance matrix based on subseries of length $8/15/30$ yields a correlation of $0.42/0.53/0.34$ between the true diameter and the first component, respectively.  

Altogether, our proposed estimator is fully data-driven, fast even for large sample sizes, automatically positive definite and can handle certain long-memory processes. The protein data example and further simulations in the Supplementary Material suggest that our approach yields robust estimators even when data are not normally distributed.  In contrast, the tapering estimator is not data-driven in all data scenarios and must be manipulated to become positive definite. To the best of our knowledge, the tapering estimator has been not studied for non-Gaussian data. Our method is implemented in the R package \texttt{vstdct}. 



\begin{appendix}
	\label{appendix}
	\appendix
	\allowdisplaybreaks
	\numberwithin{equation}{section}
	\section{Appendix}\label{app:aux}
	Throughout the appendix, we denote by $c, c_1, C, C_1, \dots $ etc. generic constants, that are independent of  $n$ and $p$. To simplify the notation, the constants are sometimes skipped and we write $\lesssim$ for less than or equal to up to constants.

	\setcounter{lemma}{1}
	\subsection{Proof of Lemma~1} \label{app:lemma1}
	We embed the $p$-dimensional Toeplitz  matrix $ \Sigma=\text{toep}(\sigma_0,\dots, \sigma_{p-1})$ in a $(2p-2)$-dimensional circulant matrix  $\widetilde{ \Sigma}=\text{toep}(\sigma_0,\dots, \sigma_{p-1}, \sigma_{p-2},\dots, \sigma_{1})$. 
	Then, $\widetilde{ \Sigma}=  U  \Lambda  U^*,$	where $ U=(2p-2)^{-1/2}[\exp\left\{\pi \texttt{i} (k-1)(l-1)/(p-1)\right\}]_{k,l=1}^{2p-2}$  with the conjugate transpose $ U^*$, and $ \Lambda$ is  a diagonal matrix with the $k$th diagonal value for $k=1,...,p$ given by
	\begin{align*}
	\lambda_k=&\sum_{l=1}^{p} \sigma_{l-1} \exp\{ \pi \texttt{i} x_k(l{-}1)\} {+}\sum_{l=p+1}^{2p-2} \sigma_{2p-1-l} \exp\{ \pi \texttt{i} x_k(l{-}1)\},
	\end{align*} where $x_k=(k-1)/(p-1)$.   
	By symmetry, $\lambda_k=\lambda_{2p-k}$ for $k=p+1,...,2p-2$.
	Furthermore,	
	$ \Sigma=  V^*  \Lambda V,$
	where $ V \in \mathbb{C}^{(2p-2)\times p}$ contains the first $p$ columns of $ U$. 
	Then, for $i,j=1,\dots,p$,	
	\begin{align*}
	( D  \Sigma  D )_{i,j}=( D  V^*  \Lambda  V   D )_{i,j}=   \sum_{r=1}^{2p-2} \lambda_r \sum_{q=1}^{p}  D_{iq}    V^*_{qr} \sum_{s=1}^{p}  V_{rs}  D_{sj}.
	\end{align*}
	Define the scaling factor $a_{s}= 2^{-1/2}$  for $s=1,p$ and $a_s=1$ for $s=2,...,p-2$, appearing in the definition of the discrete cosine transform I matrix.  Then,
	\begin{align}
	\sum_{s=1}^{p}  V_{rs}  D_{sj} &= (p-1)^{-1} \sum_{s=1}^{p} a_{s} a_{j} \exp\{ \pi \texttt{i} (r-1)x_s\} \cos\{ \pi (j-1)x_s\} \nonumber \\
	&\quad \pm \frac{a_j}{p-1} \exp(0)\cos(0) \pm \frac{a_j}{p-1} \exp\{\pi\mathrm{i}(r-1)\})\cos\{\pi(j-1)\}  \nonumber \\
	&= a_j\frac{2^{1/2}-2}{p-1} \mathbbm{1}_{\{|j-r| \text{ even}\}} +\frac{a_j}{p-1} \sum_{s=1}^{p}\exp\{ \pi \texttt{i} (r-1)x_s\} \cos\{ \pi (j-1)x_s\}\nonumber\\ 
	&=a_jb(j,r) +a_j\texttt{i}c(j,r), \nonumber
	\end{align}
	where $\mathbbm{1}$ is the indicator function and
	\begin{align*}
	&b(j,r)= \frac{2^{1/2}-2}{p-1} \mathbbm{1}_{\{|j-r| \text{ even}\}} +\frac{1}{p-1}\sum_{s=1}^{p} \cos\{ \pi (r-1)x_s\} \cos\{ \pi (j-1)x_s\},\\
	&c(j,r)=\frac{1}{p-1}\sum_{s=1}^{p} \sin\{ \pi (r-1)x_s\} \cos\{ \pi (j-1)x_s\}.
	\end{align*}
	By the symmetry properties of cosine and sine, $b(j,r)=b(j,2p-r)$ and $c(j,r)=-c(j,2p-r)$ for  $r=p+1,\dots,2p-2$.
	Using $ (\sum_{s=1}^{p} V_{rs} D_{sj} )^* =  \sum_{q=1}^{p} D_{jq} V^*_{qr}$, we obtain
	\begin{align}
	( D  \Sigma  D )_{i,j} = 
	&\begin{cases}
	a^2_j\sum_{r=1}^{2p-2} \lambda_r \left \{b(j,r)^2+c(j,r)^2 \right \}, &i=j\\
	a_ia_j\sum_{r=1}^{2p-2}  \lambda_r(b(i,r)+\texttt{i}c(i,r))(b(j,r)-\texttt{i}c(j,r)), &i\neq j.
	\end{cases} \label{absum}	
	\end{align}
	Some calculations show, for $r=1,\dots p$, that
	\begin{align}  \linespread{0.3}
	&\frac{1}{p{-}1}\sum_{s=1}^{p} \cos\{ \pi (r{-}1)x_s\}  \cos\{ \pi (j{-}1)x_s\}
	=\begin{cases}
	\frac{1}{p-1}, & r\neq j, r-j\text{ even}\\
	0, & r-j\text{ odd}\\
	\frac{p+1}{2p-2},&r=j\neq 1,p\\
	\frac{p}{p-1},&r=j=1,p
	\end{cases} \label{coscos}\\
	&\frac{1}{p{-}1}\sum_{s=1}^{p} \sin\{ \pi (r{-}1)x_s\}  \cos\{ \pi (j{-}1)x_s\}\nonumber \\
	&=\frac{\{1-(-1)^{r-j}\}}{4(p-1)} \left [\text{cot} \{\pi(r-j)/(2p-2)\}  +\text{cot} \{\pi(j+r-2)/(2p-2)\} \right].\label{sincos} 
	\end{align}
	Using the Taylor expansion of the cotangens, i.e., $\cot(x)=x^{-1}-x/3-x^3/45+...$, for $0<|x|<\pi$, one obtains for $r=1,\dots p$
	\begin{equation}
	c(j,r)=\frac{(r-1)\{1-(-1)^{r-j}\}}{\pi\{(r-1)-(j-1)\}\{(r-1)+(j-1)\}} +\mathcal{O}\left \{ \frac{2\pi(2r-2)}{3(2p-2)^2}\right \}, \label{sincos2} 
	\end{equation} where the $\mathcal{O}$ term does not depend on $j$ and the constant does not depend on $r,p$.\\
	
	{\centering \textit{Case $i=j$}}
	
	The equations (\ref{absum}) -- (\ref{sincos}) imply
	\begin{align*}
	(\bm D \bm \Sigma \bm D )_{j,j}= & \, a^2_j \sum_{\substack{r=1 \\ r \neq j,\\	r\neq 2p-j}}^{2p-2}  \lambda_r \left [\frac{(2^{1/2}-2)^2+1+2(2^{1/2}-2)(p-1)}{(p-1)^{2}}\mathbbm{1}_{\{|r-j| \text{ even}\}} +	c(j,r)^2\right ] \\
	&\,+a_j^2\lambda_j \cdot\begin{cases}
	2\left(\frac{2^{1/2}-2}{p-1} +\frac{p+1}{2p-2}\right)^2, &j\neq 1,p\\
	\left(\frac{2^{1/2}-2}{p-1} +\frac{p}{p-1}\right)^2, &j=1,p
	\end{cases} \\
	= & \,a^2_j \sum_{\substack{r=1 \\ r \neq j,\\	r\neq 2p-j}}^{2p-2}  \lambda_r \left [\frac{(2^{1/2}-1)^2}{(p-1)^{2}}\mathbbm{1}_{\{|r-j| \text{ even}\}} +	c(j,r)^2\right ] +\lambda_j \cdot\begin{cases}
	\frac{2(2^{1/2}-3/2 +p/2)^2}{(p-1)^2}, &j\neq 1,p\\
	\frac{(2^{1/2}-2 +p)^2}{2(p-1)^2}, &j= 1,p
	\end{cases}\\
	= &\,\mathcal{O}(p^{-1}) + 2a^2_j\sum_{r=1}^{p-1} \lambda_r  c(j,r)^2 +\lambda_j\{0.5+\mathcal{O}(p^{-1})\} ,
	\end{align*} where we used the symmetry properties of $\lambda_r$ and that by assumption $\| f\|<M_0$. Note that the $\mathcal{O}$ terms do not depend on $j$.
	Since the complex exponential function is Lipschitz continuous with Lipschitz constant $L=1$, it holds $\lambda_{r}= \lambda_j +L_{r,j} |r-j|(p-1)^{-1}$ where $-1\leq L_{r,j}\leq 1$ is a constant depending on $r,j$. 
	Then,
	\begin{align*}
	\sum_{r=1}^{p-1}  \lambda_r  c(j,r)^2&=\lambda_j\sum_{r=1}^{p-1}   c(j,r)^2 + p^{-1} \sum_{r=1}^{p-1}  L_{r,j}|r-j| c(r,j)^2.%
	\end{align*}
	Since $\sum_{r=1}^{p-1}  \lambda_r  c(1,r)^2=-\sum_{r=1}^{p-1}  \lambda_r  c(p,r)^2,$ it is sufficient to consider $j=1,...,p-1$. We begin with first sum.  For a shorter notation, we use  $k=r-1$ and $l=j-1$ in the following. Then, summing the squares of the first term in (\ref{sincos2}) for $l=0,...,p-2$ yields
	\begin{align*}
	\frac{1}{\pi^2}\sum_{k=0}^{p-2}   \frac{k^2 \{1{-}(-1)^{k-l}\}^2}{ (k{-}l)^2(k{+}l)^2}{=} \frac{4}{\pi^2}\begin{cases}
	\sum_{x=1}^\infty \frac{1}{(2x{-}1)^2} {-} R_1(l,p)=\frac{\pi^2}{8}{-}R_1(l,p), &l=0\\
	\sum_{x=1}^\infty \frac{(2x)^2}{(2x{-}l)^2(2x{+}l)^2} {-} R_2(l,p)=\frac{\pi^2}{16}{-} R_2(l,p), &l \text{ odd}\\
	\sum_{x=1}^\infty \frac{(2x{-}1)^2}{(2x{-}1{-}l)^2(2x-1+l)^2} {-} R_2(l,p)=\frac{\pi^2}{16}{-} R_3(l,p), &\text{else}
	\end{cases}	
	\end{align*} see Chapter 23.2 of \citet{abramowitz1964handbook} on sums of reciprocal powers.
	If $p$ is even, then the residual terms are given by
	\begin{align*}
	R_1(l,p)&=\sum_{x=(p-2)/2}^\infty \frac{1}{(2x-1)^2} = \frac{ \phi^{(1)}\{(p-3)/2\}}{4},\\
	R_2(l,p)
	&= \sum_{x=p/2}^\infty  \frac{(2x)^2}{(2x{-}l)^2(2x{+}l)^2}\\
	&=\frac{2\phi\{(p{+}l)/2\}{-}2\phi \{(p{-}l)/2\}{+}
		l [\phi^{(1)}\{(p{+}l)/2\}{+} \phi^{(1)}\{(p{-}l)/2\}]}{16l}, \\
	R_3(l,p)&= \sum_{x=p/2}^{\infty} \frac{(2x-1)^2}{(2x-1-l)^2(2x-1+l)^2}\\
	&=\frac{2\phi\{(p{-}1{+}l)/2\}{-}2\phi \{(p{-}1{-}l)/2\}{+}l [\phi^{(1)}\{(p{-}1{+}l)/2\}{+} \phi^{(1)}\{(p{-}1{-}l)/2\}]}{16l}, 
	\end{align*}
	where $\phi$ and $\phi^{(1)}$ denote the digamma function and its derivative. If $p$ is odd, similar remainder terms can be derived.
	To see that $R_i(l,p)=\mathcal{O}(p^{-1})$ for $i=1,2,3$ and uniformly in $l$ we use that
	asymptotically $\phi(x){\sim} \log(x){-}1/(2x)$ and $\phi^{(1)}(x){\sim} 1/x{+}1/(2x^2).$   
	The mixed term $4/3 \sum_{k=0}^{p-2} \frac{k^2\{1-(-1)^{k-l}\}^2}{(k{-}l)(k{+}l)(2p-2)^2}$ and the squared $\mathcal{O}$ term $16\pi^2/9 \sum_{k=0}^{p-2}\frac{k^2}{(2p-2)^4}$ are both of the order $p^{-1}$.
	Furthermore,
	\begin{align*}
	& \frac{1}{p-1} \sum_{r=1 }^{p-1}  L_{r,j}|r-j| c(r,j)^2= \frac{1}{p-1}  \sum_{k=0}^{p-2}   \frac{k^2\{1{-}(-1)^{k-l}\}^2 L_{k,l}|k-l|
	}{\pi^2(k-l)^2(k+l)^2} =\mathcal{O}\{p^{-1}\log p\},
	\end{align*}
	since the harmonic sum diverges at a rate  of $\log p$. Finally,   $\lambda_{j}= f(x_j)+\mathcal{O}\{p^{-\beta}\log p\}$
	by the uniform approximation properties of the discrete Fourier series for H\"older continuous  functions  \citep[][]{zygmund2002trigonometric}.
	All together, we have shown that $(D \Sigma D )_{j,j}= f(x_j)+ \mathcal{O}\{p^{-1}+ p^{-\beta}\log p\}$, where the $\mathcal{O}$ terms are uniform over $j=1,...,p$.\\
	
	{\centering \textit{Case $i\neq j$ and $|i-j|$ is even}}
	
	In this case, $( D  \Sigma  D )_{i,j} = a_ia_j\sum_{r=1}^{2p-2}  \lambda_r \{b(i,r)b(j,r) + c(i,r)c(j,r)\}$ since $|r-j|,$ and $|r-i|$ are  both either odd or even. Note that $b(i,r)b(j,r)=\mathcal{O}(p^{-2})$ for $r\neq i,j$  and $b(i,r)b(j,r)=\mathcal{O}(p^{-1})$ if $i=r$ or $j=r$, where the $ \mathcal{O}(\cdot)$ terms are uniformly in $i,j$. Since $b(i,r)\geq 0$ and $\|f\|_\infty<M_0$, it follows $a_ia_j\sum_{r=1}^{2p-2}  \lambda_r b(i,r)b(j,r) =\mathcal{O}(p^{-1})$ uniformly in $i,j$.  We proceed similarly as before to show that $a_ia_j\sum_{r=1}^{2p-2}\lambda_r c(i,r)c(j,r)=\mathcal{O}(p^{-1})$. 
	Setting $k=r-1, \,l=j-1, \,m=i-1$ and using that $l\neq m$ and $|l-m|$ is even, one obtains 
	\begin{align*}
	&\frac{1}{\pi^2}\sum_{k=0}^{p-2}   \frac{k^2 \{1-(-1)^{k-l}\}\{(-1)^{k-m}{-}1\}}{ (k{-}l)(k{+}l)(k{-}m)(k{+}m)}\\
	&= \frac{4}{\pi^2}\begin{cases}
	\sum_{x=1}^\infty \frac{(2x)^2}{(2x{-}l)(2x{+}l)(2x{-}m)(2x{+}m)} {-} R_1(l,m,p)={-} R_1(l,m,p), &l,m \text{ odd}\\
	\sum_{x=1}^\infty \frac{(2x{-}1)^2}{(2x{-}1{-}l)(2x{-}1{+}l)(2x{-}1{-}m)(2x{-}1{+}m)} {-} R_2(l,m,p)={-} R_2(l,m,p), &l,m \text{ even,}
	\end{cases}	
	\end{align*}
	where for even $p$ the residual terms are given by 
	\begin{align*}
	R_1(l,m,p)&= \sum_{x=p/2}^\infty\frac{(2x)^2}{(2x{-}l)(2x{+}l)(2x{-}m)(2x{+}m)}\\ 
	R_2(l,m,p)&= \sum_{x=p/2}^\infty\frac{(2x{-}1)^2}{(2x{-}1{-}l)(2x{-}1{+}l)(2x{-}1{-}m)(2x{-}1{+}m)}. 
	\end{align*}
	Similarly as before, one can show that $R_1(l,m,p)$ and $R_2(l,m,p)$ are of the order $\mathcal{O}(p^{-1})$ where the hidden constant does not depend on $l,m$.
	If $p$ is odd, analogous residual terms can be derived. 
	Using similar techniques as before, one can show that the two residual terms and the remaining mixed and square terms vanish at a rate of the order $\mathcal{O}(p^{-1})$ and uniformly in $i,j$.\\
	
	{\centering \textit{Case $i\neq j$ and $|i-j|$ is odd}}
	
	$|r-i|$ and $|r-j|$ are either odd and even, or even and odd. Without loss of generality, assume that $|r-i|$ is even. Then, $( D  \Sigma  D )_{i,j} =a_ia_j\sum_{r=1}^{2p-2}\lambda_r b(i,r)c(j,r)$. Since $b(i,\cdot)$ is an even function, $c(j,\cdot) $ is an odd function and $\lambda_r=\lambda_{2p-r}$, it follows $( D  \Sigma  D )_{i,j} =0$.  \hfill \qed

	\subsection{Periodic Smoothing Splines}\label{app:SSper}
	In this section we recapitulate the results obtained in \citet{schwarz2016unified} for the equivalent kernels of periodic smoothing spline estimators.  Given data points $\{(x_i,Y_i)\}_{i=1}^N$ with true relationship $Y_i=f(x_i)+\epsilon_i$, where $x_i=i/N$ and the residuals $\epsilon_i$ satisfy standard assumptions, the periodic smoothing spline estimator $\hat{f}$ is the solution to
	$$\min_{s\in S_{\text{per}}(2q-1,\underline{x}_N)} \left [ \frac{1}{N} \sum_{i=1}^N \{Y_i-s(x_i)\}^2 + h^{2q} \int_0^1 \{ s^{(q)}(x)\}^2\,\text{d}x\right ],$$ 
	where $h>0$ is the smoothing parameter, $q\in\mathbb{N}$  the penalty order  and $S_{\text{per}}(2q-1,\underline{x}_N)$ the periodic spline space of degree $2q-1$ based on knots $\underline{x}_N{=}(x_i)_{i=0}^N{=}\left (i/N\right)_{i=0}^N$.  
	\citet{schwarz2016unified} have shown that  $\hat{f}(x)=N^{-1}\sum_{i=1}^NW(x,x_i)Y_i$, where $W(x,x_i)$ is known as effective kernel of a periodic smoothing spline estimator and its explicit expression is 
	$$W(x,x_i)=\sum_{j=1}^N \frac{\Phi \{ Nx + q, \exp (-2\pi  \mathtt{i} j/N)\} \overline{\Phi \{ Nx_i + q, \exp (-2\pi  \mathtt{i} j/N)\}}}{Q_{2q-2}(j/N)+h^{2q}(2\pi j)^{2q} \mbox{sinc} (\pi j/N)^{2q}},$$
	where  $\Phi(\cdot,\cdot)$ denotes exponential splines \citep{Schoenberg73} with complex conjugate  $\overline{\Phi(\cdot,\cdot)}$. The polynomials $Q_{2q-2}(x)$ are formally defined as $Q_{2q-2}=\sum_{l=-\infty}^\infty\mbox{sinc}\{\pi(x+l)\}^{2q}$; for the explicit expressions see Section~4.1.1 in \citet{schwarzdiss}. A scaled version $K(x,t)$ of $W(x,t)$ in terms of the bandwidth parameter $h$ is defined as
	$K(x,t)=hW(hx,ht).$ We denote by $K_h(x,t)=K(x/h,t/h)=hW(x,t)$. Also, we use that $W(x,t)=\sum_{l=-\infty}^\infty \mathcal{W}(x,t+l)$ \citep[see Lemma~14 in][]{schwarzdiss}, where 
	$$\mathcal{W}(x,t)=\int_0^N\frac{\Phi \{ Nx + q, \exp (-2\pi  \mathtt{i} u)\} \overline{\Phi \{ Nx_i + q, \exp (-2\pi  \mathtt{i}u)\}}}{Q_{2q-2}(u)+h^{2q}(2\pi j)^{2q} \mbox{sinc} (\pi u)^{2q}}\, \text{d}u$$  
	is the asymptotic equivalent kernel of smoothing spline estimators on $\mathbb{R}$. Finally, we denote $\mathcal{K}_h(x,t)=\mathcal{K}(x/h,t/h)=h\mathcal{W}(x,t).$ 
	
	\subsection{Proof of Theorem~\ref{theorem1}} \label{app:theorem1}
	
	The structure of the proof is as follows. First, we derive the $L_\infty$ rate of the periodic smoothing spline estimator $\widehat{H(f)}$. Then, using the Cauchy-Schwarz inequality and a mean value argument, the convergence rate of the spectral density estimator $\hat{f}$ is established. With the relationship $E\|\widehat{\Sigma}-\Sigma\|\leq E \|\hat{f} - f\|_\infty^2 $ the first claim of the theorem follows. Finally, we prove the second statement on the precision matrices. For the sake of clarity, the proof of the auxiliary Lemmas~\ref{lemma:decay}, ~\ref{lemma:rewriteBins} and ~\ref{lemma:representY} are listed separately in Section~\ref{app:auxlemmas}. 	
	\subsubsection{An upper bound on  $E  \|\widehat{H(f)} - H(f)\|_\infty^2 $}
	\begin{prop}\label{prop:riskHf}
		Let $\Sigma=\Sigma(f)$ with $f \in \mathcal{F}_{\beta}$ such that $\beta>0$. If $h>0$ such that  $h\to 0$ and $hT\to \infty$,  then with $T=\lfloor p^{\upsilon}\rfloor$ for any  $\upsilon\in(1-\min\{1,\beta\}/3,1)$, the estimator $\widehat{H(f)}$ described in Section \ref{sec:method} with $q=\max\{1,\gamma\}$ satisfies for $p\to \infty$ and $n$ such that $p^{\min\{1,\beta\}}/n\to c\in(0,\infty]$ 
		$$\ E \|\widehat{H(f)} - H(f)\|_\infty^2 =\mathcal{O} \left\{ \log(np)/(nph)\right\} + \mathcal{O}(h^{2\beta}).$$
	\end{prop}
	
	\begin{proof}	Application of the triangle inequality yields a bias-variance decomposition
		\begin{align}\label{eq:BiasVar}
		&E\|\widehat{H(f)}- H(f)\|_\infty^2  \leq 2E\|\widehat{H(f)}- E\{\widehat{H(f)}\}\|_\infty^2 + 2 \|E\{\widehat{H(f)}\}- H(f)\|_\infty^2.
		\end{align}
		Set $\tilde{T}=2T-2$ and ${x}_k=(k-1)/\tilde{T}$ for $k=1,...,\tilde{T}.$ The periodic smoothing spline estimator $\widehat{H(f)}$ can be represented as a kernel estimator with respect to a kernel function $W(x,y)$ or the scaled version $K_h(x,t)=hW(x,t)$. An explicit representation of the kernel is derived in  \cite{schwarz2016unified}. A summary of their results is given in the Appendix~\ref{app:SSper}.  Lemma~4 listed in the Section~\ref{app:auxlemmas} gives the following decomposition of $\widehat{H(f)}(x)$ for $x\in[0,1]$ into a deterministic,  a Gaussian, and a non-Gaussian part
		\begin{align*}
		\widehat{H(f)}(x)&= \frac{1}{\tilde{T}} \sum_{k=1}^{\tilde{T}} W(x,x_k) Y^*_k= \frac{1}{\tilde{T}h}  \sum_{k=1}^{\tilde{T}} K_h(x,x_k) \left [H \{f(x_k)\} +\epsilon_k+ \zeta_k+\xi_k\right ], 
		\end{align*}
		where $|\epsilon_k|\lesssim(np)^{-1}+(np)^{-\beta}\log p$, $\zeta_k\sim \mathcal{N}(0,m^{-1})$  with $\mbox{cov}(\zeta_k,\zeta_l)= \mathcal{O} \{ p^{-2} +p^{-2\beta} (\log p)^2\}$ for $k\neq l$. The random variable $\xi_k$  satisfies $E|\xi_k|^\ell \lesssim (\log m )^{2\ell} \{m^{-\ell}+(T^{-1}+T^{-1}p^{1-\beta}\log p)^\ell\}$  for each integer $\ell>1$ and has mean zero. Mirroring and renumerating $\zeta_k, \eta_k,\epsilon_k$ is similar as for $Y_k^*\, (k=1,...,\tilde{T})$.  \\
		
		
		{\centering \textit{Upper bound on the variance}}
		
		Using the above representation, one can write
		\begin{equation} \label{var:firstbound}
		E \|\widehat{H(f)}- E\{\widehat{H(f)}\}\|_\infty^2= E  \left \|\frac{1}{\tilde{T}} \sum_{k=1}^{\tilde{T}} W (x, x_k) (\zeta_k+\xi_k) \right \|^2_\infty.
		\end{equation}
		First, we bound the supremum by a maximum over a finite number of points. If $q>1$,	then $W (\cdot, x_k)$ is Lipschitz continuous with constant $L>0$. In this case, it holds  almost surely that
		\begin{align*}
		&\sup_{x\in[0,1]} \left|\frac{1}{\tilde{T}} \sum_{k=1}^{\tilde{T}}  W (x, x_k)( \zeta_k+\xi_k )\right|^2\\
		&\leq  \left [\max_{1\leqslant j \leqslant \tilde{T}} \left |\frac{1}{\tilde{T}} \sum_{k=1}^{\tilde{T}}  W (x_j, x_k)( \zeta_k+\xi_k ) \right| {+}  \sup_{ \substack {x,x^\prime \in [0,1], \\ |x-x^\prime|<1/\tilde{T} }}\left |\frac{1}{\tilde{T}} \sum_{k=1}^{\tilde{T}}  \{W (x, x_k)-  W (x^\prime, x_k)\} ( \zeta_k+\xi_k )\right |\right ]^2\\
		&\leq 2 \max_{1\leqslant j \leqslant \tilde{T}} \left |\frac{1}{\tilde{T}} \sum_{k=1}^{\tilde{T}}  W (x_j, x_k)( \zeta_k+\xi_k ) \right|^2 +\frac{2L^2}{\tilde{T}^4}\left(\sum_{k=1}^{\tilde{T}} \left| \zeta_k+\xi_k\right|\right )^2 .
		\end{align*}
		Using $E\{(\sum_{k=1}^{\tilde{T}} \left| \zeta_k+\xi_k\right| )^2\}\lesssim E\{(\sum_{k=1}^{\tilde{T}} \left| \zeta_k\right|)^2\}+E\{(\sum_{k=1}^{\tilde{T}} \left| \xi_k\right|)^2\}\lesssim \tilde{T}^{2}m^{-1}$, one gets	
		\begin{align} 
		&(\ref{var:firstbound}) \lesssim E\left \{\max_{1\leqslant j \leqslant \tilde{T}} \left |\frac{1}{\tilde{T}} \sum_{k=1}^{\tilde{T}}  W (x_j, x_k)( \zeta_k+\xi_k ) \right|^2 \right\} +o\{(np)^{-1}\}. \label{ineq:sup2max}
		\end{align}
		
		If $q=1$, then  $\sum_{k=1}^{\tilde{T}} W (\cdot, x_k)$ is a piecewise linear function with knots at $x_j=j/\tilde{T}$. The factor $( \zeta_k+\xi_k )$ can be considered as stochastic weights that do not affect the piecewise linear property. Thus,  the supremum is attained at one of the knots $x_j=j/\tilde{T}$ for $j=1,...,\tilde{T}$, and (\ref{ineq:sup2max}) is also valid for $q=1$. Again with $(a+b)^2\leq 2a^2+2b^2$ we obtain
		\begin{align*} 
		&(\ref{var:firstbound}) \lesssim  E \left \{\max_{1\leqslant j \leqslant \tilde{T}} \left |\frac{1}{\tilde{T}} \sum_{k=1}^{\tilde{T}}  W (x_j, x_k) \zeta_k \right|^2 \right\}{+}E \left \{\max_{1\leqslant j \leqslant \tilde{T}} \left |\frac{1}{\tilde{T}} \sum_{k=1}^{\tilde{T}}  W (x_j, x_k)\xi_k  \right|^2 \right\}{+}o\{(np)^{-1}\}.
		\end{align*}
		We start with bounding $T_1=E\{\max_{1\leq j \leq \tilde{T}}   |(\tilde{T}h)^{-1}\sum_{k=1}^{\tilde{T}}K_h(x_j,x_k) \zeta_k|^2 \}$ with Lemma~1.6 of \citet{tsybakov2009introduction}. This requires a bound on $ \|(\tilde{T}h)^{-1}\sum_{k=1}^{\tilde{T}}K_h(x_j,x_k) \zeta_k\|_{\psi_2}^2$  where $\|\cdot\|_{\psi_2}$ denotes the sub-Gaussian norm. In case of a Gaussian random variable the norm equals to the variance. 
		See \cite{vershynin2018high} for further details on the sub-Gaussian distribution.
		
		Using the properties of the kernel function $K_h$ stated in Lemma~2 of the Supplementary Material, we obtain
		\begin{align*}
		&\left 	\|\frac{{1}}{\tilde{T}h}\sum_{k=1}^{\tilde{T}}K_h(x_j,x_k) \zeta_k\right\|_{\psi_2}^2= \text{var}\left \{\frac{{1}}{\tilde{T}h}\sum_{k=1}^{\tilde{T}}K_h(x_j,x_k) \zeta_k\right\}\\
		&=\frac{{1}}{\tilde{T}^2h^2} \sum_{k=1}^{\tilde{T}}K_h(x_j,x_k)^2\text{var} (\zeta_k)+ \frac{{1}}{\tilde{T}^2h^2} \sum_{k=1}^{\tilde{T}}K_h(x_j,x_k)\sum_{l=1}^{\tilde{T}}K_h(x_j,x_l) \text{cov} (\zeta_k,\zeta_l)\\
		&\leq C(Thm)^{-1}+C^\prime\{ p^{-2} + p^{-2\beta}(\log p)^2\}.
		\end{align*}
		Lemma~1.6 of \citep{tsybakov2009introduction} then yields 
		\begin{align} \label{var:zeta-term}
		T_1\lesssim \log(2\tilde{T})[C(Thm)^{-1}+C^\prime\{ p^{-2} + (\log p)^2p^{-2\beta}\}] =\mathcal{O}\{(Thm)^{-1}\log(2\tilde{T})+ h^{2\beta}\}.
		\end{align}
		
		To see this, note that $p^{\min\{1,\beta\}}/n\to c\in(0,\infty]$  implies $p^{-1}=\mathcal{O}(n^{-1})$. Thus, if $\beta>1$, then $$\log(2\tilde{T})\{ p^{-2} + p^{-2\beta}(\log p)^2\} \lesssim\log(2\tilde{T})p^{-2} \lesssim\log(2\tilde{T})(Tm)^{-1}.$$
		Now consider $0<\beta \leq 1$. Recall that $T=\lfloor p^{\upsilon}\rfloor$ for some fixed $\upsilon\in(1-\min\{1,\beta\}/3,1)$. One can find a constant $a=C_\upsilon$ depending on  $\upsilon$ but  not on $n,p$  such that  the inequality $\log(x)\leq x^a/a$ implies
		$\log(2\tilde{T})\log(p)^2p^{-2\beta}T^{2\beta} =\mathcal{O}(1)$. Thus, $\log(2\tilde{T})\log(p)^2p^{-2\beta}=\mathcal{O}(h^{2\beta})$ since $hT\to \infty$ by assumption.
		
		Next, we derive a bound for the second term $T_2{=}E\{\max_{1\leq j \leq \tilde{T}}  |(\tilde{T}h)^{-1}\sum_{k=1}^{\tilde{T}}K_h(x_j,x_k) \xi_k|^2 \}$. 
		The exponential decay property of the kernel $K_h$, see Lemma 2 in the Section~\ref{app:auxlemmas}, yields 
		\begin{align}
		&T_2\lesssim E \left \{\max_{1\leq j \leq \tilde{T}}  \left | \frac{{1}}{\tilde{T}h}\sum_{k=1}^{\tilde{T}}\gamma_h(x_j,x_k) |\xi_k|\right |^2 \right \}, \nonumber
		\end{align} where $\gamma_h(x,t)=\gamma^{|x-t|/h} + \gamma^{1/h} \{\gamma^{(x-t)/h}+ \gamma^{(t-x)/h}\}(1-\gamma^{1/h})^{-1}$ and $\gamma\in (0,1)$ is a constant. For some threshold $R>0$ specified later, define $\xi_k^-=|\xi_k|\mathbbm{1}{\{|\xi_k|\leq R\}}$ and $\xi_k^+=|\xi_k|\mathbbm{1}{\{|\xi_k|>R\}}$.
		Then, 
		\begin{align} \label{var:Tequation}
		T_2\lesssim E\left\{ \max_{1\leq j \leq \tilde{T}} \left |\frac{1}{\tilde{T}h} \sum_{k=1}^{\tilde{T}} \gamma_h(x_j,x_k)  \xi_k^- \right|^2\right\}+E\left \{\max_{1\leq j \leq \tilde{T}} \left|\frac{1}{\tilde{T}h} \sum_{k=1}^{\tilde{T}}\gamma_h(x_j,x_k) \xi_k^+ \right |^2\right \}.
		\end{align}
		
		The first term in (\ref{var:Tequation}) can be bounded again with Lemma~1.6 of \citet{tsybakov2009introduction}. 
		We use the fact that for not necessarily independent random variables $X_1,...,X_N$ with $|X_i|<R \, (i=1,...,N)$ it holds $  \| \sum_{i=1}^N a_i X_i  \|^2_{\psi_2} \leq 4R^2\sum_{i=1}^N a_i^2$ where  $a_1,...,a_N\in\mathbb{R}$ and $R>0$ are constants. 
		This is a consequence of Lemma~1 of \citet{azuma1967weighted} which yields 
		$E\{\exp (\lambda\sum_{i=1}^{N} a_i X_i)  \}\leq \exp ( \lambda^2R^2 \sum_{i=1}^N a_i^2 )$ for all $\lambda\in\mathbb{R}$. 
		Thus, for all $t>0$ holds $ \text{pr} (|\sum_{i=1}^{N}a_i X_i | >t ) \leq 2\exp (-\lambda t + \lambda^2 R^2 \sum_{i=1}^{N}a_i^2).$
		Setting $\lambda= t/2\cdot (R^2 \sum_{i=1}^{N} a_i^2)^{-1}$, it follows that $\sum_{i=1}^{N}a_i  X_i$ has a sub-Gaussian distribution and the sub-Gaussian norm is bounded by $2R(\sum_{i=1}^N a_i^2)^{1/2}$. 
		
		Together, we get $\| (\tilde{T}h)^{-1} \sum_{k=1}^{\tilde{T}} \gamma_h(x_j,x_k)  \xi_k^-\|^2_{\psi_2}\leq  4R^2(\tilde{T}h)^{-2}  \sum_{k=1}^{\tilde{T}} \gamma_h(x_j,x_k)^2 \lesssim R^2  (\tilde{T}h)^{-1},$ where we used the bound on $\gamma_h$ from Lemma~2. 
		Applying Lemma~1.6 of \citet{tsybakov2009introduction} then yields
		\begin{equation} \label{var:xi-term-neg} 
		E\left \{\max_{1\leq j \leq \tilde{T}} \left | \frac{1}{\tilde{T}h}\sum_{k=1}^{\tilde{T}} \gamma_h(x_j,x_k)  \xi_k^- \right |^2 \right\}\lesssim\frac{ R^2}{\tilde{T}h} \log(2\tilde{T}).
		\end{equation}
		
		To bound the second term in (\ref{var:Tequation}), we use the moment bounds for $\xi_k$ derived in Lemma~4. Then, for all integers $\ell>1$ 
		\begin{align} \label{var:xi-term-pos}
		&E\left \{\max_{1\leq j \leq \tilde{T}} \left | \frac{1}{\tilde{T}h}\sum_{k=1}^{\tilde{T}}  \gamma_h(x_j,x_k)  \xi_k^+  \right|^2 \right\} =E \left \{ \max_{1\leq j \leq \tilde{T}} \left |\frac{1}{\tilde{T}h}\sum_{k=1}^{\tilde{T}} \gamma_h(x_j,x_k)    (\xi_k^+)^{\ell} (\xi_k^+)^{1-\ell} \right |^2 \right\}\nonumber\\
		& \leq R^{2-2\ell} E \left \{ \max_{1\leq j \leq \tilde{T}}  \frac{1}{\tilde{T}^2h^2}\sum_{k=1}^{\tilde{T}}  \gamma_h(x_j,x_k)^2 \cdot \sum_{k=1}^{\tilde{T}}   (\xi_k^+)^{2\ell} \right\}  \leq R^{2-2\ell}    E \left ( \frac{C_1}{\tilde{T}h} \sum_{k=1}^{\tilde{T}}   \xi_k^{2\ell} \right) \nonumber\\ 
		& \lesssim h^{-1} R^{2-2\ell}  (\log m)^{4\ell} \{m^{-2\ell} + (T^{-1}+T^{-1}p^{1-\beta}\log p)^{2\ell}\}.
		\end{align}
		
		Combining the error bounds (\ref{var:xi-term-neg}) and (\ref{var:xi-term-pos}) and choosing $R=m^{-1/2}$ gives
		\begin{align*} 
		&T_2\lesssim \frac{\log T}{mTh}+\frac{(\log m)^{4\ell}}{h m^{1-\ell}} \cdot
		\begin{cases}
		\frac{(\log p)^{2\ell}p^{2\ell(1-\beta)}}{T^{2\ell}} +m^{-2\ell},& 0<\beta\leq 1, \\
		T^{-2\ell}+m^{-2\ell},& 1<\beta.
		\end{cases}
		\end{align*} 
		
		By assumption $T=\lfloor p^{\upsilon}\rfloor$ and $m=\lfloor np^{(1-\upsilon)}\rfloor $ for some fixed $\upsilon\in(1-\min\{1,\beta\}/3,1)$.  If $\ell$ is an integer such that $\ell \geq 1/(1-\upsilon)$, then
		\begin{align*}
		m^{\ell-1} (\log m)^{4\ell} m^{-2\ell}\lesssim m^{-l}\lesssim n^{-\ell}p^{-(1-\upsilon)\ell} =\mathcal{O}\{(np)^{-1}\},
		\end{align*} where we used $\log x\leq x^a/a$ with $a=1/(4\ell)$. 
		Now, consider $0<\beta\leq 1$. Then, $n\leq cp^{\beta}$ by assumption. Let $0<\chi<1$ be a small constant. Using the inequality $\log x\leq x^a/a$ twice with $a=\chi/(2\ell)$ gives
		\begin{align}
		&m^{\ell-1} (\log m)^{4\ell} (\log p)^{2\ell}p^{2\ell(1-\beta)}T^{-2\ell}\lesssim n^{\ell-1+\chi}p^{(1-\upsilon)(\ell-1+\chi)-2\ell\upsilon+2\ell(1-\beta)+\chi} \nonumber\\
		&\lesssim p^ {\ell (3-3\upsilon-\beta) +(\beta+\chi+1)(1-v)+\chi}. \label{condk}
		\end{align}
		Note that for $\upsilon{\in}(1-\beta/3,1)$, the condition
		$\ell (3-3\upsilon-\beta) +(\beta+\chi+1)(1-v)+\chi  <-2$ is equivalent to $	\ell>\{-2-(\beta+\chi+1)(1-v)-\chi\}/(3-3\upsilon-\beta)$. Thus, for any fixed $\upsilon$  one can find an integer $\ell$ which is independent of $n,p$ such that the right side of (\ref {condk}) is of the order $\mathcal{O}(p^{-2}) $. For the same choice of $\ell$ it holds that  $m^{\ell-1} (\log m)^{4\ell} T^{-2\ell}= \mathcal{O}(p^{-2}).$
		
		In total, choosing an integer $\ell\geq \max[ 1/(1-\upsilon),\{-2-(\beta+\chi+1)(1-v)-\chi\}/(3-3\upsilon-\beta)]$ gives
		\begin{align} \label{var:finalbound}
		\mathbb{E} \|\widehat{H(f)}- \mathbb{E}\{\widehat{H(f)}\}\|_\infty^2= \mathcal{O}\left\{\frac{\log(np)}{nph} +h^{2\beta}\right \}. 
		\end{align}

		{\centering \textit{Upper bound on the bias}}
		
		Using the representation in Lemma~4 once more gives for each $x\in[0,1]$
		$$E\{\widehat{H(f)}(x)\}- H\{f(x)\}= \frac{1}{\tilde{T}h} \sum_{k=1}^{\tilde{T}} K_h(x,x_k) \left [H\{f(x_k)\} +\epsilon_k\right ] -  H\{f(x)\}.$$
		The bounds on $\epsilon_k$ imply
		\begin{align*} 
		\left |\frac{1}{\tilde{T}h} \sum_{k=1}^{\tilde{T}}K_h(x,x_k)\epsilon_k \right | &\lesssim \frac{1}{\tilde{T}h} \sum_{k=1}^{\tilde{T}} \gamma_h(x_j,x_k) |\epsilon_k|\lesssim (np)^{-1}+(np)^{-\beta}\log p.
		\end{align*}
		Consider the case that $\beta\geq 1$. In particular, $q=\gamma$ and $f^{(q)}$ is $\alpha$-H\"older continuous. Since $f$ is a periodic function with $f(x)\in[\delta, M_0]$ and $H(y)\propto \phi(m/2) + \log \left ( 2y/m\right )$, it follows that $\{H(f)\}^{(q)}$ is also $\alpha$-H\"older continuous.
		Extending $g=H(f)$ to the entire real line, we get 
		$$ \frac{1}{\tilde{T}h} \sum_{k=1}^{\tilde{T}} K_h(x,x_k)g(x_k)=  \int_{-\infty}^\infty h^{-1}\mathcal{K}_h(x,t)g(t) \, \mbox{d}t +\mathcal{O}(\tilde{T}^{-\beta})$$
		where  $\mathcal{K}_h$ is the extension of $K_h$ to the entire real line \citep[see][]{schwarz2016unified}.
		Expanding $g(t)$ in a Taylor series around $x$ and using that $h^{-1}\mathcal{K}_h$ is a kernel of order $2q$, see Lemma~2$(iii)$, it follows that for any $x\in[0,1]$ 
		\begin{align*}
		&\frac{1}{\tilde{T}h} \sum_{k=1}^{\tilde{T}} K_h(x,x_k)g(x_k)=g(x) + \int_{-\infty}^\infty h^{-1}\mathcal{K}_h(x,t) {(x-t)^{q} \frac{g^{(q)}(\xi_{x,t})}{q!}}\, \mbox{d}t\ +\mathcal{O}(\tilde{T}^{-\beta})\\
		&= g(x) +  \int_{-\infty}^\infty h^{-1}\mathcal{K}_h(x,t) {(x-t)^{q} \frac{g^{(q)}(\xi_{x,t}) - g^{(q)}(x)}{q!}}\, \mbox{d}t  + \mathcal{O}(\tilde{T}^{-\beta})\\
		&= g(x) + \sum_{l=-\infty}^\infty \int_{x+(l-1)h}^{x+lh} \mathcal{K}_h(x, t)(x-t) ^{q} \frac{g^{(q)}(\xi_{x,t})-g^{(q)}(x)}{hq!}\, \mbox{d}t+ \mathcal{O}(\tilde{T}^{-\beta}),
		\end{align*}
		where $\xi_{x,t}$ is a point between $x$ and $t$. Using the fact that the kernel $\mathcal{K}_h$ decays exponentially and that $g^{(q)}$ is $\alpha$-H\"older continuous on $[\delta,M_0]$ with some constant $L$, one obtains	
		\begin{align*}
		&\left | \frac{1}{\tilde{T}h} \sum_{k=1}^{\tilde{T}} K_h(x,x_k)g(x_k)-g(x) \right|
		{\leq}  \frac{CL}{q!} \sum_{l=-\infty}^\infty \int_{x+(l-1)h}^{x+lh}\gamma^\frac{|x-t|}{h} |x-t|^{q} \frac{|\xi_{x,t}-x|^{\alpha}}{hq!}\, \mbox{d}t{+} \mathcal{O}(\tilde{T}^{-\beta})\\
		&\leq  \frac{CL}{q!} \sum_{l=-\infty}^\infty \int_{x+(l-1)h}^{x+lh} \gamma^\frac{|x-t|}{h} \frac{|x-t|^{\beta} }{h}\, \mbox{d}t+ \mathcal{O}(\tilde{T}^{-\beta})
		\leq h^{\beta} \frac{CL}{q!} \sum_{l=-\infty}^\infty \gamma^{|l-1|} |l|^{\beta} + \mathcal{O}(\tilde{T}^{-\beta})\\
		&= \mathcal{O}(h^{\beta}) + \mathcal{O}(\tilde{T}^{-\beta}).
		\end{align*} 
		If $0<\beta\leq 1$, then $q=1$ and $g$ is $\beta$-H\"older continuous. Since $f(x)\in[\delta, M_0]$ and the logarithm is Lipschitz continuous on a compact interval, it follows $g=H(f)$ is $\beta$-H\"older continuous. Expanding $g$ to the entire line and using Lemma~2$(iii)$ with $m=0$ gives
		\begin{align*}
		\frac{1}{\tilde{T}h} \sum_{k=1}^{\tilde{T}} K_h(x,x_k)g(x_k)-g(x)=  \int_{-\infty}^\infty \mathcal{K}_h(x,t)\frac{g(t)-g(x)}{h} \, \mbox{d}t +\mathcal{O}(\tilde{T}^{-\beta}).
		\end{align*}
		In a similar way as before, one obtains
		$$ \left | \frac{1}{\tilde{T}h} \sum_{k=1}^{\tilde{T}} K_h(x,x_k)g(x_k)-g(x) \right| = \mathcal{O}(h^{\beta}) + \mathcal{O}(\tilde{T}^{-\beta}).$$
		Note that $\tilde{T}^{-\beta}=o(h^{\beta})$ as  $hT\to \infty$ and $h\to 0$ by assumption. 
		Since the derived bounds are uniform for $x\in[0,1]$ it holds
		\begin{align} \label{bias:finalbound} 
		\left \|E\{\widehat{H(f)}(x)\}-H\{f(x)\}\right\|_{\infty}=\mathcal{O}(h^{\beta}) +\mathcal{O}\{(np)^{-1}+(np)^{-\beta} \log p\}.
		\end{align}
		Putting the bounds (\ref{var:finalbound}) and (\ref{bias:finalbound}) together gives
		\begin{flalign*} 
		&& E [\|\widehat{H(f)} - H(f)\|_\infty^2 ]=\mathcal{O} \left\{ \log(np)/(nph)\right\} + \mathcal{O}(h^{2\beta}). && 
		\end{flalign*}	
	\end{proof}
	\subsubsection{An upper bound on $E  \|\hat{f} - f\|_\infty^2 $}	\label{A:boundf} 		
		\begin{prop}\label{prop:riskf}
			Let $\Sigma=\Sigma(f)$ with $f\in \mathcal{F}_{\beta}$ such that $\beta>0$.  If $h>0$ such that  $h\to 0$ and $hT\to \infty$,  then with $T=\lfloor p^{\upsilon}\rfloor$ for any $\upsilon\in(1-\min\{1,\beta\}/3,1)$, the estimator $\hat{f}$ described in Section 3 with $q=\max\{1,\gamma\}$ satisfies for $p\to \infty$ and $n$ such that $p^{\min\{1,\beta\}}/n\to c\in(0,\infty]$  	
			$$\ E \|\hat{f} - f\|_\infty^2 =\mathcal{O} \left\{ \log(np)/(nph)\right\} + \mathcal{O}(h^{2\beta}).$$
		\end{prop}
	\begin{proof}
	By the mean value theorem, it holds for some function  $g$ between $\widehat{H(f)}$ and $H(f)$ that
	\begin{align*}
	E  \|\hat{f} - f\|_\infty^2 &= E\|H^{-1} \{\widehat{H(f)}\} - H^{-1} \{H(f)\}   \|_\infty^2  = E  \|(H^{-1})^\prime (g) ( \widehat{H(f)}-H(f) ) \|_\infty^2 \\
	&\lesssim E\|  H^{-1}(g) \{ \widehat{H(f)}-H(f) \}\|_\infty^2.
	\end{align*} Application of the Cauchy-Schwarz inequality with a constant $C>0$ yields
	\begin{equation} \label{CSsplit}
	E\|\hat{f}-f \|_\infty^2\leq  C^2 E \|\widehat{H(f)}-H(f)\|_\infty^2
	+ (E\|\hat{f}-f\|_\infty^4)^{1/2}\cdot \mbox{pr} (\|H^{-1}(g)\|_\infty > C) ^{1/2}.
	\end{equation}
	To show that the second term on the right-hand side of (\ref{CSsplit}) is negligible, we use  Markov's inequality. In order to do so, we first derive the asymptotic order of the moment generating function of $\widehat{H(f)}$ for $n,p\to \infty$. \\ 
	
	{\centering \textit{Moment generating function of $\|\widehat{H(f)} \|_\infty$} }
	
	By the exponential decay property of the kernel $K$ stated in Lemma~\ref{lemma:decay}, it holds almost surely that	
	\begin{align}
	\|\widehat{H(f)} \|_\infty &\lesssim \sup_{x\in[0,1]}  \frac{1}{\tilde{T}h} \sum_{k=1}^{\tilde{T}} \gamma_h(x,x_k)\cdot \left | Y^*_k\right |, \nonumber
	\end{align} where $\gamma_h(x,t)=\gamma^{|x-t|/h} + \gamma^{1/h} \{\gamma^{(x-t)/h}+ \gamma^{(t-x)/h}\}/(1-\gamma^{1/h})$ for some constant $\gamma\in (0,1)$. 
	First,  $\|\widehat{H(f)} \|_\infty$ is bounded with the maximum over a finite number of points.
	Calculating the derivative of $s:[0,1]\to \mathbb{R},\, x\mapsto s(x)=\sum_{k=1}^{\tilde{T}} \gamma_h(x,x_k) \left | Y^*_k\right |$ for $x\neq x_k$ gives almost surely
	\begin{align*}
	&\frac{\partial}{\partial x}  s(x) =  h^{-1} \log \gamma \sum_{k=1}^{\tilde{T}}  \left [\gamma^{|x-x_k|/h}|x-x_k| + \frac{\gamma^{1/h}}{1-\gamma^{1/h}}\left \{\gamma^{(x-x_k)/h}+ \gamma^{(x_k-x)/h} \right\} \right ] \left |Y^*_k\right|.
	\end{align*}
	Since $\partial s(x)/\partial x >0$ almost surely for $x\neq x_k$, the extrema occur at $x_k \, (k=1,...,\tilde{T}).$
	Thus, for $\lambda >0,$ the moment generating function of $\| \widehat{H(f)} \|_\infty $ is bounded, up to constants, by
	\begin{align*} 
	&E [\exp\{\lambda \| \widehat{H(f)} \|_\infty  \} ]  \lesssim \sum_{j=1}^{\tilde{T}}E\left [ \exp \left\{\frac{\lambda}{\tilde{T}h}\sum_{k=1}^{\tilde{T}} \gamma_h(x_j,x_k)\left | Y_k^* \right | \right\}\right ].
	\end{align*}
	Let $M_j=(\tilde{T}h)^{-1}\sum_{k=1}^{\tilde{T}} \gamma_h(x_j,x_k)$, which by Lemma~\ref{lemma:decay} is bounded uniformly in $j$ by some global constant $M>0$. 
	By the convexity of the exponential function we obtain
	\begin{align*} 
	&E\left [  \exp \left \{\frac{\lambda}{\tilde{T}h}\sum_{k=1}^{\tilde{T}} \gamma_h(x_j,x_k)| Y_k^* | \right \}\right ] \leq \frac{1}{M_j\tilde{T}h} \sum_{k=1}^{\tilde{T}} \gamma_h(x_j,x_k) E\left [\exp \left\{ \lambda M_j|Y_k^*| \right\} \right ].
	\end{align*}	
	Recall that $Y_k^*=2^{-1/2}\log(Q_k/m)$ and  by assumption $0\leq \delta \leq f \leq M_0$. Using Lemma~\ref{lemma:rewriteBins}, $Q_k$ can be written as a sum of $m=np/T$ independent gamma random variables, i.e., $Q_k= \sum_{i=1}^{m} A_i$ where $A_i\sim \Gamma(1/2, 2a_i)$ with $a_i\in(c_1\delta,c_2M_0)$ for some constants $c_1,\,c_2>0$. Thus, we can find some (fully correlated) random variables $Q_k^-\sim \Gamma(m/2,\tilde{c}_1\delta)$ and $Q_k^+\sim \Gamma(m/2,\tilde{c}_2M_0)$ such that $Q_k^-\leq Q_k\leq Q_k^+$ almost surely and $\tilde{c}_1,\tilde{c}_2>0$ are some constants. 
	Then, 
	\begin{align*}
	E [\exp \{ \lambda M_j |Y_k^*|\} ]&\leq  E\left [  \exp \left \{ \lambda M_j2^{-1/2} \max \left\{ | \log ( Q_k^+/m )|,  | \log (Q^-_k/m)| \right\}   \right\} \right ].
	\end{align*} 
	The moment generating function of $|\log X|$ with $X\sim\Gamma(a,b)$ is   
	\begin{align*}
	E\left \{  \exp ( t \left |\log X \right | ) \right \} &=b^{-t}\frac{  \Gamma(a-t) - \gamma (a - t, 1/b) + b^{2t} \gamma(a+t, 1/b)}{\Gamma(a)}, \quad |t|<a,
	\end{align*}
	where $\Gamma(a)$ is the gamma function and $\gamma(a,b)$ is the lower incomplete gamma function. 
	In particular, 
	\begin{equation} \label{MGFbound}
	E\{  \exp ( t  |\log X | )\}\leq\frac{  b^{-t} \Gamma(a-t)+ b^{t} \Gamma(a+t)}{\Gamma(a)}.
	\end{equation}  for $0<t<a.$
	In our setting, $X=Q_k^-/m\sim \Gamma(m/2, \tilde{c}_1\delta/m)$ or $X=Q_k^+/m\sim \Gamma(m/2, \tilde{c}_2M_0/m)$, respectively. To derive the asymptotic order of  $E [ \exp\{ \lambda \|\widehat{H(f)} \|_\infty \}]$ for $p/n\to c\in(0, \infty]$,  where  $\lambda>0$, we first establish the asymptotic order of the ratio $\Gamma(a + t)/\Gamma(a)$ for $a\to \infty$. 
	We distinguish the two cases where $t$ is independent of $a$ and where $t$ linearly depends on $a$. Based on Stirling`s approximation, one derives for $a>1$ and $ t>0$ 
	\begin{align}
	\frac{\Gamma(a+ t)}{\Gamma(a)}&=\frac{ \{2\pi (a+ t)^{-1}\}^{1/2}(a+ t)^{a+ t}\exp{(-a-t)} \left [ 1+ \mathcal{O}\{(a+ t)^{-1}\} \right] }{ (2\pi/
		a)^{1/2} a^{a}\exp{(-a)} \left \{ 1+ \mathcal{O}(a^{-1}) \right\} } \nonumber \\
	&=\{a (a+ t)^{-1}\}^{1/2} \left \{(a+ t)a^{-1}\right\}^{a} (a+ t)^{t} \exp{(- t)} \left [ \frac{ 1+ \mathcal{O}\{(a+ t)^{-1}\}}{  1+ \mathcal{O}(a^{-1})}  \right]. \label{gammaratio}
	\end{align}
	Thus, for $0<t<a$ and $t$ independent of $a$, Equation (\ref{gammaratio}) implies for $a\to \infty$ that  $\Gamma(a+ t)/\Gamma(a) =\mathcal{O}(a^{t})$. Similarly, it can be seen that  $\Gamma(a-t)/\Gamma(a) =\mathcal{O}(a^{-t})$. 
	If $0<t<a$ and $t$ linearly depends on $a$, i.e., $t=ca$ for some constant $c\in (0,1)$, then we get $\Gamma(a\pm t)/\Gamma(a) =\mathcal{O}\{a^{\pm t}  \exp(a) \}$ for $a\to \infty$.	
	Hence, for a fixed $\lambda$ not depending on $n,p$  and such that $0<\lambda<m2^{-1/2}M_j^{-1}$ we get for sufficiently large $n,p$
	\begin{align*}
	E\left [\exp \left\{\lambda M_j|Y_k^*|\right \}\right] &=E\left [\exp \left\{\lambda M_j2^{-1/2} \left| \log \left (Q_k/m\right ) \right| \right\} \right ] \\
	&\leq c_3\sum_{b\in\{\tilde{c}_1\delta/m,\tilde{c}_2M_0/m\}}(bm/2)^{-\lambda M_j2^{-1/2} }  +(bm/2)^{\lambda M_j2^{-1/2}}=\mathcal{O}(1).
	\end{align*} 
	If $\lambda=cm$ such that $0<\lambda<m2^{-1/2}M_j^{-1}$, then for sufficiently large  $n,p$
	\begin{align*}
	E\left [\exp \left\{\lambda M_j|Y_k^*|\right \}\right]&\leq c_4\exp(m/2){\sum_{b\in\{\tilde{c}_1\delta/m,\tilde{c}_2M_0/m\}}}(bm/2)^{-\lambda M_j2^{-1/2} } +(bm/2)^{\lambda M_j2^{-1/2}}\\
	& =\mathcal{O}(L^m),
	\end{align*} for some constant  $L>\exp(1/2)$.
	Set $K=\min_{j=1,...,\tilde{T}} 2^{-1/2} M_j^{-1}$ which is a constant independent of $n,p$. Altogether, we showed that for $0<\lambda <Km$ and $n,p\to \infty$
	\begin{align}
	E [ \exp\{ \lambda \|\widehat{H(f)} \|_\infty \}]=\mathcal{O}(T), & & & \text{if } \lambda \text{ does not depend on $n,p$,} \label{MGF-o1}\\
	E [ \exp\{ \lambda \|\widehat{H(f)} \|_\infty \}]=\mathcal{O}(TL^{m}), & & &\text{if }  \lambda=cm \text{ for } c\in(0,K)  \label{MGF-o2}.
	\end{align}
	
	{\centering \textit{Bounding the right hand side of (\ref{CSsplit})}}
	
	Noting that $\|f\|_\infty\leq M_0$ and $m/2 \exp \left \{\phi\left (-m/2\right) \right\}\in[1,4]$ for $m\geq1$, (\ref{MGF-o1}) implies for some constants $c_0,c_1>0$ and $n,p \to \infty$
	\begin{align*}
	E \|\widehat{f} - f\|_\infty^4 &\leq c_0E \|\widehat{f}\|_\infty^4  +c_1 =  c_0E \|  (m/2)\exp \{ 2^{1/2}  \widehat{H(f)}-\phi(m/2)  \}  \|_\infty^4 +c_1 \\
	&\leq c_0\left |(m/2)^4 \exp \{ -4\phi(m/2)\}\right | \cdot E[ \exp \{4\cdot 2^{1/2}   \|\widehat{H(f)} \|_\infty \}]+c_1  =\mathcal{O}(T).
	\end{align*} 
	Since $g$ lies between $\widehat{H(f)}$ and $H(f)$, and $H^{-1}\propto \exp$,  it follows  $H^{-1}(g)\leq H^{-1}\{\widehat{H(f)}\}+f$ almost surely pointwise. Thus, for $C>\|f\|_\infty=M_0$, it holds
	\begin{align*} 
	\mbox{pr}  \left\{ \|H^{-1}(g) \|_\infty > C \right\} \leq \mbox{pr} \left [ \|H^{-1}\{\widehat{H(f)}\}\|_\infty >C-M_0 \right]\leq \mbox{pr} \left[\|\widehat{H(f)}\|_\infty>c_1 \right],
	\end{align*} where $c_1=H(C-M_0)$.
	Applying  Markov's inequality for $t=cm$ with $c\in(0,K)$ and $C=2L^{4/c}+M_0$, where $c,K,L$ are the constants in (\ref{MGF-o2}), gives 
	\begin{align*} 
	\mbox{pr}  \{ \|H^{-1}(g) \|_\infty > C  \} &\leq  \exp(-c_1t)  E  [ \exp\{ t \|\widehat{H(f)}\|_\infty \}  ]=\mathcal{O}(TL^{-m}).
	\end{align*}  Thus, the right-hand side of (\ref{CSsplit}) is asymptotically negligible to the first term in (\ref{CSsplit}).
	Together with Proposition~\ref{prop:riskHf} it follows $$E  \|\hat{f} - f\|_\infty^2 =\mathcal{O} \left\{ \log(np)/(nph)\right\} + \mathcal{O}(h^{2\beta}).$$ 
	\end{proof}

	\subsubsection{An upper bound on $E\|\widehat{\Sigma}-\Sigma\|^2$}	
	Using the fact that the spectral norm of a Toeplitz matrix is upper bounded by the supremum norm of its spectral density, and that Proposition~\ref{prop:riskf} holds for every $ f\in \mathcal{F}_{\beta}$, we get
	\begin{equation} \label{theo:sigmaf} 
	\sup _{f\in \mathcal{F}_{\beta}}\,	E_f \|\widehat{\Sigma}- \Sigma(f)\|^2 \leq \sup _{f\in\mathcal{F}_{\beta}}\, E_f \|\hat{f} - f\|_\infty^2= 	\mathcal{O} \left\{ \log(np)/(nph)\right\} +\mathcal{O}(h^{2\beta}).
	\end{equation} 	
	
 \subsubsection{An upper bound on $E\|\hat{\Omega}-\Sigma^{-1}\|^2$}
	According to the mean value theorem, for a function  $g$ between $\widehat{H(f)}$ and $H(f)$, it holds that
\begin{align*}
E \|\hat{{\Omega}} - {\Omega}\|^2 &\leq E  \|1/\hat{f}- 1/f\|_\infty^2= E \|1/H^{-1} \{\widehat{H(f)}\} - 1/H^{-1} \{H(f)\}\|_\infty^2\\
&= E \|(1/H^{-1})^\prime (g) \{ \widehat{H(f)}-H(f) \}\|_\infty^2 \lesssim E  \| 1/H^{-1}(g) \{ \widehat{H(f)}-H(f)\} \|_\infty^2.
\end{align*} 
Application of the Cauchy-Schwarz inequality with a constant $C>0$ yields
\begin{equation*} 
E \|1/\hat{f}- 1/f\|_\infty^2\leq  C^2 E \|\widehat{H(f)}-H(f)\|_\infty^2 
+ (E \|1/\hat{f}- 1/f\|_\infty^4)^{1/2} \mbox{pr}  \left [ \|1/\{H^{-1}(g)\}\|_\infty > C \right ] ^{1/2}.
\end{equation*} 
Note that $\|1/H^{-1}(g) \|_\infty  \leq  2/m \exp ( 2^{1/2} \|g \|_\infty)\exp\{\phi (m/2)\}\leq c_1\|H^{-1}(g) \|_\infty$ for some constant $c_1>0$ not depending on $n,p$.
Choosing the same constant $C$ as in the proof of Proposition 2 it follows
\begin{align*}
\mbox{pr} \left [ \left\|1/\{H^{-1}(g)\}\right \|_\infty > {C} \right ]  = \mathcal{O}(T L^{-m}).
\end{align*}
Noting that $\|1/f\|_\infty\leq 1/\delta$ and $2/m \exp \left \{ \phi(m/2)\right \} \in[0.25,1]$ for $m\geq1$, (\ref{MGF-o1}) implies for some constants $c_2,c_3>0$ and $n,p \to \infty$
\begin{align*}
E \|1/\hat{f}- 1/f\|_\infty^4
&\leq c_2|(2/m)^4 \exp \{ 4\phi(m/2)\} | \cdot E   \| \exp  \{-4 2^{1/2} \widehat{H(f)}  \}  \|_\infty  +c_3 \\
&\leq  c_2|(2/m)^4 \exp \{ 4\phi(m/2)\} | \cdot E   \exp  \{4 2^{1/2}\| \widehat{H(f)}\|_\infty  \}    +c_3 =\mathcal{O}(T). 
\end{align*} 
Since the derived bounds hold for each $\Sigma(f)$ with $f\in\mathcal{F}_\beta$, we get all together 
\begin{equation} \label{theo:omega}
\sup _{ f\in\mathcal{F}_{\beta}}\,	E_f  [\|\widehat{\Omega} - {\Sigma}^{-1}(f)\|^2  ]=\mathcal{O} \left\{ \log(np)/(nph))\right\} + \mathcal{O}(h^{2\beta}).
\end{equation}

	\subsubsection{Optimal bandwith parameter $h$}

Minimizing the right side of (\ref{theo:sigmaf}) for the bandwidth parameter $h$ yields $h\asymp\left\{\log(np)/(np)\right \}^{\frac{1}{2\beta+1}}$.  In particular, $h\to 0$. Furthermore, $$hT\geq c p^{1-\min\{\beta,1\}/3-(1+\min\{1,\beta\})/(2\beta+1)}  \{\log(np)\} ^{1/(2\beta+1)}\to \infty $$ since by assumption $n\lesssim p^{\min\{1,\beta\}}$ and $T=\lfloor p^\upsilon \rfloor$ with $\upsilon\in(1-\min\{1,\beta\}/3,1)$. 
Thus, substituting $h\asymp\left\{\log(np)/(np)\right \}^{1/(2\beta+1)}$ into (\ref{theo:sigmaf}) gives the second result.\hfill \qed

	\section{Proofs of Auxiliary Lemmas for Theorem~1} \label{app:auxlemmas}
	This section states three technical lemmas needed for the proof of Theorem~\ref{theorem1}. The first lemma lists some properties of the kernel $K_h$ and its extension $\mathcal{K}_h$ on the real line. The proof is based on \citet{schwarz2016unified} and \cite{schwarzdiss}.
	\begin{lemma} \label{lemma:decay} Let $h>0$ be the bandwith parameter depending on $N$.
		\begin{enumerate}[label=(\roman*)]
			\item There are constants $0<C<\infty$ and $0<\gamma <1$ such that for all  $x,t\in [0,1]$
			\begin{align*}
			&|\mathcal{K}_h(x,t)|< C\gamma^{|x-t|/h},\\ 
			&|K_h(x,t)|< C  [\gamma^{|x-t|/h} + \gamma^{1/h}\{\gamma^{(x-t)/h}+ \gamma^{(t-x)/h} \} /(1-\gamma^{1/h})].
			\end{align*}
			\item  Let $t_i=(i-1)/(N-1),\,i=1,...,N$. If $Nh\to\infty$ for $N\to\infty$, then 
			$$\frac{1}{Nh}\sum_{i=1}^N \gamma_h(x,t_i)=\mathcal{O}(1), \quad \frac{1}{Nh}\sum_{i=1}^N \gamma^2_h(x,t_i) = \mathcal{O}(1)$$  uniformly for $x\in[0,1]$, where $\gamma_h(x)= \gamma^{|x-t|/h} + \gamma^{1/h}\{\gamma^{(x-t)/h}+ \gamma^{(t-x)/h} \} /(1-\gamma^{1/h})$.
			\item It holds
			\begin{align*}
			&h^{-1}\int_{-\infty}^\infty (x-t)^m \mathcal{K}_h(x,t)\, \text{d}t = \delta_{m,0},& &\text{ for } m=0,...,2q-1,\\
			&h^{-1}\int_{-\infty}^\infty (x-t)^m \mathcal{K}_h(x,t)\, \text{d}t\neq0, &&\text{ for } m=2q.			
			\end{align*}
		\end{enumerate}
	\end{lemma}
	\begin{proof}	 
		$(i)$ See Lemma~16  of \citet{schwarzdiss} for the proof of the first statement. Furthermore, for $x,t\in [0,1]$ holds
		\begin{align*}
		\left |{K}_h(x,t)\right |&\leq\sum_{l=-\infty} ^\infty \left |\mathcal{K}_h (x, t +l) \right |\\
		&\leq C \sum_{l=-\infty} ^\infty \gamma^{|x-t-l|/h}=C\left \{ \gamma^{|x-t|/h} + \sum_{l>0} ^\infty \gamma^{(x-t+l)/h} +\sum_{l>0} ^\infty \gamma^{(-x+t+l)/h} \right \}\\
		& = C [\gamma^{|x-t|/h} + \gamma^{1/h}\{\gamma^{(x-t)/h}+ \gamma^{(t-x)/h} \} /(1-\gamma^{1/h})].
		\end{align*} 
		$(ii)$   Since $\gamma_h(x,\cdot)$ is differentiable everywhere except at $x$, it holds 
		\begin{align*}
		&\frac{1}{Nh}\sum_{i=1}^N \gamma_h(x,t_i)=   h^{-1}\int_0^1 \gamma_h(x,t) \, \text{d}t+\mathcal{O}\{(Nh)^{-1}\}\\
		&=  h^{-1}\int_0^1 \gamma^{|x-t|/h} \, \text{d}t + h^{-1}\frac{\gamma^{1/h}}{1-\gamma^{1/h}} \int_0^1\gamma^{(x-t)/h}+ \gamma^{(t-x)/h}  \,\text{d}t + \mathcal{O}\{(Nh)^{-1}\}.
		\end{align*}
		Note that 
		\begin{align*}
		&h^{-1}\int_0^1 \gamma^{|x-t|/h} \, \text{d}t =\{\gamma^{(1-x)/h} +\gamma^{x/h} -2\}/\log\gamma \\
		\text{and } &h^{-1}\int_0^1\gamma^{(x-t)/h}+ \gamma^{(t-x)/h}  \, \text{d}x =  \gamma^{-t/h} \{\gamma^{1/h} -1\}/\log\gamma -\gamma^{t/h} \{\gamma^{-1/h} -1\}/\log\gamma.
		\end{align*} 
		In particular, for some constants $C_1,C_2>0$ depending on $\gamma\in(0,1)$ but not on $h$ and $x$, it holds $h^{-1}\int_0^1 \gamma^{|x-t|/h} \, \text{d}t\leq C_1$ and  $h^{-1}\gamma^{1/h}(1-\gamma^{1/h})^{-1}\int_0^1\gamma^{(x-t)/h}+ \gamma^{(t-x)/h}  \, \text{d}x \leq C_2$. 
		The proof to show $h^{-1}\int_0^1\gamma_h(x,t)^2 \, \text{d}t\leq C_3$ for some constant $C_3>0$ is similar. By assumption $(Nh)^{-1}=o(1)$ for $N\to \infty$ which concludes the proof. 
		
		$(iii)$ See Lemma~15 of \citet{schwarzdiss} with $p=2q-1$. 
	\end{proof}

	The next lemma below states that the sum of the correlated gamma random variables in each bin can be rewritten as a sum of independent gamma random variables. 
	\begin{lemma} \label{lemma:rewriteBins}
		Let  $\Sigma=\Sigma(f)$ with $f\in\mathcal{F}_{\beta}$ and $\beta>0$. Then,  for $n,p\to \infty$ such that $p/n\to c\in(0,\infty]$ and for $p\to \infty$ and $n$ constant,  $Q_k=\sum_{j=(k-1)p/T+1}^{kp/T} \sum_{i=1}^n\tilde{W}_{i,j}$ $(k=1,...,T)$
		where $$  \tilde{W}_{i,j}\overset{indep.}{\sim}
		\Gamma \left\{1/2,  \,2f(x_j)+\mathcal{O}\left (T^{-1}+T^{-1}p^{1-\beta}\log p \right )\right\}
		$$
		for $i=1,...,n$ and $j=(k-1)m+1,...,km$, and $x_j=(j-1)/(2p-2)$.
	\end{lemma} 
	\begin{proof} 
		It is sufficient to show the statement for $n=1$ by independence of the $Y_i$.  Then, the number of points per bin is $m=p/T$. For simplicity, the index $i$ is skipped in the following.  First, we write $Q_k$ as a matrix-vector product and refactor it so that it corresponds to a sum of independent scaled chi-squared random variables. In the second step, we calculate the scaling factors.  Let  $E^{(km)}$ be a diagonal matrix with ones on the  $(k-1)m+1,...,km$th entries and otherwise zero diagonal elements. Then, 
		$$ Q_k=\sum_{j=(k-1)m+1}^{km} W_{j}=  Y^T DE^{(km)}D Y.$$ 
		Define $X=\Sigma^{-1/2} Y\sim \mathcal{N}_p(0_p, I_p)$, where $I_p$ is the $p\times p$ identity matrix. Let $U^T\Gamma U$ be the eigendecomposition of $\Sigma^{1/2}  DE^{(jm)} D\Sigma^{1/2}$ with eigenvalues $(\lambda_j)_{j=1}^p$. The matrix $\Sigma^{1/2}$ is  invertible since $\Sigma$ invertible. Then,
		$$ Q_k=  X^T \Sigma^{1/2}  DE^{(km)} D \Sigma^{1/2} X= X^T{U}^T{\Gamma U}  X=\tilde{ X}^T\Gamma \tilde{ X} = \sum_{j=1}^p \lambda_j \tilde{ X}^2_j,$$
		where $\tilde{ X}=UX$. Since $ U$ is an orthogonal matrix, it follows that $\tilde{ X}\sim\mathcal{N}_p(0_p,I_p)$. In particular,  $Q_k$ is a sum of independent scaled  chi-squared random  variables $\tilde{W}_j=\lambda_j\tilde{X}^2_j\sim\Gamma( 1/2, 2\lambda_j),$ where the scaling factors are the eigenvalues $\lambda_j$. It remains to calculate the $\lambda_j$.
		For this we use that the matrix $\Sigma^{1/2} DE^{(km)} D \Sigma^{1/2}$ is similar to the matrix 
		$$A=D \Sigma^{-1/2} \left (\Sigma^{1/2} DE^{(km)}D\Sigma^{1/2} \right) \Sigma^{1/2}D=E^{(km)}D \Sigma D.$$
		In particular, $\Sigma^{1/2}D E^{(km)}D \Sigma^{1/2}$ has the same eigenvalues as $A$.
		The rows of $A$ are zero except for the $(k-1)m+1,...,km$th rows, which equal the $(k-1)m+1,...,km$th rows of $D \Sigma D$.
		By Lemma~1,
		$$\left ( D\Sigma D\right )_{i,j} = f(x_i)\delta_{i,j} + \frac{1+(-1)^{|i-j|}}{2} \mathcal{O}\left(p^{-1}+p^{-\beta} \log p \right ).$$ 
		Define $ B=(E^{(km)}D\Sigma DE^{(km)})_{r,s=(k-1)m+1}^{km}$. Then, $\lambda( A) = \lambda(B) \cup \{0\}$ and $ A$ has an eigenvalue zero with multiplicity at least $p-m$. Since $B$ is symmetric all eigenvalues are real-valued.	Application of the Gershgorin circle theorem  to $B$ yields that the remaining $m$ eigenvalues $\lambda_{(k-1)m+1},...,\lambda_{km}$ are contained in the following set
		\begin{align*}
		& \bigcup_{j=(k-1)m+1}^{km}  \left \{ z \in \mathbb{R}: \, | f(x_j) +\mathcal{O}\left(p^{-1}+p^{-\beta} \log p \right ) -z|= \mathcal{O}\left(mp^{-1}+mp^{-\beta} \log p \right ) \right \}.
		\end{align*} 
		Since $f$ is $\min\{\beta,1\}$-H\"older continuous, we have $\lambda_j= f(x_j)+\mathcal{O}(mp^{-1}+m p^{-\beta}\log p )$, for $j=(k-1)m+1,...,km$.  The proof is concluded using $T^{-1}=mp^{-1}$. 
	\end{proof}
	
	Finally, Lemma~\ref{lemma:representY} gives explicit bounds for the stochastic and deterministic errors of the variance-stabilizing transform. 
	\begin{lemma} \label{lemma:representY}
		If $\Sigma=\Sigma(f)$ with $f\in\mathcal{F}_{\beta}$ and $\beta>0$, then $Y^*_{k}=2^{-1/2}\log(Q_k/m)$ can be written as
		\begin{align*}
		& Y_{k}^*= H\left \{f \left (x_k \right)  \right \} +\epsilon_k+m^{-1/2}Z_k+ \xi_k, &(k=1,...,T)
		\end{align*}
		where $x_k=(k-1)/(2T-2),$ $|\epsilon_k|\lesssim (np)^{-1}+(np)^{-\beta}\log p, \, Z_k\sim\mathcal{N}(0,1),$ and $E(\xi_k)=0$. Furthermore,
		\begin{equation*}   \begin{aligned}
		&E|\xi_k|^\ell \lesssim (\log m )^{2\ell} \left \{m^{-\ell}+\left(T^{-1}+T^{-1} p^{1-\beta}\log p\right)^\ell \right\} & (\ell\in\mathbb{N}_{>1}),\\
		& \text{cov}(Z_k,Z_l)= \mathcal{O} \{ mp^{-2} + mp^{-2\beta}(\log p)^2\}& (k\neq l =1,...,T).	
		\end{aligned}
		\end{equation*} The results hold for  $n,p\to \infty$ such that $p/n\to c\in(0,\infty]$ and for $p\to \infty$ with $n$ constant.
	\end{lemma}
	\begin{proof} In the regression setting of \citet{cai2010nonparametricfest} and \citet{brown2010nonparametric} the gamma random variables in each bin are independent. Using Lemma~\ref{lemma:rewriteBins}, we can rewrite our setting into one with independent observations per bin. Thus, the first step of the proof  is to apply  the results of \citet{cai2010nonparametricfest} and \citet{brown2010nonparametric}  on the approximation of $Y^*_k$ by a Gaussian random variable $m^{-1/2}Z_k$ and a non-Gaussian remainder $\xi_k$.  Since in our setting the observations of different bins $i\neq j$ are correlated,  the covariance between the Gaussian random variables is derived in the second part of the proof.	
		
		By Lemma~\ref{lemma:rewriteBins}, $Q_k$ can be written as sum of $m=np/T$ independent gamma random variables with  mean function $\tilde{f}(x)=f(x)+S(x)$, where $S(x)=\mathcal{O}(T^{-1}+T^{-1}p^{1-\beta}\log p )$. 
		It follows from Lemma~1 and \ref{lemma:rewriteBins} that the $S$ term  depends continuously on $x\in[0,1]$ and that a global constant hidden in the $\mathcal{O}$ term exists which is independent of the $x$.
		Furthermore, $f\in[\delta, M_0]\subset \mathbb{R}_{>0}$ by assumption. As $\tilde{f}$ is the mean of gamma random variables, it holds that $\tilde{f}>0$. Thus,  $\tilde{f}$ is continuous and takes values in a compact set $[\epsilon,\nu] \subset \mathbb{R}_{>0}$ of the support of the gamma distribution. By Theorem~1 of \citet{cai2010nonparametricfest} for the gamma distribution 
		it follows $$Y_{k}^*=H\left \{\tilde{f}(x_k^* )\right\} +m^{-1/2}Z_{k}+\xi_{k}  \quad (k=1,...,T),$$ 
		where $(k-1)/(2T-2)\leq x_k^* \leq k/(2T-2)$, $Z_{k}\sim \mathcal{N}(0,1)$, and  $\xi_{k}$ is a mean zero random variable.   $Z_k$ is defined via quantile coupling. It is the Gaussian approximation to the standardized version of $Q_k$ where the $W_{ij}$ in the $k$th bin are assumed to be i.i.d. In particular, $Z_k$ approximates $$\bar{Q}_{k}=\frac{\sum_{j=(k-1)p/T+1}^{kp/T} \sum_{i=1}^n \bar{W}_{i,j}-m\tilde{f}(x_k^*)}{(2m)^{\frac{1}{2}}\tilde{f}(x_k^*)},$$
		where $\bar{W}_{i,j}\overset{\text{i.i.d.}}{\sim}\Gamma(1/2,2\tilde{f}(x_k^*))$ and such that $\mbox{cov}(\bar{W}_{i,j},\bar{W}_{i,h})=\mbox{cov}({W}_{i,j},{W}_{i,h})$ for $j=(k-1)p/T+1,...,kp/T$ and $h\in\{1,...,p\}\setminus\{(k-1)p/T+1,...,kp/T\}$.  
		
		Let $\theta$ be the maximum difference of the observations' means in each bin. Then,
		$$\theta=\max_{k=1,...,T} \max_{j,l =(k-1)p/T+1,...,kp/T} \left | E[ \tilde{W}_{1,j}]-E[ \tilde{W}_{1,l}] \right| \leq c_1\left(T^{-1}+T^{-1}p^{1-\beta}\log p\right).$$ 
		Application of Lemma~4 and Theorem~1 of \citet{brown2010nonparametric} yield that $\xi_{k}$ satisfies   $$E[|\xi_{k}|^\ell] \lesssim (\log m)^{2\ell}  \{m^{-\ell}+(T^{-1}+T^{-1} p^{1-\beta}\log p)^\ell  \}$$ for any constant integer $\ell>0$ and where the hidden constant depends on $\ell$.
		Next, 
		$$ Y_{k}^*= H\left \{f \left (\frac{k-1}{2T-2} \right)  \right \} +\epsilon_k+m^{-1/2}Z_k+ \xi_k \quad (k=1,...,T),$$
		where $\epsilon_k=H [f \left \{(k-1)/(2T-2) \right\}  ]-H \{\tilde{f} \left (x_k^* \right) \}$ and $|\epsilon_k|\lesssim (np)^{-1}+(np)^{-\beta}\log p $ by the Lipschitz continuity of the logarithm and the $\alpha$-H\"older continuity of $f$. 
		
		To compute $\mbox{cov}(Z_k,Z_l)$ for $k\neq l$ we use  Hoeffding's covariance identity.  In particular, we show that $\mbox{cov}(Z_k,Z_l)\lesssim \mbox{cov}(\bar{Q}_k,\bar{Q}_l)$. The claim then follows from 
		\begin{align*}
		0\leq \mbox{cov}(\bar{Q}_k,\bar{Q}_l)
		&=\frac{m\mathcal{O}\{\mbox{cov}(W_{i,j},W_{i,h})\}}{2\tilde{f}(x_k^*)\tilde{f}(x_l^*)}=\mathcal{O}(mp^{-2}+mp^{-2\beta}\log p ),
		\end{align*} where $j=(k-1)p/T+1,...,kp/T$ and $h=(l-1)p/T+1,...,lp/T$.
		
		Let $\Phi$ and $F_{\bar{Q}} $ denote the  cumulative distribution functions of $Z_k,\, Z_l,\, \bar{Q}_k$ and $ \bar{Q}_l$. Since the $Z_k$ are defined via quantile coupling, it holds $Z_k=\Phi^{-1}\{F_{\bar{Q}}(\bar{Q}_k)\}$, see  \citet{komlos1975approximation, mason2012quantile}.
		Furthermore, define the uniform random variables $U_k=\Phi(Z_k)=F_{\bar{Q}}(\bar{Q}_k)$ and $U_l=\Phi(Z_l)=F_{\bar{Q}}(\bar{Q}_l)$ with cumulative distribution function $F_U$. The corresponding joint cumulative distribution functions are denoted by $F_{Z,Z},\, F_{\bar{Q},\bar{Q}},\, F_{U,U}$. By Hoeffding's covariance identity,
		\begin{align*}
		\mbox{cov}(\bar{Q}_k,\bar{Q}_l)&= \int _{\mathbb {R} }\int _{\mathbb {R} }F_{\bar{Q},\bar{Q}}(q,r)-F_{\bar{Q}}(q)F_{\bar{Q}}(r)\,\text{d}q\,\text{d}r\\
		&=\int _{\mathbb {R} }\int _{\mathbb {R} }\left\{F_{Z,Z}(x,y)-\Phi(x)\Phi(y)\right\} \frac{\phi(x)}{f_{\bar{Q}}\{F_{\bar{Q}}^{-1}(x)\}}\frac{\phi(y)}{f_{\bar{Q}}\{F_{\bar{Q}}^{-1}(y)\}}\,\text{d}x\,\text{d}y.
		\end{align*} 
		Let $\rho=\mbox{cov}(Z_k,Z_l)$. Then, the identity
		$$F_{ZZ}(x,y)=\frac{1}{2\pi} \int_0^\rho \frac{1}{(1-r^2)^{1/2}} \exp \left \{ -\frac{x^2-2rxy+y^2}{2(1-r^2)}\right \}\,\text{d}r +\Phi(x)\Phi(y)$$
		implies $F_{Z,Z}(x,y)-\Phi(x)\Phi(y)\geq 0\, \text{ for all } x,y\in \mathbb{R}$ if and only if $ \rho\geq 0$, see Equation 9.3.1.3 of \citet{patel1996handbook}.
		Since $\mbox{cov}(\bar{Q}_k,\bar{Q}_l)\geq 0$ and the ratio of two densities is non-negative,  it follows $\mbox{cov}(Z_k,Z_l)\geq 0$.
		Furthermore,
		\begin{align*}
		\mbox{cov}(Z_l,Z_k)&= \int _{\mathbb {R} }\int _{\mathbb {R} }F_{Z,Z}(x,y)-\Phi(x)\Phi(y)\,\text{d}x\,\text{d}y\\
		&= \int_0^1 \int_0^1 \left\{F_{UU}(u,v)-F_U(u)F_U(v) \right\} \frac{1}{\phi\{\Phi^{-1}(u)\}}\frac{1}{\phi\{\Phi^{-1}(v)\}}\,\text{d}u\,\text{d}v\\
		&\leq  \int_{1/2}^1 \int_{1/2}^1 \left\{F_{UU}(u,v)-F_U(u)F_U(v) \right\} \frac{1}{\phi\{\Phi^{-1}(u)\}}\frac{1}{\phi\{\Phi^{-1}(v)\}}\,\text{d}u\,\text{d}v \\
		&\, + \int_0^1 \int_0^{1/2} \left\{F_{UU}(u,v)-F_U(u)F_U(v) \right\} \frac{2}{\phi\{\Phi^{-1}(u)\}}\frac{1}{\phi\{\Phi^{-1}(v)\}}\,\text{d}u\,\text{d}v.
		\end{align*} 
		The last integral corresponds to $E(Z_kZ_l^-)=\mbox{cov}(Z_k,Z_l)/2$ where $Z_l^-{=}Z_l\mathbbm{1}\{Z_l\leq0\}$. It remains to show that it exists a $c>0$ such that $f_{\bar{Q}}\{F_{\bar{Q}}^{-1}(u)\}/\phi\{\Phi^{-1}(u)\}\leq   c$ for all $u\in (1/2,1)$. Then, it follows
		\begin{align*}
		\mbox{cov}(Z_l,Z_k)
		&\leq 2c^2 \int_0^1 \int_0^1 \left\{F_{UU}(u,v)-F_U(u)F_U(v) \right\} \frac{1}{f_{\bar{Q}}\{F_{\bar{Q}}^{-1}(u)\}}\frac{1}{f_{\bar{Q}}\{F_{\bar{Q}}^{-1}(v)\}}\,\text{d}u\,\text{d}v\\
		&=2c^2\mbox{cov}(\bar{Q}_k,\bar{Q}_l).
		\end{align*}
		Furthermore, for  $u\in (1/2,1)$ we have
		\begin{align}
		\frac{\partial}{\partial u}\frac{f_{\bar{Q}}\{F_{\bar{Q}}^{-1}(u)\}}{\phi\{\Phi^{-1}(u)\}}
		&= \frac{f^\prime_{\bar{Q}}\{F_{\bar{Q}}^{-1}(u)\}}{f_{\bar{Q}}\{F_{\bar{Q}}^{-1}(u)\}\phi\{\Phi^{-1}(u)\}} -\frac{\Phi^{-1}(u)f_{\bar{Q}}\{F_{\bar{Q}}^{-1}(u)\}}{\phi\{\Phi^{-1}(u)\}^2} \label{derivative}.
		\end{align} Since $\Phi^{-1}(u)>0$ for $u\in(1/2,1)$, the second term is  positive. To see that the first term is  negative for $u\in (1/2,1)$ we rewrite  $f_{\bar{Q}}$ and $F^{-1}_{\bar{Q}}$ with respect to the non-normalized gamma random variable $X=\sum_{j=(k-1)p/T+1}^{kp/T} \sum_{i=1}^n \bar{W}_{i,j}\sim\Gamma(m/2, 2\tilde{f}(x_k^*))$. Hence,
		\begin{align*}
		&f_{\bar{Q}}(x)=\sigma f_{X}(x\sigma +\mu),\quad  & F_{\bar{Q}}^{-1}(u)=\sigma^{-1}\{F^{-1}_{X}(u)-\mu\},
		\end{align*} where  $\sigma$ and $\mu$ are the  standard deviation and the mean of $X$. 
		Since the mode of $A\sim\Gamma(a,b)$ is at $b (a-1)$ and $x\sigma +\mu=2\tilde{f}(x_k)(m/2-1)$ implies $x=-(2/m)^{1/2}$, it follows
		that $f_{\bar{Q}}(x)$ is monotone decreasing for $x\geq-(2/m)^{1/2}$. Furthermore, $F_{\bar{Q}}(-(m/2)^{1/2})\leq 0.5$ for all $m\in\mathbb{N}$ as $f_{\bar{Q}}(x)$ is right-skewed.
		In particular, $-(m/2)^{1/2}\leq F^{-1}_{\bar{Q}}(1/2)$ for all $m\in\mathbb{N}$.  Finally, since $f_{\bar{Q}}(-(2/m)^{1/2})\to \phi(0)$  for $m\to \infty$ there is a constant $c>0$ not depending on $m$ such that
		\begin{equation*} 
		\frac{f_{\bar{Q}}\{F_{\bar{Q}}^{-1}(u)\}}{\phi\{\Phi^{-1}(u)\}}\leq\frac{f_{\bar{Q}}\{-(2/m)^{1/2}\}}{\phi(0)}  \leq c.
		\end{equation*} 
	\end{proof}
	
	\newpage 
	\section{Supplementary Material}
	\subsection{Simulation Study with Gamma Distribution}
	\label{app:nogauss_simulation}
	
	\subsubsection{Data simulation}
	The observations follow  a centered gamma distribution with covariance matrices $ \Sigma_1,\, \Sigma_2,\, \Sigma_3$ of examples ${1,\, 2,\, 3}$, i.e., $ Y_i=\Sigma^{1/2}_jZ_i \,(j=1,3;i=1,...,n),$ where $Z_i=(z_{i,1},....,z_{i,p})^T$ with $z_{i,1}+2^{1/2},...,z_{i,p}+2^{1/2} \overset{\text{i.i.d.}}{\sim} \Gamma(2,2^{-1/2})$. For example $2$, the parameter \texttt{innov} of the R function \texttt{arima.sim} is used to pass the innovations $z_{i,1}+2^{1/2},...,z_{i,p}+2^{1/2} \overset{\text{i.i.d.}}{\sim} \Gamma(2,2^{-1/2})$.  The resulting process is multiplied by $(1.2)^{1/2}$. 
	
	\subsubsection{Derivation of the variance-stabilizing transform function}
	The covariance of $DY_i$ is $\text{cov}(DY_i)=D \Sigma D$ and the marginal distribution of $D_j^TY_i$ follows a centered gamma distribution with variance $f(x_j)+o(1)$. Thus, $D_j^TY_i+E(D_j^TY_i)\sim \Gamma(2,\theta)$ with $\theta=[\{f(\pi x_j)+o(1)\}/2]^{1/2}$. 
	Recall that the density of $A\sim\Gamma(k,\theta)$ and its centered version $B=A-E[A]$  are given by
	
	$$ {\displaystyle g_A(x)={\frac {1}{\Gamma (k)\theta ^{k}}}x^{k-1}e^{-x/\theta }}\mathbbm{1}\{x\geq 0\}$$
	$$g_B(x) ={\frac {1}{\Gamma (k)\theta ^{k}}}(x+k\theta)^{k-1}e^{-(x+k\theta)/\theta }\mathbbm{1}\{x\geq-k\theta\}.$$
	It follows that the density of $B^2$ is then
	\begin{align*}
	g_{B^2}(x)&=\frac{g_B(x^{1/2})+g_B(-x^{1/2})}{2x^{1/2}} \mathbbm{1}\{x\geq 0\} \\
	&= \frac {1}{\Gamma (k)\theta ^{k}} \frac{(x^{1/2}+k\theta)^{k-1}e^{-(x^{1/2}+k\theta)/\theta}}{2x^{1/2}}\mathbbm{1}\{x\geq 0\}  \\
	& \quad +  \frac {1}{\Gamma (k)\theta ^{k}} \frac{ (-x^{1/2}+k\theta)^{k-1}e^{-(-x^{1/2}+k\theta)/\theta}} {2x^{1/2}} \mathbbm{1}\{ -(k\theta)<-x^{1/2}<0\}  \\
	\end{align*}
	With this density one obtains $E(B^2)=k\theta^2=$ and $\text{var}(B^2)= c_k(k\theta)^2$ for some explicit constant $c_k$ depending on $k$.
	Hence, $E(W_{i,j})=E\{(D_j^TY_i)^2\}=f(\pi x_j)+o(1)$ and $\text{var}\{(D_j^TY_i)^2\}=5\{f(\pi x_j)+o(1)\}^2$. The corresponding variance-stabilizing function
	$G:\mathbb{R} \to \mathbb{R}$ satisfies $G^{\prime}(\mu)=V^{-1/2}(\mu)$, where $V(\mu)$ is the variance as a function of the mean $\mu$.
	In particular, for the distribution of $W_{i,j}$ one obtains $G(\mu)=5^{-1/2}\log(\mu)$ with $\mu=f(x_j)+o(1)$. The QQ-plots in Fig.~\ref{qqGamma} confirm the correctness of the log-transform to be  the variance-stabilizing function. The transformed data $Y_k^*, (k=1,...,T)$ for all three examples look very similar to normally distributed random variables with homogeneous variance.
	
	\subsubsection{Results} 
	Since the variance-stabilizing function of the $W_{i,j}$ is a scaled logarithm, our method can be applied to estimate $\Sigma$ from the (centered) gamma-distributed data $Y_1,...,Y_n$. The tuning parameters of the tapering estimator and for our method are chosen the same as for the setting with Gaussian data, see Section~\ref{sec:simulation}. The bin number $T=500$ was selected  from the QQ-plots, see also Fig.~\ref{qqUniform} in the next section. Table \ref{tab_Simulation_A_gamma}--\ref{tab_Simulation_C_gamma} show  the results for  (A) $p=5000,\, n=1$, (B) $p=1000,\, n=50$ and  (C) $p=5000,\, n=10$, respectively.
	
	\begin{table}[H]
		\centering
		\scriptsize
		\def~{\hphantom{0}}
		\caption{ (A) $p=5000,\, n=1$: Errors of the Toeplitz covariance matrix and the spectral density estimators with respect to the spectral and the $L_2$ norm;  average computation time of the covariance estimators in seconds for one Monte Carlo sample is in the last column; all numbers multiplied with 100 except the last column }{
			\begin{tabular}{lccccccc}
				&\multicolumn{2}{c}{ Process (1)} &\multicolumn{2}{c}{ Process (2) } &\multicolumn{2}{c}{ Process (3)} &time  \\
				& $\|\widehat{\Sigma}-\Sigma\|^2$ & $\|\hat{f}-f\|^2_2$ &  $\|\widehat{\Sigma}-\Sigma\|^2$ & $\|\hat{f}-f\|^2_2$ &  $\|\widehat{\Sigma}-\Sigma\|^2$ & $\|\hat{f}-f\|^2_2$ & in sec \\[5pt]	
				our method (GCV) & 1.181 & 0.434 & 2.285 & 0.609 & 5.182 & 0.884 & 4.238 \\ 
				our method (ML) & 0.973 & 0.391 & 2.121 & 0.581 & 4.787 & 0.819 & 4.241 \\ 
				tapering (CV) & 0.951 & 0.372 & 2.626 & 0.803 & 4.156 & 1.151 & 4.682 \\ 
				sample covariance & 16750.719 & 3380.642 & 20195.498 & 3631.043 & 24828.431 & 3597.585 & 0.350 \\ 
				our method (GCV-oracle) & 0.794 & 0.346 & 1.701 & 0.516 & 4.172 & 0.757 & \\ 
				our method (ML-oracle) & 0.820 & 0.356 & 1.965 & 0.550 & 4.789 & 0.812 &\\ 
				tapering (oracle) & 0.759 & 0.335 & 1.313 & 0.443 & 1.991 & 0.539 &  \\ 
		\end{tabular}}
		\label{tab_Simulation_A_gamma}
	\end{table}
	
	\begin{table}[H]
		\centering
		\scriptsize
		\def~{\hphantom{0}}
		\caption{ (B) $p=1000,\, n=50$: Errors of the Toeplitz covariance matrix and the spectral density estimators with respect to the spectral and the $L_2$ norm;  average computation time of the covariance estimators in seconds for one Monte Carlo sample is in the last column; all numbers multiplied with 100 except the last column }{
			\begin{tabular}{lccccccc}
				&\multicolumn{2}{c}{ Process (1)} &\multicolumn{2}{c}{ Process (2) } &\multicolumn{2}{c}{ Process (3)} &time  \\
				& $\|\widehat{\Sigma}-\Sigma\|^2$ & $\|\hat{f}-f\|^2_2$ &  $\|\widehat{\Sigma}-\Sigma\|^2$ & $\|\hat{f}-f\|^2_2$ &  $\|\widehat{\Sigma}-\Sigma\|^2$ & $\|\hat{f}-f\|^2_2$ & in sec \\[5pt]	
				our method (GCV) & 0.125 & 0.042 & 0.457 & 0.142 & 0.590 & 0.095 & 23.169 \\ 
				our method (ML) & 0.106 & 0.038 & 0.525 & 0.151 & 0.736 & 0.106 & 23.093 \\ 
				tapering (CV) & 0.122 & 0.040 & 0.497 & 0.154 & 0.410 & 0.087 & 24.444 \\ 
				sample covariance & 80.845 & 57.839 & 103.125 & 61.602 & 126.620 & 60.425 & 0.155 \\ 
				our method (GCV-oracle) & 0.087 & 0.033 & 0.390 & 0.134 & 0.552 & 0.090 & \\ 
				our method (ML-oracle) & 0.095 & 0.035 & 0.515 & 0.150 & 0.731 & 0.105 & \\ 
				tapering (oracle) & 0.074 & 0.032 & 0.371 & 0.137 & 0.307 & 0.068 & \\
		\end{tabular}}
		\label{tab_Simulation_B_gamma}
	\end{table}

	\begin{table}[H]
		\centering
		\scriptsize
		\def~{\hphantom{0}}
		\caption{ (C) $p=5000,\, n=10$: Errors of the Toeplitz covariance matrix and the spectral density estimators with respect to the spectral and the $L_2$ norm;  average computation time of the covariance estimators in seconds for one Monte Carlo sample is in the last column; all numbers multiplied with 100 except the last column }{
			\begin{tabular}{lccccccc}
				&\multicolumn{2}{c}{ Process (1)} &\multicolumn{2}{c}{ Process (2) } &\multicolumn{2}{c}{ Process (3)} &time  \\
				& $\|\widehat{\Sigma}-\Sigma\|^2$ & $\|\hat{f}-f\|^2_2$ &  $\|\widehat{\Sigma}-\Sigma\|^2$ & $\|\hat{f}-f\|^2_2$ &  $\|\widehat{\Sigma}-\Sigma\|^2$ & $\|\hat{f}-f\|^2_2$ & in sec \\[5pt]	
				our method (GCV) & 0.138 & 0.042 & 0.386 & 0.129 & 0.626 & 0.099 & 4.347 \\ 
				our method (ML) & 0.109 & 0.037 & 0.441 & 0.137 & 0.700 & 0.105 & 4.352 \\ 
				tapering (CV) & 0.175 & 0.047 & 0.485 & 0.142 & 0.508 & 0.092 & 680.471 \\ 
				sample covariance & 651.720 & 352.037 & 846.298 & 380.690 & 1061033.779 & 381.164 & 1.202 \\ 
				our method (GCV-oracle) & 0.091 & 0.033 & 0.330 & 0.121 & 0.557 & 0.091 &  \\ 
				our method (ML-oracle) & 0.097 & 0.034 & 0.434 & 0.136 & 0.697 & 0.105 &  \\ 
				tapering (oracle) & 0.076 & 0.031 & 0.307 & 0.120 & 0.305 & 0.068 & \\ 	
		\end{tabular}}
		\label{tab_Simulation_C_gamma}
	\end{table}

	\subsection{Simulation Study with Uniform Distribution}

	\subsubsection{Data simulation}
	The observations are generated as follows: set $ Y_i=\Sigma^{1/2}_jZ_i \,(i=1,...,n;j=1,3)$ where  $\Sigma_j$  is the Toeplitz covariance matrix of example $j$ and  $Z_i=(z_{i,1},...,z_{i,p})^T$ with $z_{i,1},...,z_{i,p} \overset{\text{i.i.d.}}{\sim} \text{Unif}[-3^{1/2},3^{1/2}]$. For example $2$, the parameter \texttt{innov} of the R function \texttt{arima.sim} is used to pass the innovations $X_1,...,X_n \overset{\text{i.i.d.}}{\sim} \text{Unif}[-3^{1/2},3^{1/2}]$. The resulting process is multiplied by $(1.2)^{1/2}$. 
	
	
	\subsubsection{Variance-stabilizing transform function}
	For this example the variance-stabilizing transform function cannot be derived explicitly. The QQ-plots in Fig.~\ref{qqUniform} indicate that the $\log$-transform is variance-stabilizing for $(D_j^tY_i)^2$ for all three examples, as the transformed data $Y_k^*\, (k=1,...,T)$ with $T=500$ look very similar to normally distributed random variables with homogeneous variance. 
	Thus, our method can be applied to estimate $\Sigma$ from the data $Y_1,...,Y_n$. 
	\subsubsection{Results}
	Tables \ref{tab_Simulation_A_uniform}--\ref{tab_Simulation_C_uniform}  show the results for  (A) $p=5000,\, n=1$, (B) $p=1000,\, n=50$ and  (C) $p=5000,\, n=10$,  respectively.
	The tuning parameters of the tapering estimator and for our method are chosen the same as for the setting with Gaussian data, see Section~\ref{sec:simulation}.

	\begin{table}[H]
		\centering
		\scriptsize
		\def~{\hphantom{0}}
		\caption{ (A) $p=5000,\, n=1$: Errors of the Toeplitz covariance matrix and the spectral density estimators with respect to the spectral and the $L_2$ norm;  average computation time of the covariance estimators in seconds for one Monte Carlo sample is in the last column; all numbers multiplied with 100 except the last column }{
			\begin{tabular}{lccccccc}
				&\multicolumn{2}{c}{ Process (1)} &\multicolumn{2}{c}{ Process (2) } &\multicolumn{2}{c}{ Process (3)} &time  \\
				& $\|\widehat{\Sigma}-\Sigma\|^2$ & $\|\hat{f}-f\|^2_2$ &  $\|\widehat{\Sigma}-\Sigma\|^2$ & $\|\hat{f}-f\|^2_2$ &  $\|\widehat{\Sigma}-\Sigma\|^2$ & $\|\hat{f}-f\|^2_2$ & in sec \\[5pt]	
				our method (GCV) & 0.737 & 0.230 & 1.952 & 0.488 & 4.010 & 0.616 & 4.282 \\ 
				our method (ML) & 0.572 & 0.190 & 1.560 & 0.422 & 3.991 & 0.610 & 4.279 \\ 
				tapering (CV) & 0.490 & 0.163 & 2.282 & 0.703 & 2.843 & 0.769 & 4.873 \\ 
				sample covariance & 17131.995 & 3599.128 & 19077.308 & 4031.457 & 26960.006 & 3770.437 & 0.355 \\ 
				our method (GCV-oracle) & 0.392 & 0.139 & 1.271 & 0.367 & 3.280 & 0.543 &  \\ 
				our method (ML-oracle) & 0.417 & 0.147 & 1.461 & 0.398 & 3.903 & 0.599 &\\ 
				tapering (oracle) & 0.371 & 0.130 & 1.022 & 0.314 & 1.463 & 0.342 &  \\ 
		\end{tabular}}
		\label{tab_Simulation_A_uniform}
	\end{table}
	
	\begin{table}[H]
		\centering
		\scriptsize
		\def~{\hphantom{0}}
		\caption{ (B) $p=1000,\, n=50$: Errors of the Toeplitz covariance matrix and the spectral density estimators with respect to the spectral and the $L_2$ norm;  average computation time of the covariance estimators in seconds for one Monte Carlo sample is in the last column; all numbers multiplied with 100 except the last column }{
			\begin{tabular}{lccccccc}
				&\multicolumn{2}{c}{ Process (1)} &\multicolumn{2}{c}{ Process (2) } &\multicolumn{2}{c}{ Process (3)} &time  \\
				& $\|\widehat{\Sigma}-\Sigma\|^2$ & $\|\hat{f}-f\|^2_2$ &  $\|\widehat{\Sigma}-\Sigma\|^2$ & $\|\hat{f}-f\|^2_2$ &  $\|\widehat{\Sigma}-\Sigma\|^2$ & $\|\hat{f}-f\|^2_2$ & in sec \\[5pt]	
				our method (GCV) & 0.090 & 0.024 & 0.397 & 0.116 & 0.682 & 0.092 & 31.574 \\ 
				our method (ML) & 0.079 & 0.021 & 0.452 & 0.123 & 0.766 & 0.096 & 31.473 \\ 
				tapering (CV) & 0.134 & 0.030 & 0.424 & 0.126 & 0.423 & 0.074 & 23.444 \\ 
				sample covariance & 80.230 & 56.741 & 102.424 & 61.282 & 123.505 & 58.847 & 0.136 \\ 
				our method (GCV-oracle) & 0.057 & 0.016 & 0.332 & 0.106 & 0.567 & 0.078 &  \\ 
				our method (ML-oracle) & 0.066 & 0.018 & 0.443 & 0.122 & 0.757 & 0.094 &  \\ 
				tapering (oracle) & 0.047 & 0.015 & 0.318 & 0.110 & 0.260 & 0.052 & \\ 	
		\end{tabular}}
		\label{tab_Simulation_B_uniform}
	\end{table}
	
	\begin{table}[H]
		\centering
		\scriptsize
		\def~{\hphantom{0}}
		\caption{ (C) $p=5000,\, n=10$: Errors of the Toeplitz covariance matrix and the spectral density estimators with respect to the spectral and the $L_2$ norm;  average computation time of the covariance estimators in seconds for one Monte Carlo sample is in the last column; all numbers multiplied with 100 except the last column }{
			\begin{tabular}{lccccccc}
				&\multicolumn{2}{c}{ Process (1)} &\multicolumn{2}{c}{ Process (2) } &\multicolumn{2}{c}{ Process (3)} &time  \\
				& $\|\widehat{\Sigma}-\Sigma\|^2$ & $\|\hat{f}-f\|^2_2$ &  $\|\widehat{\Sigma}-\Sigma\|^2$ & $\|\hat{f}-f\|^2_2$ &  $\|\widehat{\Sigma}-\Sigma\|^2$ & $\|\hat{f}-f\|^2_2$ & in sec \\[5pt]	
				our method (GCV) & 0.090 & 0.023 & 0.353 & 0.112 & 0.568 & 0.079 & 4.292 \\ 
				our method (ML) & 0.074 & 0.019 & 0.412 & 0.119 & 0.635 & 0.085 & 4.285 \\ 
				tapering (CV) & 0.127 & 0.025 & 0.437 & 0.124 & 0.427 & 0.066 & 638.830 \\ 
				sample covariance & 662.955 & 365.878 & 824.441 & 379.398 & 1049580.980 & 386.305 & 1.205 \\ 
				our method (GCV-oracle) & 0.060 & 0.016 & 0.314 & 0.105 & 0.478 & 0.070 &  \\ 
				our method (ML-oracle) & 0.066 & 0.017 & 0.408 & 0.118 & 0.626 & 0.085 &  \\ 
				tapering (oracle) & 0.043 & 0.013 & 0.299 & 0.105 & 0.230 & 0.046 &  \\ 
		\end{tabular}}
		\label{tab_Simulation_C_uniform}
	\end{table}

	\subsection{QQ-Plots} \label{qqplots}
	In this section, QQ-plots of the transformed data $Y_k^*\, (k=1,...,T)$ for the simulated data are shown. The QQ-plots are displayed for scenario (A) $p=5000$, $n=1$  as it contains fewer data points than scenario (B) and (C) and thus the largest Gaussian approximation error for the transformed data $Y_k^*$.
	Figure~\ref{qqGauss} demonstrates the impact of the bin number $T$ on the transformed  data $Y_k^*$ for example (1) with Gaussian data $Y$ (see simulation study in Section~6 of the main article). The plots justify the choice of $T=500$ for the simulation study, as in this case the $Y_k^*$ are very close to Gaussian distributed random variables. Figure~\ref{qqGamma} and \ref{qqUniform} show the  QQ-plots of $Y_k^*$ for example 1--3 with bin number $T=500$, where the original data $Y$ is gamma distributed or uniform distributed, respectively. The plots look similar to the QQ-plots in Fig.~\ref{qqGauss} for the Gaussian data $Y$ and indicate that our method to for Toeplitz covariance estimation is applicable also for these non-Gaussian data.   
	
	\begin{figure}[H]
		\begin{subfigure}[h]{0.32\textwidth}
			\includegraphics[width=\linewidth]{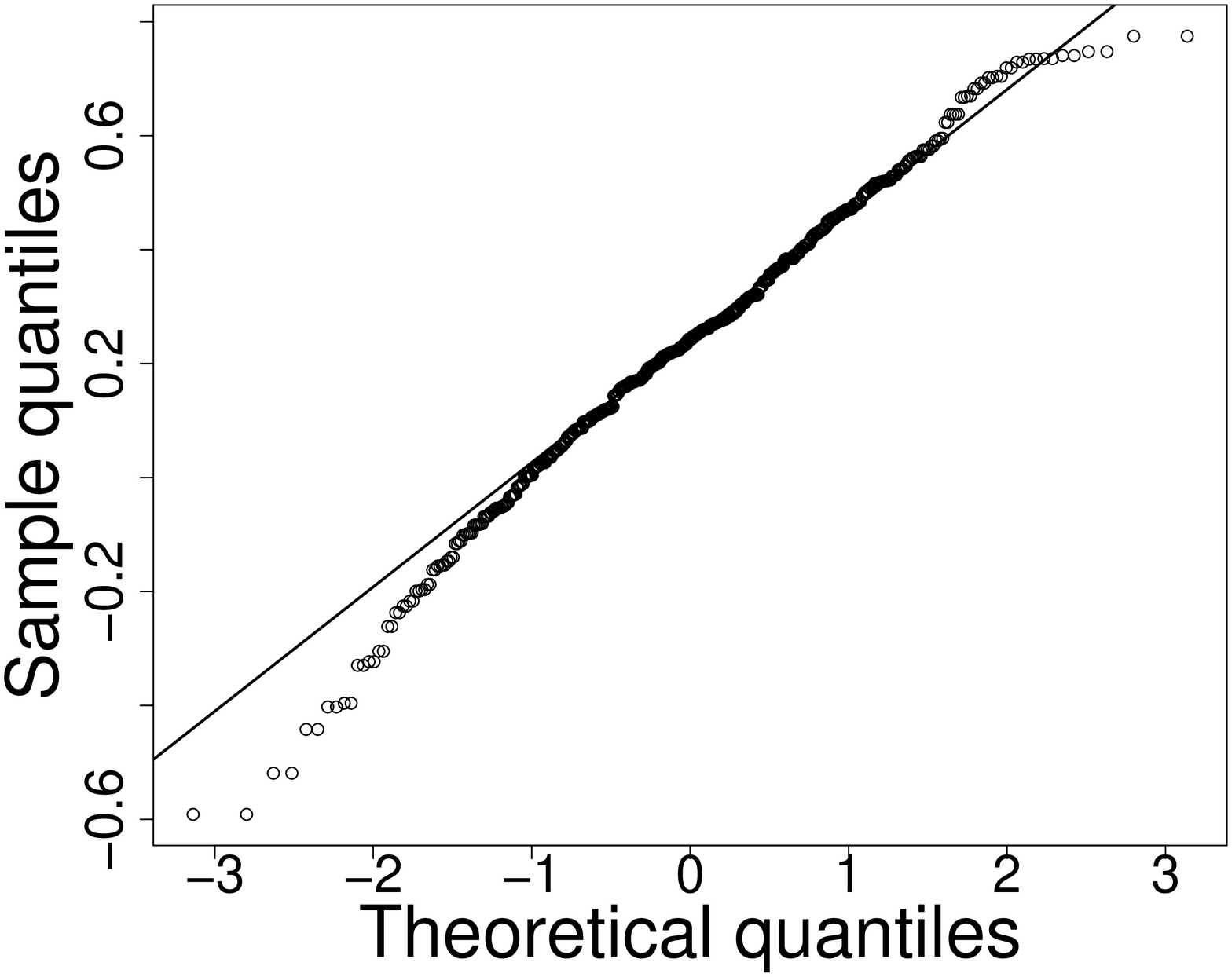} 
			\caption*{ example (1), $T=300$ }
		\end{subfigure}\hfill
		\begin{subfigure}[h]{0.32\textwidth}
			\includegraphics[width=\linewidth]{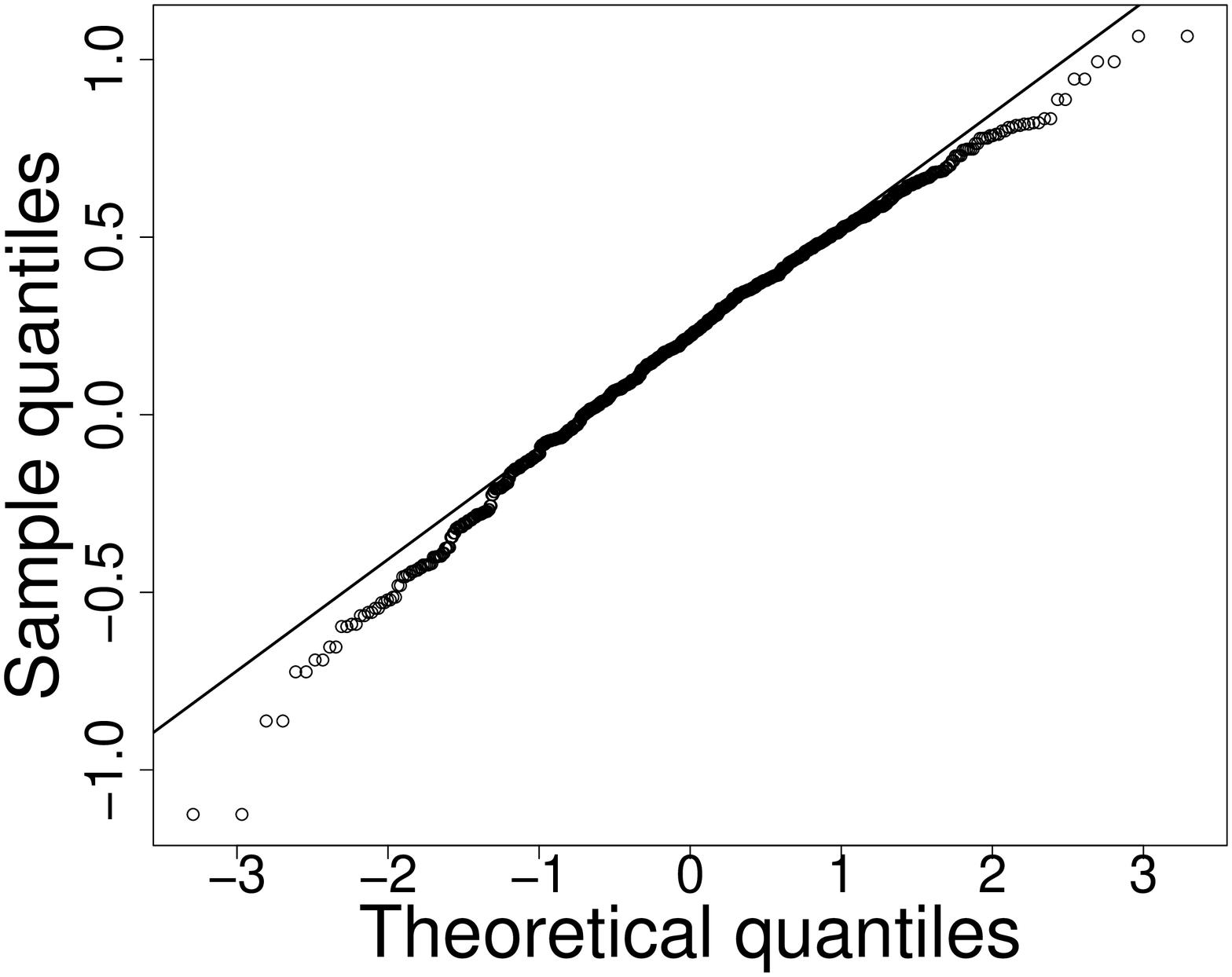}
			\caption*{ example (1), $T=500$}
		\end{subfigure}\hfill
		\begin{subfigure}[h]{0.32\textwidth}
			\includegraphics[width=\linewidth]{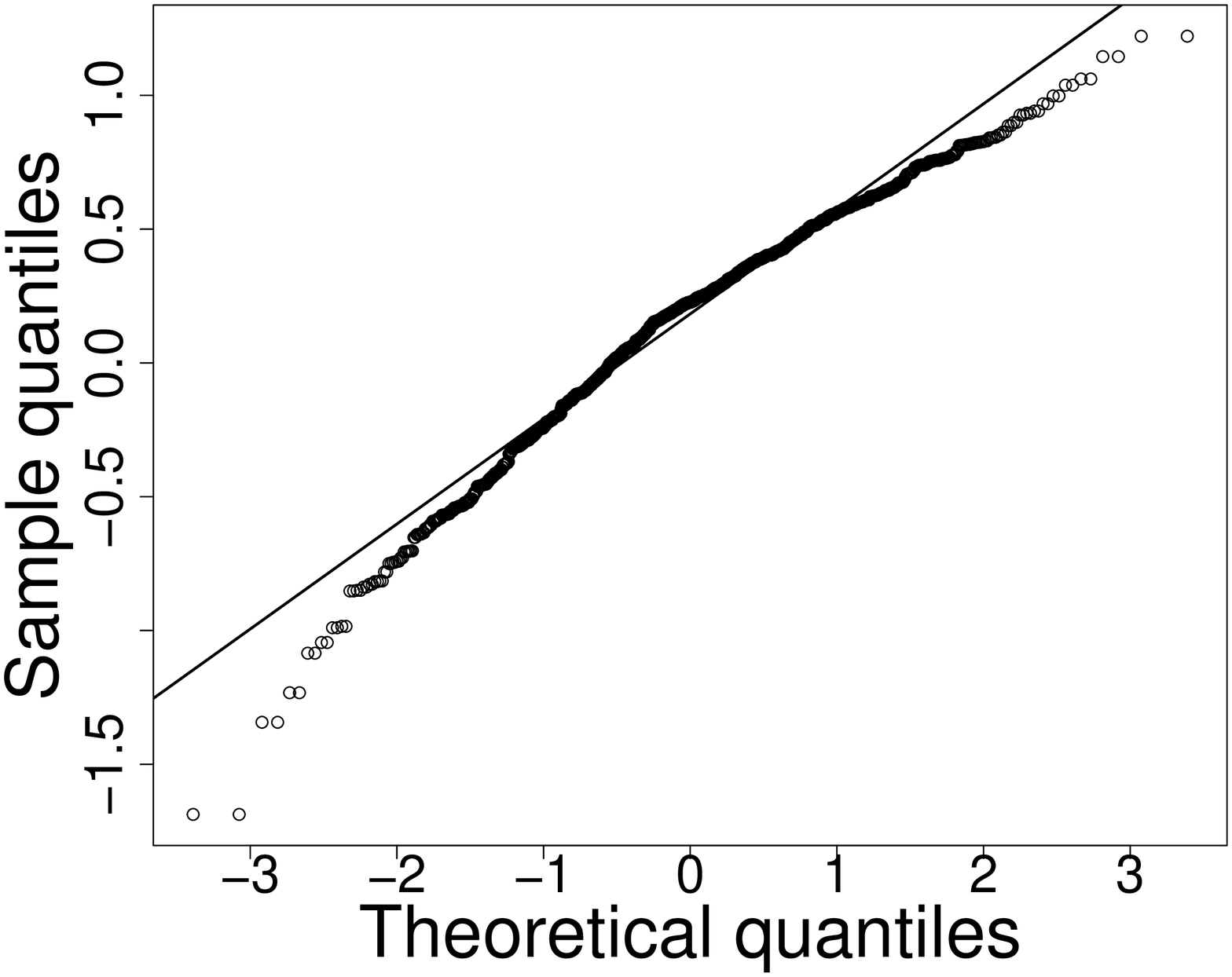}
			\caption*{  example (1), $T=700$ }
		\end{subfigure}\hfill
		\caption{Simulation study with Gaussian data; QQ-plots of binned and transformed data $Y_k^*\, (k=1,...,2T-2)$ with different bin numbers $T$ for example (1) with $p=5000,\, n=1$.  }
		\label{qqGauss}
	\end{figure}
	
	\begin{figure}[H]
		\begin{subfigure}[h]{0.32\textwidth}
			\includegraphics[width=\linewidth]{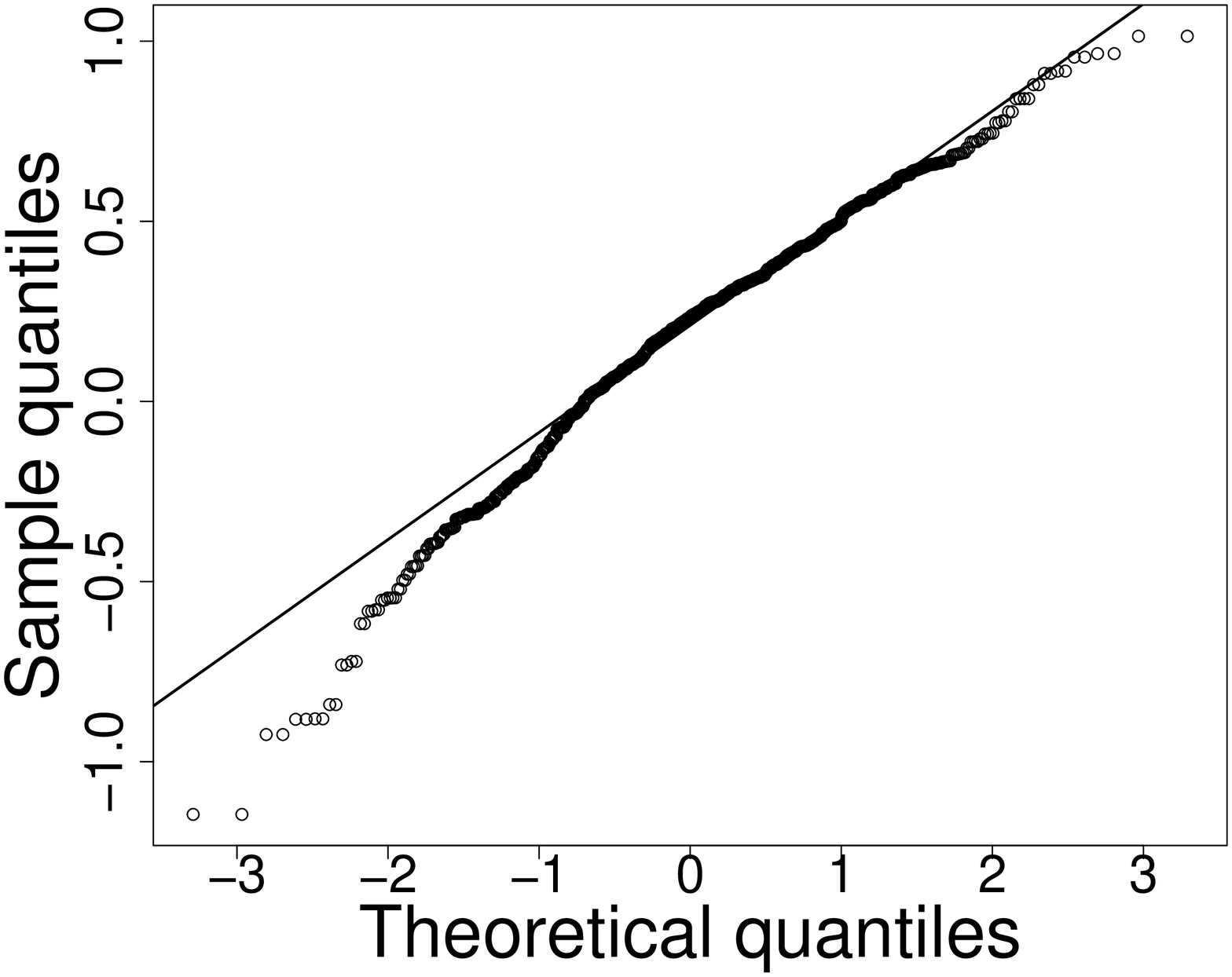} 
			\caption*{  example (1), $T=500$ }
		\end{subfigure}\hfill
		\begin{subfigure}[h]{0.32\textwidth}
			\includegraphics[width=\linewidth]{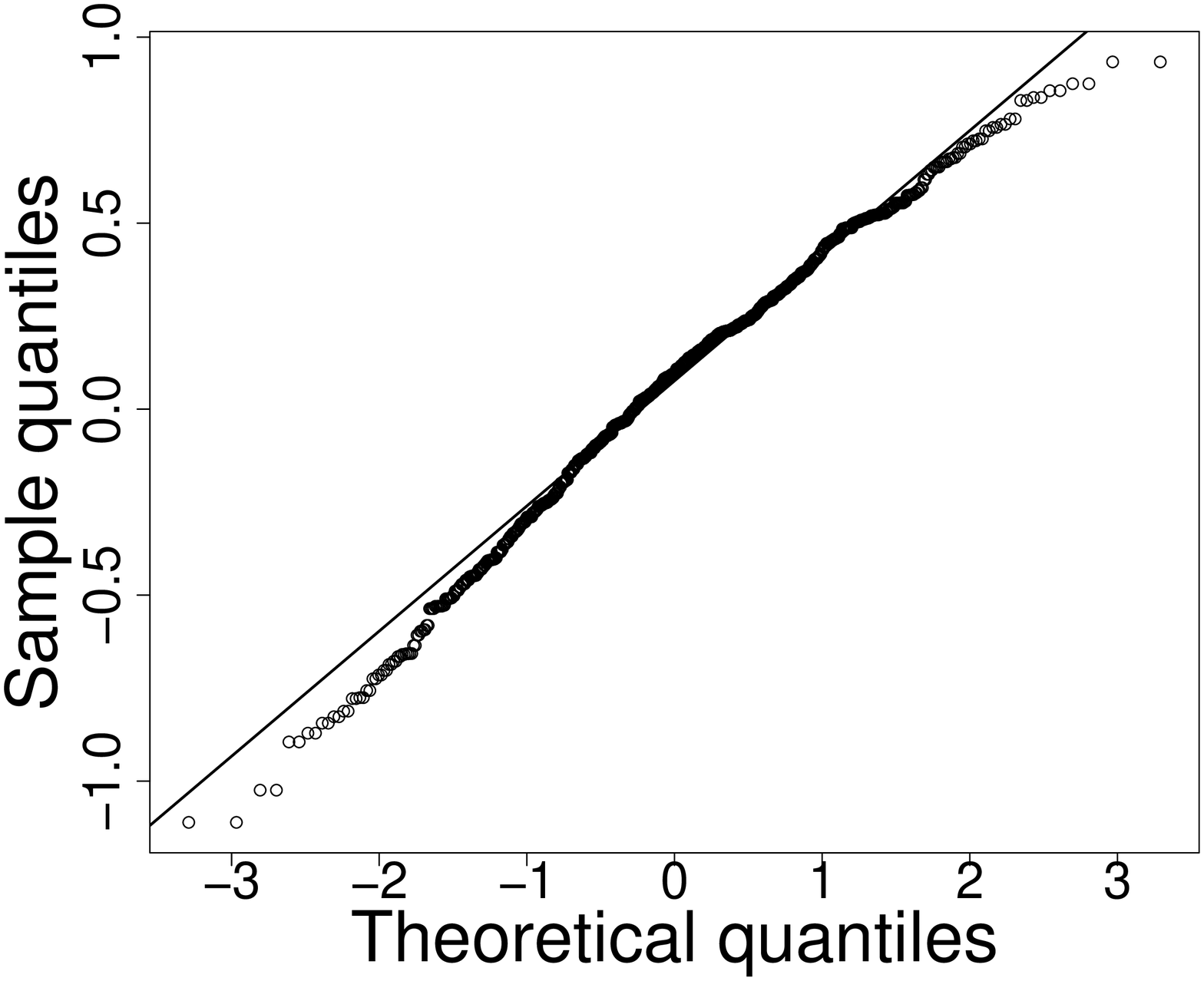} 
			\caption*{  example (2), $T=500$ }
		\end{subfigure}\hfill
		\begin{subfigure}[h]{0.32\textwidth}
			\includegraphics[width=\linewidth]{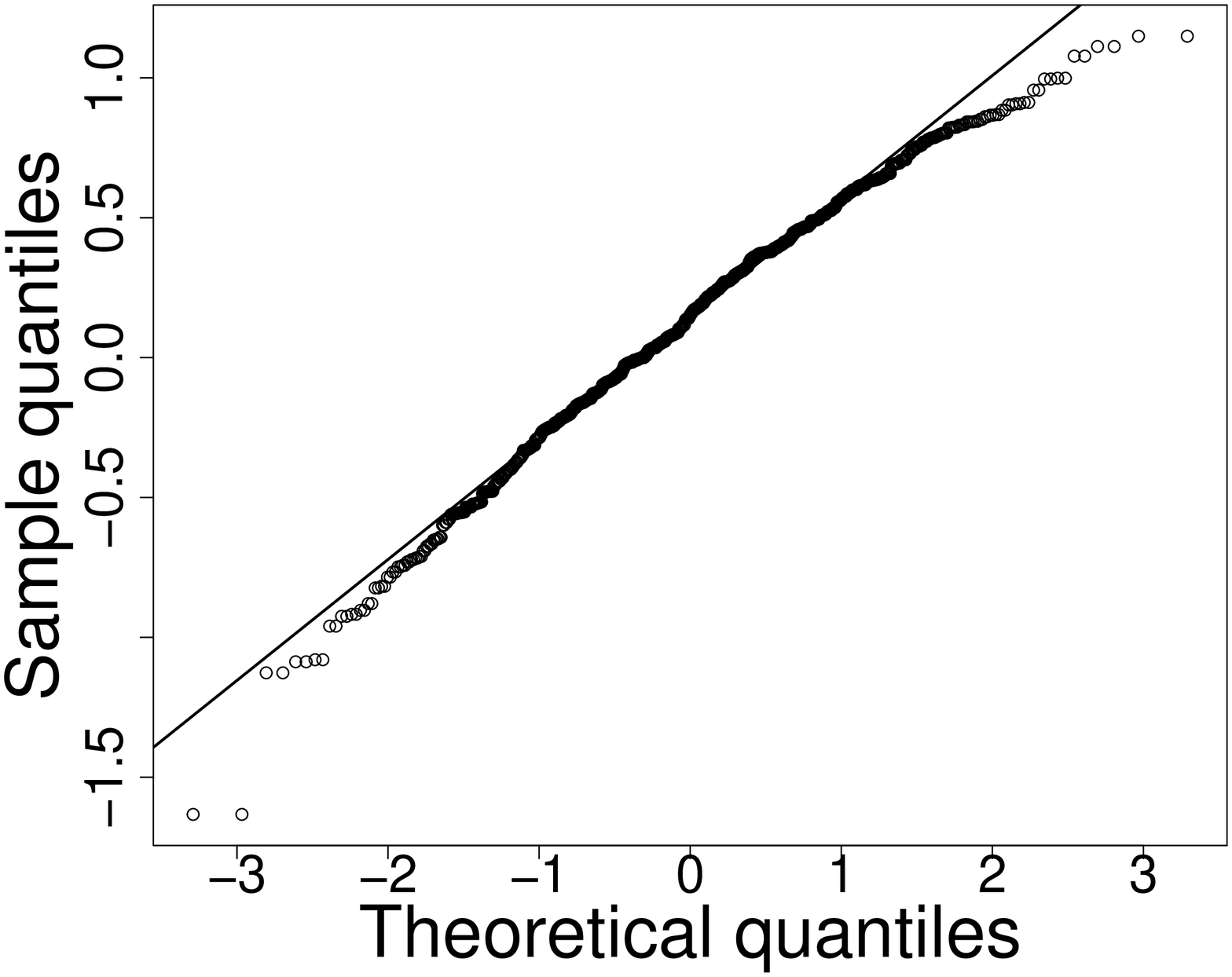} 
			\caption*{  example (3), $T=500$ }
		\end{subfigure}\hfill
		\caption{Simulation study with gamma data; QQ-plots of binned and transformed data $Y_k^*\, (k=1,...,2T-2)$ with $T=500$ for examples (1)--(3) with $p=5000,\, n=1$.  }
		\label{qqGamma}
	\end{figure}
	\begin{figure}[H]
		\begin{subfigure}[h]{0.32\textwidth}
			\includegraphics[width=\linewidth]{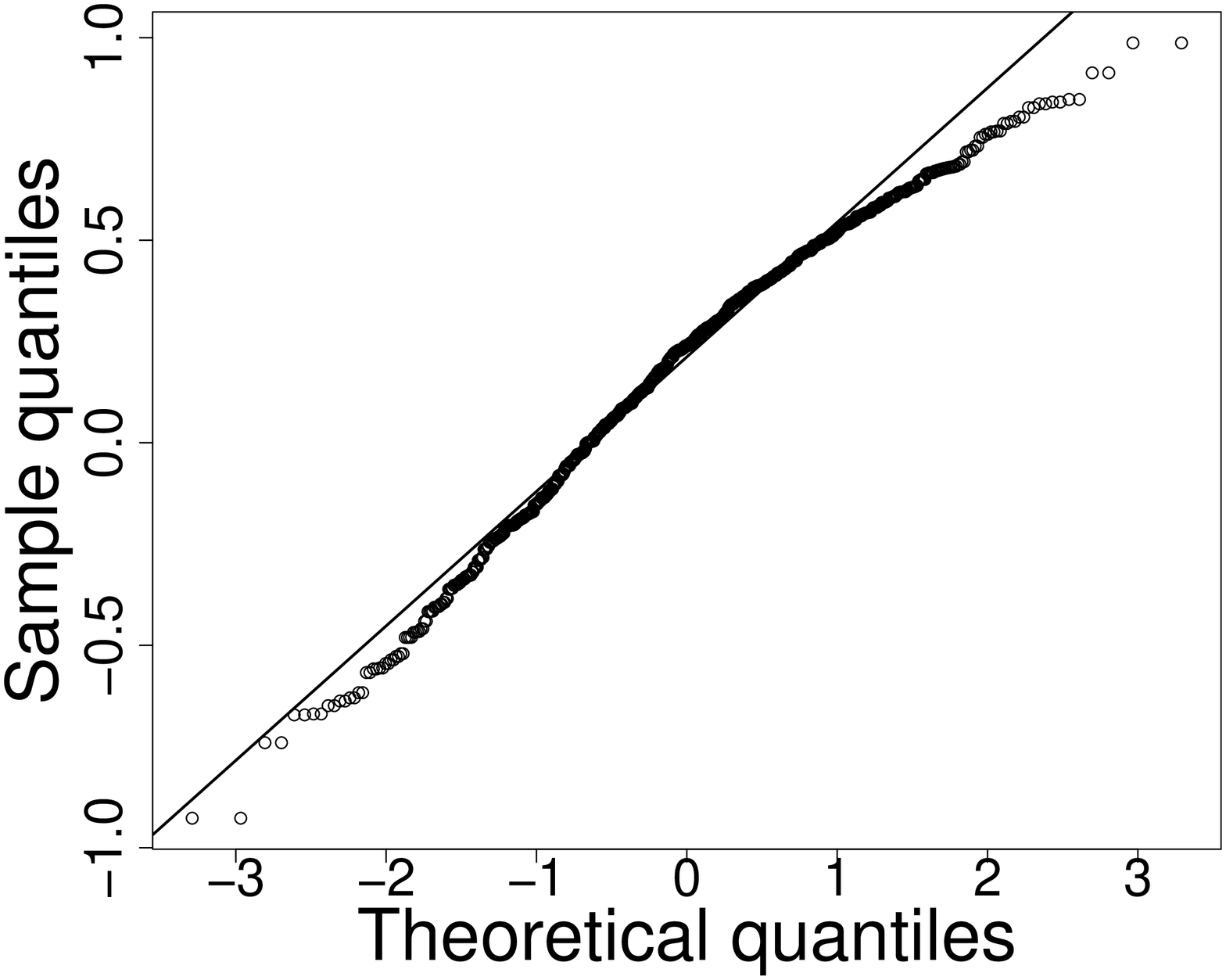} 
			\caption*{  example (1), $T=500$ }
		\end{subfigure} \hfill
		\begin{subfigure}[h]{0.32\textwidth}
			\includegraphics[width=\linewidth]{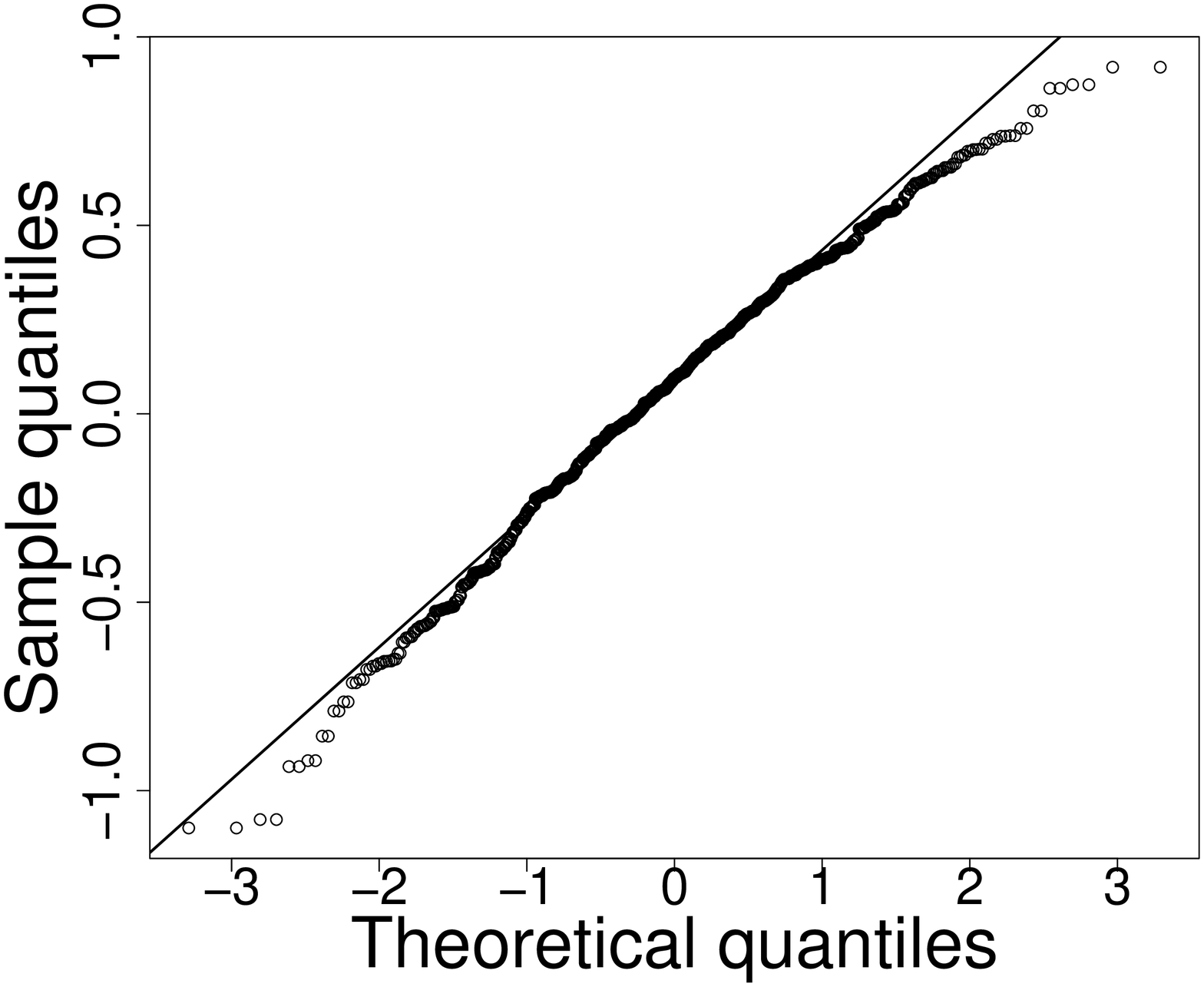} 
			\caption*{  example (2), $T=500$ }
		\end{subfigure} \hfill
		\begin{subfigure}[h]{0.32\textwidth}
			\includegraphics[width=\linewidth]{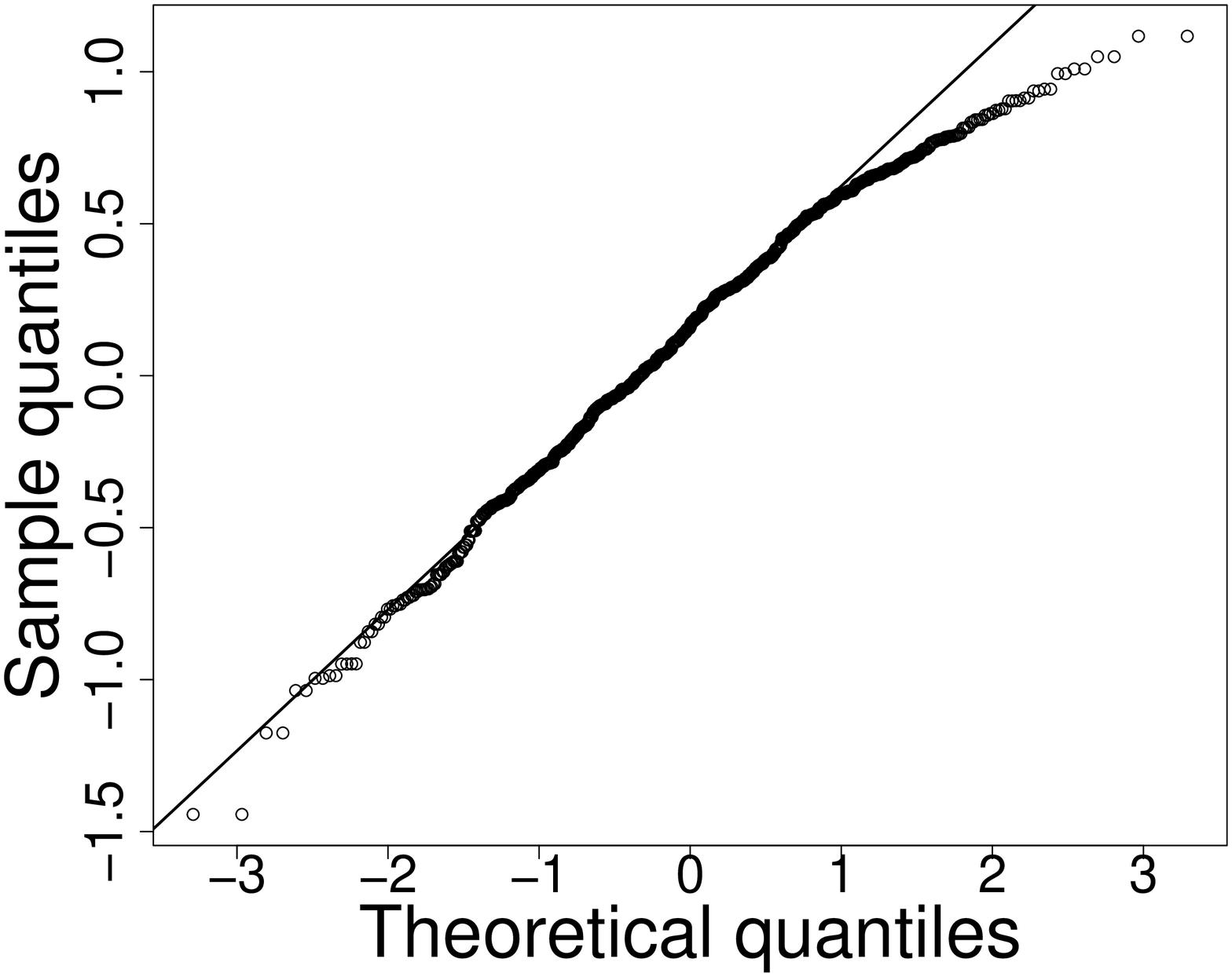} 
			\caption*{  example (3), $T=500$ }
		\end{subfigure} \hfill
		\caption{Simulation study with uniform data; QQ-plots of binned and transformed data $Y_k^*\, (k=1,...,2T-2)$ with $T=500$ for examples (1)--(3) with $p=5000,\, n=1$.  }
		\label{qqUniform}
	\end{figure}


\end{appendix}
\bibliographystyle{apalike}
\bibliography{Literature}

\end{document}